\documentclass[a4paper,onecolumn,11pt,unpublished,allowtoday]{quantumarticle}
\pdfoutput=1

\makeatletter
\usepackage[utf8]{inputenc}
\usepackage[english]{babel}
\usepackage[T1]{fontenc}
\usepackage{amsmath}
\usepackage{amsthm}
\usepackage{amssymb}
\usepackage{enumitem}
\usepackage{array}
\usepackage{xcolor}
\usepackage{tikz}
\usepackage{tikz-cd}
\usepackage[skins,breakable]{tcolorbox}
\usepackage{colortbl}
\usepackage{float}
\usepackage[mathcal]{euscript}

\definecolor{deepblue}{RGB}{47,62,85}
\definecolor{midblue}{RGB}{65,99,137}
\definecolor{urlpurple}{RGB}{98,74,131}
\definecolor{palepurple}{RGB}{236,230,246}
\colorlet{lightblue}{palepurple!45!white}
\colorlet{keyframe}{urlpurple!65!white}
\colorlet{keyfill}{palepurple!28!white}

\usepackage{hyperref}
\hypersetup{colorlinks=true, linkcolor=deepblue, citecolor=midblue, urlcolor=urlpurple,
  pdftitle={Classification of bipartite observables I: stable and uniformly stable observables},
  pdfauthor={Eduardo Due\~nez and Jos\'e Iovino}}

\RequirePackage{amssymb}
\RequirePackage{mathrsfs}

\newcommand{\CC}{\mathbb{C}}
\newcommand{\EE}{\mathbb{E}}

\newcommand{\NN}{\mathbb{N}}

\newcommand{\RR}{\mathbb{R}}

\newcommand{\cC}{\mathcal{C}}

\newcommand{\cH}{\mathcal{H}}

\newcommand{\cM}{\mathcal{M}}

\newcommand{\cO}{\mathcal{O}}

\newcommand{\cS}{\mathcal{S}}
\newcommand{\cT}{\mathcal{T}}
\newcommand{\cU}{\mathcal{U}}
\newcommand{\cV}{\mathcal{V}}
\newcommand{\cW}{\mathcal{W}}

\newcommand{\fD}{\mathfrak{D}}

\makeatletter
\RequirePackage{amsmath}
\RequirePackage{amssymb}
\RequirePackage{mathtools}

\DeclareMathOperator{\Tr}{Tr}

\newcommand{\ee}{\mathrm{e}}
\newcommand{\oA}{\otimes_{\mathsf{A}}}
\newcommand{\oB}{\otimes_{\mathsf{B}}}

\providecommand{\cH}{\mathcal{H}}

\makeatletter
\RequirePackage{amsmath}

\makeatletter
\def\di@Ind#1#2{#1\setbox0=\hbox{$#1x$}\kern\wd0\hbox to 0pt{\hss$#1\mid$\hss}%
  \smash{\lower.9\ht0\hbox to 0pt{\hss$#1\smile$\hss}}\kern\wd0}
\newcommand{\ind}{\mathrel{\mathpalette\di@Ind{}}}
\makeatother

\makeatletter
\RequirePackage[nosort]{cleveref}

\newcommand{\di@crefnames}[3]{\crefname{#1}{#2}{#3}\Crefname{#1}{#2}{#3}}
\di@crefnames{section}{Section}{Sections}
\di@crefnames{subsection}{Section}{Sections}
\di@crefnames{subsubsection}{Section}{Sections}
\di@crefnames{chapter}{Chapter}{Chapters}
\di@crefnames{part}{Part}{Parts}
\di@crefnames{appendix}{Appendix}{Appendices}
\di@crefnames{table}{Table}{Tables}
\di@crefnames{figure}{Figure}{Figures}
\di@crefnames{theorem}{Theorem}{Theorems}
\di@crefnames{keytheorem}{Theorem}{Theorems}
\di@crefnames{proposition}{Proposition}{Propositions}
\di@crefnames{corollary}{Corollary}{Corollaries}
\di@crefnames{lemma}{Lemma}{Lemmas}
\di@crefnames{fact}{Fact}{Facts}
\di@crefnames{conjecture}{Conjecture}{Conjectures}
\di@crefnames{definition}{Definition}{Definitions}
\di@crefnames{example}{Example}{Examples}
\di@crefnames{remark}{Remark}{Remarks}
\di@crefnames{question}{Question}{Questions}
\di@crefnames{axiom}{Axiom}{Axioms}
\di@crefnames{notation}{Notation}{Notations}
\di@crefnames{exercise}{Exercise}{Exercises}

\AtBeginDocument{%
}
\makeatletter
\def\@currname{di-theorem-rudin}\def\@curroptions{}
\RequirePackage{amsthm}
\RequirePackage{aliascnt}

\def\di@rudin@parent{section}
\relax

\newtheoremstyle{rudin}%
  {\topsep}{\topsep}{\itshape}{}{\bfseries}{.}{ }%
  {\thmname{#1}\thmnumber{ #2}\thmnote{ (#3)}}
\newtheoremstyle{rudindef}%
  {\topsep}{\topsep}{\upshape}{}{\bfseries}{.}{ }%
  {\thmname{#1}\thmnumber{ #2}\thmnote{ (#3)}}

\theoremstyle{rudin}
\ifx\di@rudin@parent\@empty
  \newtheorem{theorem}{Theorem}
\else
  \newtheorem{theorem}{Theorem}[\di@rudin@parent]
\fi
\newcommand{\DiSharedTheorem}[2]{%
  \newaliascnt{#1}{theorem}%
  \newtheorem{#1}[#1]{#2}%
  \aliascntresetthe{#1}%
}
\DiSharedTheorem{proposition}{Proposition}
\DiSharedTheorem{corollary}{Corollary}
\DiSharedTheorem{lemma}{Lemma}
\DiSharedTheorem{fact}{Fact}
\theoremstyle{rudindef}
\DiSharedTheorem{definition}{Definition}
\DiSharedTheorem{example}{Example}
\DiSharedTheorem{remark}{Remark}
\DiSharedTheorem{question}{Question}
\makeatletter

\theoremstyle{rudin}
\DiSharedTheorem{keytheorem}{Theorem}
\tcolorboxenvironment{keytheorem}{%
  colback=keyfill, colframe=keyframe, boxrule=0.4pt, arc=1pt,
  left=6pt, right=6pt, top=2pt, bottom=2pt,
  before skip=10pt, after skip=10pt}

\theoremstyle{rudindef}

\newtcolorbox{goalbox}[1]{breakable, colback=keyfill, colframe=keyframe, boxrule=0.4pt,
  arc=1pt, left=6pt, right=6pt, top=4pt, bottom=4pt,
  before skip=10pt, after skip=10pt,
  before upper={\textbf{#1.}\ }}

\numberwithin{figure}{section}

\crefname{enumi}{Question}{Questions}
\Crefname{enumi}{Question}{Questions}

\newcommand{\cat}[2]{#1{}^\frown #2}
\newcommand{\seg}{\sqsubseteq}
\newcommand{\ulimit}[2]{\lim_{#1,#2}}
\newcommand{\pos}{\operatorname{pos}}
\newcommand{\mist}{\operatorname{mist}}
\newcommand{\gr}{\operatorname{gr}}
\makeatother

\title{Classification of bipartite observables I: stable and uniformly stable observables}
\author{Eduardo Due\~nez}
\email{Eduardo.Duenez@utsa.edu}
\affiliation{Department of Mathematics, The University of Texas at San Antonio, One UTSA Circle, San Antonio, TX 78249, USA}
\author{Jos\'e Iovino}
\email{jose.iovino@utsa.edu}
\affiliation{Department of Mathematics, The University of Texas at San Antonio, One UTSA Circle, San Antonio, TX 78249, USA}

\begin{document}
{\renewcommand{\thefootnote}{}\footnotetext{This version: \today. The most recent version of this paper is available at \url{https://utsa-quantum-math.github.io/resources/bipartite-observables-1-stability/}; it may supersede the version deposited on arXiv.}}

\begin{abstract}
  This is the first of three papers that read the dividing lines of Shelah's classification theory as properties of bipartite quantum observables, through their expectation functions on product states.
  Here we study stability: the independence of the expectation from the order in which the preparations of the two systems are idealized.
  We show that stability is characterized by weak compactness, by factorization through a reflexive Banach space, and by the uniform simulation of idealized preparations by finite mixtures of actual ones.
  A uniform version of stability, which bounds at every margin the finite staircases of responses, governs adaptive discrimination and the online prediction of thresholded expectations;
  for the swap test, Hong--Ou--Mandel coincidences and discrete equality tests, it yields bounds independent of dimension.
\end{abstract}

\maketitle

{\small

  \par}
\medskip

\setcounter{tocdepth}{2}
\tableofcontents

%#
%## SECTION: Introduction: one property, many languages
%#
\section{Introduction: one property, many languages}\label{sec:intro}

\noindent Idealized limits, such as infinite energy, vanishing bandwidth or the thermodynamic limit, in which particle number and volume grow without bound at fixed density, are familiar in physics, and their order can matter, as it does for the thermodynamic and vanishing-field limits in spontaneous symmetry breaking.
We study when two independently prepared quantum systems can be idealized in either order without changing the expectation of an observable of the joint system, regarded as a function of the two preparations.
We show that this order independence is a form of model-theoretic stability, characterized analytically by weak compactness, geometrically by factorization through a reflexive Banach space, and operationally by the simulation of idealized preparations with finite mixtures of actual ones.
A uniform version of stability governs, in turn, how many preparations adaptive queries of expectation values can tell apart, and how many mistakes are made in predicting on which side of a threshold an expectation lies.

Let $H_1$ and $H_2$ be complex separable Hilbert spaces describing the two systems, which we call Alice’s and Bob’s.
We write $\fD(H_1)$ and $\fD(H_2)$ for their sets of density operators, the positive operators of trace one that describe their actual states, or preparations.
An independent preparation of the joint system is represented by $\rho\otimes\sigma$, with $\rho\in\fD(H_1)$ and $\sigma\in\fD(H_2)$.
For a bounded operator $C$ on the Hilbert-space tensor product $H_1\otimes H_2$, we study the expectation function
\[
  \varphi_C(\rho,\sigma)=\Tr\bigl((\rho\otimes\sigma)C\bigr).
\]
An \emph{observable} is a self-adjoint bounded operator and represents a bounded measurable quantity;
for an observable $C$, the expectation $\varphi_C$ is real-valued and lies in $[-\|C\|,\|C\|]$.
We also admit arbitrary bounded operators, whose two \emph{hermitian parts} $\tfrac12(C+C^*)$ and $\tfrac1{2i}(C-C^*)$ are observables with expectations equal to the real and the imaginary part of $\varphi_C$, respectively.
When $C$ is an orthogonal projection, it represents a yes/no test, and $\varphi_C(\rho,\sigma)$ is the probability of a yes answer.

Asymptotic regimes are described by sequences of preparations:
states of increasing energy, outputs of circuits of increasing depth, single-photon states of increasingly narrow bandwidth.
Such sequences need not converge to a density operator:
if Alice’s system is an oscillator, for instance, the sequence of its energy eigenstates has no limit among density operators.
In the weak$^*$ topology, in which states converge when all their expectations converge, the limit points of this sequence are states of the algebra
$B(H_1)$ of bounded operators on $H_1$, that is, positive linear functionals on $B(H_1)$ that take the value one at the identity, and none of them is given by a density operator.
We call a state that is not given by a density operator \emph{idealized} (see \cref{sec:setting});
idealized states play the role that plane waves and position eigenstates play in textbook quantum mechanics.

Limits of sequences of preparations can always be taken along an ultrafilter.
Recall that, for a nonprincipal ultrafilter $\cU$ on $\NN$, every bounded sequence $(a_k)_{k\in\NN}$ of complex numbers has a unique ultralimit $\ulimit{k}{\cU}a_k$, the complex number $a$ with $\{k:|a_k-a|<\varepsilon\}\in\cU$ for every $\varepsilon>0$;
it is the ordinary limit whenever the latter exists.
For a sequence $(\rho_k)_{k\in\NN}$ in $\fD(H_1)$, the state $A\mapsto\ulimit{k}{\cU}\Tr(\rho_kA)$ of $B(H_1)$ is the \emph{ultralimit state} $\ulimit{k}{\cU}\rho_k$;
for the energy eigenstates of an oscillator, it is an idealized state that assigns probability zero to every finite set of energy levels.

Idealizing both systems requires two limits, which can be taken in either order.
We call an observable $C$ \emph{stable} if for all sequences $(\rho_k)_{k\in\NN}$ in $\fD(H_1)$ and $(\sigma_l)_{l\in\NN}$ in $\fD(H_2)$ we have
\begin{equation}\label{eq:double-limit-intro}
\lim_{l\to\infty}\lim_{k\to\infty}\varphi_C(\rho_k,\sigma_l)
=\lim_{k\to\infty}\lim_{l\to\infty}\varphi_C(\rho_k,\sigma_l)
\end{equation}
whenever both iterated limits exist;
a bounded function of two variables is called stable when it satisfies the same condition.
The term comes from model theory and Banach space theory, where versions of this condition, imposed on formulas and on norms, respectively, define stability~\cite{Shelah:1971b,Krivine-Maurey:1981};
it is unrelated to dynamical stability or to stability of matter.

An instability occurs already for an observable that is diagonal in a product basis, and hence classical.
Let $(|k\rangle)_{k\in\NN}$ and $(|l\rangle)_{l\in\NN}$ be orthonormal sequences in
$H_1$ and $H_2$, for instance the energy eigenstates of two identical oscillators, and let the \emph{order projection} $P$ be the orthogonal projection onto the closed span of the vectors $|k\rangle\otimes |l\rangle$ with $k<l$.
For the density operators $\rho_k=|k\rangle\langle k|$ and $\sigma_l=|l\rangle\langle l|$, the value $\varphi_P(\rho_k,\sigma_l)$ is $1$ if $k<l$ and $0$ if $k\ge l$, so the left side of~\eqref{eq:double-limit-intro} is $0$ and the right side is $1$:
the event “Bob’s system has the higher energy” has no well-defined probability once both energies are idealized to infinity.

Fix an observable $C$.
We regard Alice’s system as the system under study and a preparation $\sigma$ of Bob’s system as a \emph{probe} of it.
For each $\rho\in\fD(H_1)$, the \emph{response} of $\rho$ is the function $r_\rho\colon\fD(H_2)\to[-\|C\|,\|C\|]$ given by $r_\rho(\sigma)=\varphi_C(\rho,\sigma)$;
it records the expectation of the joint measurement of $C$ in the product state $\rho\otimes\sigma$ as the probe varies.
The roles of the two systems are symmetric in principle, and the choice implies no order in time between the preparations: it fixes which preparation indexes the response and which is its variable.
Dually, Alice’s preparation $\rho$ leaves on Bob’s system the \emph{effective observable} $\Tr_1\bigl((\rho\otimes I)C\bigr)$, obtained by taking the partial trace over Alice’s system, whose expectation in every preparation $\sigma$ of Bob’s system is $\varphi_C(\rho,\sigma)$.
The order of idealization bears on three operational tasks:
\begin{enumerate}[label=\textup{(\roman*)},leftmargin=*]
\item
  \emph{Simulation.}
  An idealized preparation is not available in the laboratory,
  but actual preparations, those given by density operators, are available, and so are their finite mixtures, in which Alice chooses one of finitely many actual preparations at random with prescribed probabilities.
  Once an order of idealization is fixed, the task is to reproduce the responses of the idealized preparation, to any prescribed accuracy, by those of a single finite mixture, which must serve for every probe, idealized probes included.
\item
  \emph{Independent combination.}
  Given an idealized preparation of each system, one seeks an extension of $\varphi_C$ to such pairs that is weak$^*$ continuous in each preparation separately and symmetric in the two systems.
\item
  \emph{Discrimination and learning.}
  Bob learns about an unknown preparation of Alice’s system online, probe by probe:
  he chooses each probe according to the previous answers and finds out whether the response, estimated by repeating the measurement, lies above or below a fixed threshold.
  For a positive integer $m$, a protocol with $m$ probes has $2^m$ possible answer sequences, so in principle $m$ probes can tell apart exponentially many preparations.
  What matters is whether the largest number of preparations that $m$ probes can tell apart grows exponentially or only polynomially with $m$, and how many mistakes Bob makes if he predicts each answer before he receives it.
  These numbers count queries about expectations;
  estimating the expectations from repeated preparations is a separate statistical task.
\end{enumerate}

\subsection*{Main results}
Theorem~A below shows that tasks~(i) and~(ii) can be carried out exactly when the observable is stable, task~(i) with the probing system idealized first, and that stability can be read in seven mathematical languages.
Three notions from functional analysis enter its statement.
A subset of a Banach space is \emph{relatively weakly compact} if its closure in the weak topology, in which a net converges exactly when its values under every bounded linear functional converge, is compact in the same topology.
A Banach space is \emph{reflexive} if the canonical embedding identifies it with its second dual.
A \emph{mean} on a set is a finitely additive probability measure defined on all of its subsets.

\begin{goalbox}{Theorem A (stability in seven languages)}
  Let $C$ be an observable on $H_1\otimes H_2$.
  Then the following are equivalent:
  \begin{enumerate}[label=\textup{(A\arabic*)},leftmargin=*]
  \item
    \emph{Analysis: order of idealization.}
    The observable $C$ is stable.
  \item
    \emph{Topology: idealized responses.}
    The effective observables $\Tr_1((\rho\otimes I)C)$, for $\rho\in\fD(H_1)$, form a relatively weakly compact subset of the Banach space $B(H_2)$.
  \item
    \emph{Geometry: factorization.}
    There are a reflexive Banach space $E$ and bounded affine maps $u\colon\fD(H_1)\to E$ and $v\colon\fD(H_2)\to E^*$ such that $\varphi_C(\rho,\sigma)=\langle u(\rho),v(\sigma)\rangle$ for all $\rho\in\fD(H_1)$ and $\sigma\in\fD(H_2)$.
  \item
    \emph{Algebra: idealized states.}
    The two orders of idealization assign the same expectation to $C$ on every pair of states, one of each system.
  \item
    \emph{Convexity: simulation.}
    With Bob’s system idealized first, every idealized preparation of Alice’s system can be simulated to arbitrary accuracy by finite mixtures of actual preparations, uniformly over all probes, idealized ones included.
  \item
    \emph{Probability: randomization.}
    For every mean on $\fD(H_1)$ and every mean on $\fD(H_2)$, the two iterated integrals of $\varphi_C$ against them agree.
  \item
    \emph{Independence: independent combination.}
    The expectation $\varphi_C$ extends to the pairs of states, one of each system, by a function that is weak$^*$ continuous in each state separately.
  \end{enumerate}
\end{goalbox}

The theorem holds, with the same proofs, for every bounded operator $C$ on $H_1\otimes H_2$.
The equivalences are proved in \cref{thm:weakcompact} and \cref{sec:geometry} for (A1)--(A4) and~(A7), in \cref{thm:polyapprox} for~(A5), and in \cref{thm:fubini} for~(A6).
Six of the seven conditions have classical ancestors.
Condition~(A1) is the double limit condition of Krivine and Maurey~\cite{Krivine-Maurey:1981} and of Shelah~\cite{Shelah:1971b}, and (A2) is Grothendieck’s criterion for weak compactness~\cite{Grothendieck:1952}.
Condition~(A3) is the affine form of a factorization theorem of Iovino~\cite{Iovino:1999c}, which rests on that of Davis, Figiel, Johnson and Pełczyński.
Condition~(A4) is Arens regularity in the sense of Arens, Pym and Ülger~\cite{Arens:1951,Pym:1965,Ulger:1987} and, in the model-theoretic reading below, Harrington’s lemma~\cite{Harnik-Harrington:1984};
(A5) comes from Mazur’s theorem, and (A6) is the Fubini property of Pták and Simons~\cite{Ptak:1963b,Simons:1968}.
Condition~(A7) states task~(ii) directly: it is the separate continuity that makes the independent combination of two states well defined.
All seven express a single extension property.
Weak compactness places the limiting responses in the weak closure of the actual ones, and convexity turns this inclusion into uniform approximation by mixtures.
The two orders of idealization extend the expectation to pairs of states, and their agreement yields the extension in~(A7).
When the conditions hold, this extension is unique and symmetric in the two systems:
on the complex space of stable bounded operators, it is the \emph{independent product} of the two states (see \cref{thm:independent-product}).

Task~(iii) requires more.
Fix a threshold $\gamma$ and a \emph{margin} $\delta>0$, and count a response only when it lies at least $\delta/2$ away from $\gamma$.
A \emph{$\delta$-ladder} of length $N$ is a finite staircase like that of the order projection:
preparations $\rho_1,\dots,\rho_N$ of Alice’s system and probes $\sigma_1,\dots,\sigma_N$ such that the response of $\rho_i$ to $\sigma_j$ lies on one side of the band when $i\le j$ and on the other side when $i>j$.
The observable $C$ is \emph{uniformly stable} if, at every margin, its ladders have bounded length.
A \emph{$\delta$-tree} of height $n$ is the corresponding adaptive protocol:
at each step Bob probes and moves to the left or to the right according to the side of the threshold on which the response falls, and after $n$ probes he has told apart $2^n$ preparations, any two of which differ by at least $\delta$ at the probe where their branches separate (see \cref{def:trees}).
Finally, for a nonprincipal ultrafilter $\cU$ on $\NN$, we regard sequences $x=(x_k)_{k\in\NN}$ in $\fD(H_1)$ and $y=(y_k)_{k\in\NN}$ in $\fD(H_2)$ as families of experiments indexed by $k$, and we call
$\Phi^\cU_C(x,y)=\ulimit{k}{\cU}\varphi_C(x_k,y_k)$ the \emph{ultrapower expectation} of $C$.

\begin{goalbox}{Theorem B (uniform stability in five languages)}
  Let $C$ be an observable on $H_1\otimes H_2$.
  Then the following are equivalent:
  \begin{enumerate}[label=\textup{(B\arabic*)},leftmargin=*]
  \item
    \emph{Analysis: uniform order of idealization.}
    The observable $C$ is uniformly stable.
  \item
    \emph{Combinatorics: decision trees.}
    At every margin, the heights of the decision trees of $C$ are bounded.
  \item
    \emph{Ultraproducts: families of experiments.}
    For every nonprincipal ultrafilter $\cU$ on $\NN$, the ultrapower expectation $\Phi^\cU_C$ is stable.
  \item
    \emph{Growth: discrimination.}
    At every margin, the largest number of preparations that $m$ adaptive probes can tell apart is bounded by a polynomial in $m$.
  \item
    \emph{Learning: threshold prediction.}
    At every margin, the side of a fixed threshold on which the response of an unknown preparation lies can be predicted online with a number of mistakes bounded independently of the threshold.
  \end{enumerate}
\end{goalbox}

Conditions (B1)--(B3) are equivalent by \cref{thm:uniform}, which also shows that “every” may be replaced by “some” in~(B3);
(B4) is equivalent to them by \cref{thm:growth}, and (B5), for the threshold game of \cref{sec:learn}, by \cref{prop:littlestone}.
Every uniformly stable observable is stable, but not conversely (see \cref{exa:sector}).
The decision trees, which describe the adaptive protocols of task~(iii), connect uniform stability with the counts in~(B4) and~(B5):
uniform stability bounds their height at margin $\delta$ by an integer $d$, a counting argument of the kind that proves the Sauer--Shelah lemma turns this bound into the polynomial $\binom m0+\binom m1+\dots+\binom md$ of degree $d$, and the same integer $d$ bounds the number of mistakes;
if uniform stability fails, then at some margin protocols with $m$ probes tell apart all $2^m$ preparations of a tree, for every positive integer $m$.
The polynomial bound is a finite and adaptive counterpart of Rowbottom’s original definition of stability as a bound on the number of types (see \cref{rem:rowbottom}).

\subsection*{Examples}
In \cref{sec:physics} we show that stability holds for several basic observables of quantum optics and quantum information:
\begin{itemize}
\item
  \emph{Product and local measurements} $A\otimes B$
  and $A\otimes I+I\otimes B$, respectively, for bounded self-adjoint operators $A$ on $H_1$ and $B$ on $H_2$;
  the self-adjoint members of the minimal tensor product, the operator-norm closure of the finite sums of products of bounded operators;
  all compact (for instance, finite-rank) observables, such as the projections onto Bell states, the maximally entangled states of two qubits (quantum two-level systems);
  and all observables when one of the systems is finite-dimensional (see \cref{thm:stable} and \cref{cor:compact-and-findim}).
\item
  \emph{The swap test}~\cite{Buhrman-Cleve-Watrous-deWolf:2001}, for two systems with the same Hilbert space:
  its acceptance probability $\tfrac12(1+\Tr\rho\sigma)$
  determines the overlap $\Tr\rho\sigma$ of two unknown states, and its observable, the projection onto the symmetric subspace, is stable, as are the projection onto the antisymmetric subspace and the swap operator, which exchanges the two systems
  (see \cref{sec:swap}).
\item
  \emph{Hong--Ou--Mandel coincidences}~\cite{Hong-Ou-Mandel:1987}:
  a balanced ($50{:}50$) beam splitter, such as a half-silvered mirror, transmits each single incoming photon with probability $\tfrac12$ and reflects it otherwise;
  when two single photons whose internal (spectral and polarization) states are $\rho$ and $\sigma$ enter its two ports, a coincidence, a detection at both output ports, occurs with probability $\tfrac12(1-\Tr\rho\sigma)$ (see \cref{sec:swap}).
\item
  \emph{Discrete equality tests}, the projections onto the event that two quantities with pure point spectra,
  that is, with an orthonormal basis of eigenvectors, take equal values, such as “equal photon numbers” in two modes of light or “equal energies” of two particles confined in identical boxes (see \cref{prop:comparison}).
\end{itemize}
For infinite-dimensional systems, the swap operator, the projections of the swap test and of Hong--Ou--Mandel coincidences, and the equality tests between quantities with infinitely many common values are not norm limits of finite sums of products
(see \cref{prop:notmin}, \cref{sec:swap} and \cref{prop:comparison}).
They are nevertheless stable for a simple reason:
the expectation of each has the form $\varphi_C(\rho,\sigma)=\langle u(\rho),v(\sigma)\rangle$
for bounded maps $u$ from $\fD(H_1)$ and $v$ from $\fD(H_2)$ into a common Hilbert space.
We call such observables \emph{Hilbert-factorable}, and in \cref{prop:hilbertfactor} we show that every Hilbert-factorable observable is stable;
for the swap operator, the expectation $\Tr\rho\sigma$ is the Hilbert--Schmidt inner product of the two density operators.
Condition~(A3) of Theorem~A says that, with Hilbert space replaced by a reflexive Banach space, this form characterizes stability.
Hilbert-factorable observables are moreover uniformly stable, with explicit bounds:
if the vectors $u(\rho)$ and $v(\sigma)$ have norm at most one, then at every margin $\delta>0$ and in every dimension, the polynomial in~(B4) has degree at most $16/\delta^2$, and the mistake bound in~(B5) is at most $16/\delta^2$ (see \cref{thm:perceptron}).

By contrast, the comparison $[A<B]$, the projection onto the event that Bob’s quantity
$B$ takes a larger value than Alice’s quantity $A$, is unstable whenever $A$ and $B$ are self-adjoint operators on $H_1$ and $H_2$ with pure point spectra that are unbounded above, as we prove in \cref{prop:comparison};
examples are “Bob’s system has more photons” and “Bob’s system has the higher energy”, the event of the order projection when the vectors $|k\rangle$ and $|l\rangle$ are the energy eigenstates of two identical oscillators.
This instability, like that of the order projection, has nothing to do with noncommutativity:
both projections are diagonal in a product basis.

\begin{goalbox}{Consequences for preparations and probes}
  \begin{enumerate}[label=\textup{(\arabic*)},leftmargin=*]
  \item
    \textbf{Finite simulation with bounds independent of dimension.}
    For the swap test, Hong--Ou--Mandel coincidences and equality tests, every ultralimit of a sequence of actual preparations is reproduced within
    $\varepsilon$, uniformly over all probes, by an equal-weight mixture of $\lceil4/\varepsilon^2\rceil$ terms of the sequence, for every $\varepsilon>0$ and in every dimension (see \cref{thm:montecarlo,cor:montecarlo}).
    The order of idealization in~(A5) cannot be reversed:
    with Alice’s system idealized first, simulation by arbitrary actual preparations can succeed even for an unstable observable, and stability is characterized instead by the simulation of each ultralimit of a sequence of actual preparations by mixtures from every tail of that same sequence (see \cref{thm:polyapprox,exa:order-matters}).
  \item
    \textbf{Randomization and the detection of instability.}
    Pure preparations and pure probes, given by rank-one orthogonal projections, suffice to detect instability (see \cref{prop:pure}).
    Countably additive Borel randomization of the preparations of either system, for the trace-norm topology, nevertheless always makes the two iterated averages agree, so that the means in~(A6) must be genuinely finitely additive (see \cref{prop:sigmaadditive}).
  \item
    \textbf{Uniform bounds and their boundary.}
    For the swap test and Hong--Ou--Mandel coincidences, decision trees at a margin $\delta>0$ have height at most $4/\delta^2$, and for equality tests at most $16/\delta^2$, by Novikoff’s perceptron bound (see \cref{sec:trees}).
    Stability alone gives no such bound:
    a comparison of levels within spin multiplets, the irreducible representations of the group $\mathrm{SU}(2)$ through which rotations act in quantum mechanics, is stable but not uniformly stable (see \cref{exa:sector}).
  \item
    \textbf{Algebraic and topological closure.}
    Stable bounded operators are closed under complex linear combinations, adjoints and norm limits, and include the minimal tensor product and all operators whose expectations factor through a Hilbert space.
    Stable yes/no tests are closed under negation and under combining mutually exclusive tests.
    Stability is not preserved by the quantum-logical “and”, the projection onto the intersection of two ranges, even for two uniformly stable projections, nor by limits in the strong operator topology, that is, by pointwise limits on vectors
    (see \cref{sec:lattice}).
  \end{enumerate}
\end{goalbox}

The simulation in~(A5) is approximation of degree one: the approximating responses are convex combinations of actual responses, which are linear polynomials in those responses.
This is possible because the expectation is affine in each preparation and the admissible preparations form a convex set:
for a positive integer $m$, preparations $\rho_1,\dots,\rho_m\in\fD(H_1)$ and probabilities $p_1,\dots,p_m\ge0$ with $p_1+\dots+p_m=1$, we have, for every $\sigma\in\fD(H_2)$,
\[
  \sum_{j=1}^m p_j\varphi_C(\rho_j,\sigma)
  =\varphi_C\Bigl(\sum_{j=1}^m p_j\rho_j,\sigma\Bigr).
\]
This degree-one form of approximation connects the present setting with our work on deep computations,
idealized limits of computations of unbounded depth, where approximation by polynomials in finitely many realized values is one expression of tame behavior~\cite{Alva-Duenez-Iovino-Walton:2024,Duenez-Iovino-MatosWiederhold-Salvetti-Tall:2026}.
For Hilbert-factorable observables, Maurey’s empirical method bounds the number of terms of a simulating mixture, and Novikoff’s perceptron bound the heights of decision trees;
the two numbers measure different features of the expectation, the second belonging to the stronger classification by uniform stability.

\subsection*{The model-theoretic reading}
In real-valued model theory~\cite{Keisler:2023}, the expectation
$\varphi_C$ of an observable $C$ is the interpretation of a formula $\varphi(x,y)$, a real-valued predicate whose variables $x$ and $y$ range over the preparations of Alice’s and of Bob’s system, respectively.
Each of the physical notions above has a standard model-theoretic counterpart, which we describe in \cref{sec:independence} and in the historical remarks at the end of the sections;
the local stability theory on which we draw was developed by Ben Yaacov and Usvyatsov~\cite{BenYaacov-Usvyatsov:2010} in the continuous logic of Ben Yaacov, Berenstein, Henson and Usvyatsov~\cite{BenYaacov-Berenstein-Henson-Usvyatsov:2008}.
We take as the model $\cM_C$ the structure formed by the preparations of the two systems with the predicate $\varphi$.
An idealized preparation of Alice’s system determines a $\varphi$-type over the model, that is, an assignment of a value to $\varphi(x,\sigma)$ for every preparation $\sigma$ of Bob’s system.
This type is realized in an ultrapower of the model, in the sense that its values are the values $\varphi(a,\sigma)$ at some element $a$ of the ultrapower;
for an actual preparation the type is realized in the model itself, and it is given by the effective observable that the preparation leaves on Bob’s system.
Relative weak compactness of the set of effective observables corresponds to two properties of types together:
their definability, that is, the fact that the values of each type,
as a function of the parameters, are given up to uniform approximation by formulas with parameters from the model,
and the fact that every extension of a type to idealized parameters that is approximately finitely satisfiable in the model still takes the values given by the definition of the type.
Here an extension of a type is \emph{approximately finitely satisfiable} in the model if, for any finitely many of its parameters and every
$\varepsilon>0$, some element of the model has, at each of these parameters, a value of $\varphi$ within $\varepsilon$ of the value that the extension assigns.
Definability alone does not suffice, since for the order projection every type determined by an idealized preparation of Alice’s system is definable.

The other notions of Theorems~A and~B have counterparts of the same kind.
The two orders of idealization correspond to the two orders of realizing coheirs,
the extensions of types to larger sets that are approximately finitely satisfiable in the model, and the agreement of the two orders corresponds to Harrington’s lemma, by which the definitions of two types commute.
Means on the sets of preparations correspond to Keisler measures, the finitely additive probability measures on definable sets;
ladders and decision trees correspond to Shelah’s order property and to his binary trees of formulas, respectively;
and families of experiments correspond to elements of ultrapowers.
Under this correspondence, the double limit condition~\eqref{eq:double-limit-intro} becomes the stability of the formula~$\varphi$ in the model, the absence there of infinite ladders with a positive gap, which is strictly weaker than its stability in the theory of the model;
the latter corresponds to uniform stability (see \cref{thm:uniform} and \cref{exa:sector}).
The correspondence is one of form, and the proofs below are analytic;
it guides the arrangement of \cref{tab:rosetta,tab:uniform-rosetta}, in which the model-theoretic reading of each row is printed in italics after the semicolon in the first column.

\begin{table}[tp]
  \centering
  \footnotesize
  \renewcommand{\arraystretch}{1.18}
  \setlength{\arrayrulewidth}{0.4pt}
  \begin{tabular}{>{\columncolor{palepurple}\raggedright\arraybackslash}p{3.8cm}>{\raggedright\arraybackslash}p{4.9cm}>{\raggedright\arraybackslash}p{4.9cm}}
    \multicolumn{3}{c}{\large\textsc{\textcolor{deepblue}{Stability}}} \\[2pt]
    \arrayrulecolor{urlpurple!55!white}\hline
    \rowcolor{palepurple}
    \textbf{Mathematical structure}\newline\emph{physical interpretation; model theory} (section) & \textbf{Equivalent condition} & \textbf{Obstruction} \\
    \hline
    \rowcolor{lightblue}\multicolumn{3}{l}{\textsc{Stability of the expectation on the given preparation spaces} (Theorem~A)} \\
    \hline
    \textbf{Double limits}\newline\emph{order of idealization; stability in the model} (\S\ref{sec:stability}) & the two iterated limits agree whenever both exist & sequences with distinct iterated limits; an infinite ladder with a positive gap \\
    \textbf{Weak compactness}\newline\emph{idealized responses; definability and coheirs} (\S\ref{sec:topology}) & the slice map is weakly compact; Bob-first responses of every Alice state are represented on all probes by an operator on Bob’s space & the slice map is not weakly compact; some Bob-first response has no such operator representation on all probes \\
    \textbf{Reflexive factorization}\newline\emph{expectation pairing; reflexive representation} (\S\ref{sec:geometry}) & $\varphi_C(\rho,\sigma)=\langle u(\rho),v(\sigma)\rangle$, with bounded affine $u$ into a reflexive space $E$ and $v$ into $E^*$ & no bounded factorization of the expectation through a reflexive space and its dual \\
    \textbf{Arens regularity}\newline\emph{idealized states; coheirs in two orders} (\S\ref{sec:algebra}) & the two extensions agree at $C$ on every pair of states, one of each system & the two extensions disagree at $C$ on some pair of states, one of each system \\
    \textbf{Convex approximation}\newline\emph{simulation; definitions by averages} (\S\ref{sec:simulation}) & every idealized response is approximated by finite mixtures, uniformly on all probes, with Bob idealized first & some Bob-first response has positive uniform distance from all actual responses \\
    \textbf{Fubini for means}\newline\emph{randomization; Keisler measures} (\S\ref{sec:probability}) & the two iterated averages agree for every pair of finitely additive probability measures & some pair gives unequal averages; countably additive randomization never detects this obstruction \\
    \textbf{Symmetric independent product}\newline\emph{independent combination; Harrington’s lemma} (\S\ref{sec:independence}) & the common values define a state on the stable operators; for each such $C$, its expectation has a unique separately weak$^*$ continuous extension & at an unstable $C$, no separately weak$^*$ continuous extension to all pairs of states, one of each system, exists \\
    \hline
  \end{tabular}
  \caption{Stability through analysis, geometry, probability and model theory: equivalent conditions on a bounded observable $C$ (middle column) and their obstructions (right column).}
  \label{tab:rosetta}
\end{table}

\begin{table}[tp]
  \centering
  \footnotesize
  \renewcommand{\arraystretch}{1.18}
  \setlength{\arrayrulewidth}{0.4pt}
  \begin{tabular}{>{\columncolor{palepurple}\raggedright\arraybackslash}p{3.8cm}>{\raggedright\arraybackslash}p{4.9cm}>{\raggedright\arraybackslash}p{4.9cm}}
    \multicolumn{3}{c}{\large\textsc{\textcolor{deepblue}{Uniform stability}}} \\[2pt]
    \arrayrulecolor{urlpurple!55!white}\hline
    \rowcolor{palepurple}
    \textbf{Mathematical structure}\newline\emph{physical interpretation; model theory} (section) & \textbf{Equivalent condition} & \textbf{Obstruction} \\
    \hline
    \textbf{Ladder bounds}\newline\emph{uniform order of idealization; order property at a margin} (\S\ref{sec:ladders-trees}) & at every positive margin, ladder lengths are bounded & at some positive margin, ladders have every finite length \\
    \textbf{Tree bounds}\newline\emph{adaptive response queries; Shelah’s trees} (\S\ref{sec:ladders-trees}) & at every positive margin, tree heights are bounded & at some positive margin, trees have every finite height \\
    \textbf{Ultrapower stability}\newline\emph{families of experiments; stability in the theory} (\S\ref{sec:ultra}) & the ultrapower expectation is stable for every nonprincipal ultrafilter on $\NN$, equivalently for some such ultrafilter & at some positive margin, every such ultrapower encodes an infinite decision tree \\
    \textbf{Polynomial growth}\newline\emph{discrimination; finite counting counterpart of $\lambda$-stability} (\S\ref{sec:growth}) & at every margin, $m$ adaptive expectation queries realize at most $\sum_{i=0}^{d}\binom mi$ transcripts, for an integer $d\ge0$ independent of $m$ and the threshold & at some margin, protocols realize all $2^m$ transcripts for every positive integer $m$ \\
    \textbf{Finite mistake bounds}\newline\emph{threshold prediction; Littlestone dimension} (\S\ref{sec:learn}) & at every margin, a finite mistake bound in the fixed-threshold response game, uniform in the threshold; for the swap test, at most $4/\delta^2$ at margin $\delta$ & at some margin, no finite mistake bound is uniform in the threshold \\
    \hline
  \end{tabular}
  \caption{Uniform stability through finite configurations, ultrapowers and learning: equivalent conditions on a bounded self-adjoint observable (middle column) and their obstructions (right column).}
  \label{tab:uniform-rosetta}
\end{table}

\subsection*{History and context}
As the ancestors of the conditions of Theorem~A show, the double limit condition~\eqref{eq:double-limit-intro}, for bounded functions on a product of two sets, has been found, often independently, in several parts of mathematics, from functional analysis and the theory of bilinear forms to Banach space theory and model theory;
Pták’s combinatorial lemma on convex means~\cite{Ptak:1963b} belongs to the same circle.
Its uniform form appears, through Shelah’s trees, as finite Littlestone dimension, a bound on the height of decision trees, in learning theory~\cite{Littlestone:1988,Chase-Freitag:2019}.
Iovino observed that, for two-valued functions, the failure of the double limit condition is Shelah’s order property,
the existence of an infinite staircase of values like that of the order projection, and, building on Krivine and Maurey, he developed stability for real-valued structures~\cite{Iovino:1994a,Iovino:1999a,Iovino:1999c}.

The failure of stability leaves open how much regularity the limiting responses retain.
In Part~II~\cite{Duenez-Iovino:2026b} we classify the topology of limiting responses under the weaker requirement that each is a pointwise limit of a sequence of continuous functions;
convexity, supplied by probabilistic mixtures, restricts the possible response spaces.
In Part~III~\cite{Duenez-Iovino:2026c} we turn to forking, the dependence of a type on additional parameters.
We study it in the algebraic label structures of quantum error-correcting stabilizer codes, whose theory is simple and unstable, and in the associated mixed response structures, where the passage to mixtures can change the classification of the complete theory.

\paragraph{Organization of the paper.}
The proofs proceed from the two orders of idealization to their analytic and combinatorial consequences.
We define the states and the two orders of idealization in \cref{sec:setting}, establish the double limit criterion in \cref{sec:stability}, and classify the example observables in \cref{sec:physics}.
For Theorem~A, weak compactness is the central step:
in \cref{sec:char} we relate it to reflexive factorization and to the agreement of the two orders of idealization,
and we then derive simulation by mixtures in \cref{sec:simulation}, the Fubini identity in \cref{sec:probability}, and the independent product in \cref{sec:independence}.
For Theorem~B, we turn in \cref{sec:trees} to the finite configurations encoded by the expectation function, which yield the ultrapower criterion and the bounds on growth and prediction.
With the two classifications established, we examine in \cref{sec:lattice} which linear, topological and lattice operations preserve stability:
linear combinations, adjoints, norm limits and orthogonal sums do, whereas meets and strong operator limits need not.
We conclude with the outlook in \cref{sec:outlook} and the questions in \cref{sec:questions}.

%#
%## SECTION: Physics: two orders of idealization
%#
\section{States and the two orders of idealization}\label{sec:setting}

\noindent The expectation on pairs of density operators determines two extensions to arbitrary states, according to which variable is extended first.
We construct these extensions through the duality between trace-class and bounded operators.
We then compare them:
stability will be the condition under which they agree.

%#
%### SUBSECTION: Actual states and effective observables
%#
\subsection{Actual states and effective observables}
As in the Introduction, the complex separable Hilbert spaces $H_1$ and $H_2$ describe Alice’s and Bob’s systems;
throughout, the letter $H$ denotes a complex separable Hilbert space, and we write $\fD(H)$ for the set of density operators on $H$, the positive operators of trace one.
We use inner products that are linear in the second slot.
For vectors $\psi$ and $\chi$ of $H$ and a bounded operator $A$ on $H$, we write $\langle\chi|A|\psi\rangle=\langle\chi,A\psi\rangle$ and $|\psi\rangle\langle\chi|$ for the rank-one operator $\xi\mapsto\psi\langle\chi,\xi\rangle$.
For a Banach space $E$, the dual pairing $\langle x,\ell\rangle=\ell(x)$ for $x\in E$ and $\ell\in E^*$ remains bilinear.
The space $\cT(H)$ of trace-class operators on $H$ consists of the operators $\tau$ whose trace norm $\|\tau\|_1=\Tr|\tau|$, where $|\tau|=(\tau^*\tau)^{1/2}$, is finite.
With this norm, the space $\cT(H)$ has as its dual the algebra $B(H)$ of bounded operators on $H$, each operator $A\in B(H)$ acting by $\tau\mapsto\Tr(\tau A)$.
For a unit vector $\psi$, the pure state $|\psi\rangle\langle\psi|$, the orthogonal projection onto the line through $\psi$, is often written simply $\psi$.

\begin{definition}[Bilinear form and slice maps]\label{def:slice-and-bilinear}
  Let $C\in B(H_1\otimes H_2)$ be given.
  \begin{enumerate}[label=\textup{(\arabic*)}]
  \item
    The \emph{bilinear form} of $C$ is the map
    $m_C(\tau,\tau')=\Tr\bigl((\tau\otimes\tau')C\bigr)$ on $\cT(H_1)\times\cT(H_2)$.
  \item
    The \emph{slice maps} of $C$ are the operators
    $T_C\colon\cT(H_1)\to B(H_2)$ and $S_C\colon\cT(H_2)\to B(H_1)$ determined by
    \[
      \Tr\bigl(\tau'\,T_C(\tau)\bigr)=m_C(\tau,\tau')=\Tr\bigl(\tau\,S_C(\tau')\bigr)
      \qquad\text{for all }\tau\in\cT(H_1)\text{ and }\tau'\in\cT(H_2).
    \]
  \end{enumerate}
\end{definition}

The restriction of $m_C$ to $\fD(H_1)\times\fD(H_2)$ is the expectation $\varphi_C(\rho,\sigma)=\Tr\bigl((\rho\otimes\sigma)C\bigr)$ of the Introduction, and both slice maps have norm at most $\|C\|$.
For a self-adjoint $C$ and a density operator $\rho$ on
$H_1$, the operator $T_C(\rho)=\Tr_1\bigl((\rho\otimes I)C\bigr)$
is obtained from $(\rho\otimes I)C$ by the partial trace over Alice’s system, the linear operation that sends $\tau\otimes B$ to $\Tr(\tau)\,B$ for trace-class $\tau$ on $H_1$ and bounded $B$ on $H_2$:
for every preparation of Bob’s system, $T_C(\rho)$ has the same expectation as $C$ on the corresponding product state.
This identity of expectations defines the \emph{effective observable};
it does not specify a local implementation of the joint measurement.
For bounded operators $A$ on $H_1$ and $B$ on $H_2$, the product $C=A\otimes B$ yields $T_C(\rho)=\Tr(\rho A)\,B$:
every preparation of Alice’s system leaves a multiple of $B$ on Bob’s side, and the slice map $T_C$ has rank at most one.

%#
%### SUBSECTION: Idealized states
%#
\subsection{Idealized states}
A \emph{state} of $B(H)$ is a positive linear functional $\omega$ with $\omega(I)=1$.
The state space $\cS(B(H))\subseteq B(H)^*$ carries the weak$^*$ topology, the topology of pointwise convergence on $B(H)$, in which it is compact.
Density operators give the \emph{normal}, or actual, states $A\mapsto\Tr(\rho A)$;
we call the remaining states \emph{idealized}.
For instance, let $(|k\rangle)_{k\in\NN}$ be an orthonormal sequence in $H$ (the energy eigenstates of an oscillator, say), and let $\cU$ be a nonprincipal ultrafilter on $\NN$.
With the ultralimits along $\cU$ recalled in the Introduction, the state $\omega$ of $B(H)$ defined by
\[
  \omega(A)=\ulimit{k}{\cU}\langle k|A|k\rangle
\]
vanishes on every compact operator, and it is therefore idealized, since every normal state $A\mapsto\Tr(\rho A)$ takes a positive value at the projection onto an eigenvector of $\rho$ with a positive eigenvalue.
When the vectors $|k\rangle$ are the energy eigenstates of an oscillator, the state $\omega$ assigns probability zero to every finite set of energy levels and thus represents escape beyond every finite energy cutoff.
More generally, for a sequence $(\rho_k)_{k\in\NN}$ in $\fD(H)$, the \emph{ultralimit state} $\ulimit{k}{\cU}\rho_k$ is the state $A\mapsto\ulimit{k}{\cU}\Tr(\rho_kA)$.

\begin{lemma}\label{lem:density}
  Let $H$ be a complex separable Hilbert space.
  The normal states are weak$^*$ dense in the state space $\cS(B(H))$.
  Every $\omega\in B(H)^*$ is a combination $\omega=c_1\omega_1-c_2\omega_2+ic_3\omega_3-ic_4\omega_4$ of four states $\omega_1,\dots,\omega_4$ with $0\le c_j\le\|\omega\|$ for $j=1,\dots,4$.
\end{lemma}

\begin{proof}
  For the first statement, suppose that a state $\omega$ is not in the weak$^*$ closure of the convex set $\fD(H)$.
  By the Hahn--Banach theorem, there is a self-adjoint $A\in B(H)$ with $\omega(A)>\sup_{\rho\in\fD(H)}\Tr(\rho A)=\sup\operatorname{spec}A$, which is impossible for a state.
  The second statement is the Jordan decomposition of the hermitian parts of $\omega$ (see~\cite[Section~4.3]{Kadison-Ringrose:1983b}).
\end{proof}

%#
%### SUBSECTION: The two orders of idealization
%#
\subsection{The two orders of idealization}
Suppose that Alice’s preparation has a limiting state $\omega_1$ of $B(H_1)$ and Bob’s has a limiting state $\omega_2$ of $B(H_2)$.
These states may be normal or idealized.
Fix an observable $C\in B(H_1\otimes H_2)$.
There are two natural ways of assigning it an expectation at these states.
The same constructions apply to arbitrary bounded operators and will be defined on all of $B(H_1\otimes H_2)$.

\emph{Alice idealized first.}
Each actual state $\sigma$ of Bob’s system leaves on Alice’s side the effective observable $S_C(\sigma)$, to which Alice’s idealized preparation assigns the expectation $\omega_1(S_C\sigma)$.
The map $\tau'\mapsto\omega_1(S_C\tau')$ is a bounded linear functional on $\cT(H_2)$, so there is a unique operator $Z_{\omega_1}\in B(H_2)$ with
\[
  \Tr(\tau' Z_{\omega_1})=\omega_1(S_C\tau')\qquad\text{for all }\tau'\in\cT(H_2).
\]
In other words, $Z_{\omega_1}$ is the image $S_C'(\omega_1)$ of $\omega_1$ under the transpose $S_C'\colon B(H_1)^*\to B(H_2)$ of $S_C$, the map that composes each functional on $B(H_1)$ with $S_C$.
The operator $Z_{\omega_1}$ is the effective observable that Alice’s idealized preparation leaves on Bob’s side.
Bob then evaluates his state on it.

\emph{Bob idealized first.}
Symmetrically, Bob’s preparation leaves on Alice’s side the operator $W_{\omega_2}=T_C'(\omega_2)$ in $B(H_1)$.
It satisfies $\Tr(\tau W_{\omega_2})=\omega_2(T_C\tau)$ for all $\tau\in\cT(H_1)$, and Alice evaluates her state on it.

\begin{definition}[The two products]\label{def:products}
  For states $\omega_1$ of $B(H_1)$ and $\omega_2$ of $B(H_2)$, and for $C\in B(H_1\otimes H_2)$, we define
  \[
    (\omega_1\oA\omega_2)(C)=\omega_2\bigl(Z_{\omega_1}\bigr),
    \qquad
    (\omega_1\oB\omega_2)(C)=\omega_1\bigl(W_{\omega_2}\bigr),
  \]
  where $Z_{\omega_1}=S_C'(\omega_1)\in B(H_2)$ and
  $W_{\omega_2}=T_C'(\omega_2)\in B(H_1)$ are the images of $\omega_1$ and $\omega_2$ under the transposes of the slice maps of $C$.
\end{definition}

The subscript names the system that is idealized first.
When $C$ is self-adjoint, the operators $Z_{\omega_1}$ and $W_{\omega_2}$ are the effective observables that the two states leave on Bob’s and on Alice’s side, respectively, and both product values lie in $[-\|C\|,\|C\|]$.
The next lemma explains the names of the products:
each product is an iterated limit, taken in one of the two orders.
It also expresses the products through the bitransposes
$T_C''\colon\allowbreak B(H_1)^*\to B(H_2)^{**}$ and $S_C''\colon\allowbreak B(H_2)^*\to B(H_1)^{**}$ of the slice maps, that is, the transposes of their transposes $T_C'$ and $S_C'$.

\begin{lemma}\label{lem:products}
  \emergencystretch=1em
  Let $C\in B(H_1\otimes H_2)$, and let $\omega_1$ and $\omega_2$ be states of $B(H_1)$ and of $B(H_2)$.
  \begin{enumerate}[label=\textup{(\alph*)}]
  \item
    If $(\rho_i)_{i\in\mathcal{I}}$ and $(\sigma_j)_{j\in\mathcal{J}}$ are nets of normal states of $B(H_1)$ and of $B(H_2)$, indexed by directed sets $\mathcal{I}$ and $\mathcal{J}$, with weak$^*$ limits $\omega_1$ and $\omega_2$, then
    \[
      (\omega_1\oA\omega_2)(C)=\lim_j\lim_i\varphi_C(\rho_i,\sigma_j),
      \qquad
      (\omega_1\oB\omega_2)(C)=\lim_i\lim_j\varphi_C(\rho_i,\sigma_j).
    \]
  \item
    If $\omega_1=\ulimit{k}{\cU}\rho_k$ and $\omega_2=\ulimit{l}{\cV}\sigma_l$ are the ultralimit states of sequences of normal states along nonprincipal ultrafilters $\cU$ and $\cV$ on $\NN$, then
    \[
      (\omega_1\oA\omega_2)(C)=\ulimit{l}{\cV}\ulimit{k}{\cU}\varphi_C(\rho_k,\sigma_l),
      \qquad
      (\omega_1\oB\omega_2)(C)=\ulimit{k}{\cU}\ulimit{l}{\cV}\varphi_C(\rho_k,\sigma_l).
    \]
  \item
    The two products satisfy
    \[
      (\omega_1\oB\omega_2)(C)=\langle T_C''(\omega_1),\omega_2\rangle,
      \qquad
      (\omega_1\oA\omega_2)(C)=\langle S_C''(\omega_2),\omega_1\rangle .
    \]
  \item
    The map $\omega_2\mapsto(\omega_1\oA\omega_2)(C)$, defined on all states $\omega_2$ of $B(H_2)$, is weak$^*$ continuous.
  \end{enumerate}
\end{lemma}

\begin{proof}
  (a) For a normal state $\sigma$, we have
  \[
    \Tr(\sigma Z_{\omega_1})=\omega_1(S_C\sigma)
    =\lim_i\Tr(\rho_iS_C\sigma)=\lim_i\varphi_C(\rho_i,\sigma),
  \]
  and then $\omega_2(Z_{\omega_1})=\lim_j\Tr(\sigma_jZ_{\omega_1})$.
  The formula for $\oB$ is symmetric.
  (b) The same computation applies, with the limits over $i$ and $j$ replaced by the ultralimits along $\cU$ and $\cV$.
  (c) The pairing $\langle T_C''(\omega_1),\omega_2\rangle$ equals $\omega_1(T_C'\omega_2)=\omega_1(W_{\omega_2})$ by the definition of the bitranspose, and symmetrically for $S_C''$.
  (d) The map is evaluation at the operator $Z_{\omega_1}\in B(H_2)$ of \cref{def:products}.
\end{proof}

The bitranspose formulas of \cref{lem:products} make sense for all $\omega_1\in B(H_1)^*$ and $\omega_2\in B(H_2)^*$ and are bilinear in them;
so extended, the two products are the two \emph{Arens extensions} of the bilinear form $m_C$ to the biduals $B(H_1)^*$ of $\cT(H_1)$ and $B(H_2)^*$ of $\cT(H_2)$~\cite{Arens:1951}, and \emph{Arens regularity} of $m_C$ is the statement that they agree.
In \cref{cor:normal-factor} below, we show that when one of the two states $\omega_1$ and $\omega_2$ is actual, the two products always agree.

\begin{corollary}[Instability requires two idealized factors]\label{cor:normal-factor}
  If $\omega_1$ and $\omega_2$ are states of $B(H_1)$ and of $B(H_2)$, respectively, and one of them is normal, then $(\omega_1\oA\omega_2)(C)=(\omega_1\oB\omega_2)(C)$ for every $C\in B(H_1\otimes H_2)$.
\end{corollary}

\begin{proof}
  If $\omega_1$ is given by a density operator $\rho$, the operator $Z_{\omega_1}$ is $T_C(\rho)$, and the product $(\omega_1\oB\omega_2)(C)$ is $\Tr(\rho\,T_C'\omega_2)=\omega_2(T_C\rho)$;
  both products therefore equal $\omega_2(T_C\rho)$.
  If $\omega_2$ is given by a density operator $\sigma$, the operator $T_C'(\sigma)$ coincides with $S_C(\sigma)$, and both products equal $\omega_1(S_C\sigma)$.
\end{proof}

%#
%## SECTION: Analysis: stability is the double limit property
%#
\section{Analysis: the double limit property and ladders}\label{sec:stability}

\noindent Whether two idealizations can be exchanged depends only on the expectation function.
We therefore study the double limit property for bounded functions on arbitrary products: its formulation through ultrafilters, the ladders that witness its failure, and its preservation under the operations used later.
The same arguments then apply to observables, to their restrictions and to their ultrapowers.

\begin{definition}[Double limit property]\label{def:stable}
  Let $X$ and $Y$ be sets.
  A bounded function $f\colon X\times Y\to\CC$ is \emph{stable} if for all sequences $(x_k)_{k\in\NN}$ in $X$ and $(y_l)_{l\in\NN}$ in $Y$, we have
  \[
    \lim_{l\to\infty}\lim_{k\to\infty}f(x_k,y_l)=\lim_{k\to\infty}\lim_{l\to\infty}f(x_k,y_l)
  \]
  whenever both iterated limits exist.
  A bounded operator $C\in B(H_1\otimes H_2)$, where $H_1$
  and $H_2$ are the Hilbert spaces of Alice’s and Bob’s systems, is \emph{stable} if its expectation
  $\varphi_C(\rho,\sigma)=\Tr\bigl((\rho\otimes\sigma)C\bigr)$ is stable on the product $\fD(H_1)\times\fD(H_2)$ of the sets of density operators on $H_1$ and on $H_2$.
\end{definition}

Two elementary facts are used throughout.
First, every bounded array $(a_{kl})_{k,l\in\NN}$ of complex numbers, with rows indexed by
$k$ and columns by $l$, has subsequences of rows and columns along which both iterated limits exist, and for a finite family of such arrays these subsequences can be chosen common to all arrays of the family:
pass to a subsequence of columns along which every row converges (a diagonal argument), then to a subsequence of rows along which every remaining column converges, and then to further subsequences along which the two outer limits exist.
Second, passing to subsequences does not change iterated limits that already exist.

The basic tool is the following extraction lemma, which turns a gap between two iterated ultrafilter limits into a \emph{ladder}:
an array that is close to one limit on and above the diagonal and close to the other below it.
The limits are the ultralimits along nonprincipal ultrafilters on $\NN$ recalled in the Introduction.

\begin{lemma}[Ladder extraction]\label{lem:ladder}
  Let $(a_{kl})_{k,l\in\NN}$ be a bounded complex array, let $\cU$ and $\cV$ be nonprincipal ultrafilters on $\NN$, and let
  \[
    L_1=\ulimit{k}{\cU}\ulimit{l}{\cV}a_{kl},\qquad L_2=\ulimit{l}{\cV}\ulimit{k}{\cU}a_{kl}
  \]
  be the two iterated limits.
  For every $\varepsilon>0$ there are indices $k_1<k_2<\cdots$ and $l_1<l_2<\cdots$ such that $|a_{k_il_m}-L_1|\le\varepsilon$ for $i\le m$ and $|a_{k_il_m}-L_2|\le\varepsilon$ for $i>m$.
\end{lemma}

\begin{proof}
  The indices are chosen alternately, each in a set that belongs to the relevant ultrafilter.
  We write $\alpha_k$ for the row limit $\ulimit{l}{\cV}a_{kl}$ and $\beta_l$ for the column limit $\ulimit{k}{\cU}a_{kl}$.
  Given $k_1,l_1,\dots,k_{m-1},l_{m-1}$, choose $k_m>k_{m-1}$ in the set
  \[
    \{k:\ |\alpha_k-L_1|\le\varepsilon/2\ \text{and}\
    |a_{kl_i}-\beta_{l_i}|\le\varepsilon/2\ \text{for all }i<m\},
  \]
  which belongs to $\cU$, and then $l_m>l_{m-1}$ in the set
  \[
    \{l:\ |\beta_l-L_2|\le\varepsilon/2\ \text{and}\
    |a_{k_il}-\alpha_{k_i}|\le\varepsilon/2\ \text{for all }i\le m\},
  \]
  which belongs to $\cV$.
  For $i\le m$, we have $|a_{k_il_m}-L_1|\le|a_{k_il_m}-\alpha_{k_i}|+|\alpha_{k_i}-L_1|\le\varepsilon$, and for $i>m$, the choice of $k_i$ after $l_m$ ensures that $|a_{k_il_m}-L_2|\le|a_{k_il_m}-\beta_{l_m}|+|\beta_{l_m}-L_2|\le\varepsilon$.
\end{proof}

\begin{corollary}\label{cor:ultra-stable}
  Let $X$ and $Y$ be sets, and let $f\colon X\times Y\to\CC$ be a bounded function.
  Then the function $f$ is stable if and only if, for all sequences $(x_k)_{k\in\NN}$ in $X$ and $(y_l)_{l\in\NN}$ in $Y$ and all nonprincipal ultrafilters $\cU$ and $\cV$ on $\NN$, we have
  \[
    \ulimit{k}{\cU}\ulimit{l}{\cV}f(x_k,y_l)=\ulimit{l}{\cV}\ulimit{k}{\cU}f(x_k,y_l).
  \]
\end{corollary}

\begin{proof}
  If ordinary iterated limits exist, they equal the ultrafilter ones;
  this gives one direction.
  For the other, suppose that, for some sequences $(x_k)_{k\in\NN}$ in $X$ and $(y_l)_{l\in\NN}$ in $Y$ and some nonprincipal ultrafilters $\cU$ and $\cV$ on $\NN$, the two ultrafilter iterated limits $L_1$ and $L_2$ differ.
  Apply \cref{lem:ladder} to $a_{kl}=f(x_k,y_l)$ with $2\varepsilon<|L_1-L_2|$, and pass to further subsequences $(k_{i_p})_{p\in\NN}$ and $(l_{m_q})_{q\in\NN}$ of the indices it provides, along which both ordinary iterated limits exist.
  For fixed $p$, we have $i_p\le m_q$ for all large $q$, so the limit over $q$ is within $\varepsilon$ of $L_1$, and hence so is the iterated limit $\lim_p\lim_q$.
  Symmetrically, the iterated limit $\lim_q\lim_p$ is within $\varepsilon$ of $L_2$.
  The two iterated limits differ, so the function $f$ is unstable.
\end{proof}

Once \cref{lem:ladder} has been applied, the proof uses only the ladder inequalities, and it therefore yields a criterion that we use repeatedly:
if $f$ is a bounded function on $X\times Y$ and there are sequences
$(x_k)_{k\in\NN}$ in $X$ and $(y_l)_{l\in\NN}$ in
$Y$, indices $k_1<k_2<\cdots$ and $l_1<l_2<\cdots$, complex numbers
$L_1$ and $L_2$, and a real number $\varepsilon$ with
$2\varepsilon<|L_1-L_2|$ such that the values $a_{k_il_m}=f(x_{k_i},y_{l_m})$ satisfy $|a_{k_il_m}-L_1|\le\varepsilon$ for $i\le m$ and $|a_{k_il_m}-L_2|\le\varepsilon$ for $i>m$, then the function $f$ is unstable.

\begin{lemma}[Preservation of stability]\label{lem:preservation}
  \emergencystretch=3em
  Let $X$ and $Y$ be sets.
  \begin{enumerate}[label=\textup{(\alph*)}]
  \item
    \textup{(Sums with one-variable coefficients.)} For every positive integer
    $r$, if $f_1,\dots,f_r$ are stable functions on $X\times Y$ and the functions $c_1,\dots,c_r\colon\allowbreak X\to\CC$ and $d_1,\dots,d_r\colon\allowbreak Y\to\CC$ are bounded, then the function $(x,y)\mapsto\sum_{j=1}^rc_j(x)d_j(y)f_j(x,y)$ is stable.
    In particular, linear combinations of stable functions and separated sums $\sum_{j=1}^rc_j(x)d_j(y)$ are stable.
  \item
    \textup{(Uniform limits.)} If $f$ and $f'$ are bounded functions on
    $X\times Y$, the function $f'$ is stable, and a real number
    $\varepsilon\ge0$ satisfies $\sup_{X\times Y}|f-f'|\le\varepsilon$, then, for all sequences $(x_k)_{k\in\NN}$ in $X$ and $(y_l)_{l\in\NN}$ in $Y$ along which both iterated limits of $f$ exist, these two limits differ by at most $2\varepsilon$.
    In particular, uniform limits of stable functions are stable.
  \item
    \textup{(Composition.)} If $f$ is stable on $X\times Y$, then for all sets $X'$ and $Y'$ and all maps $g\colon X'\to X$ and $h\colon Y'\to Y$, the function $(x',y')\mapsto f(g(x'),h(y'))$ is stable on $X'\times Y'$.
  \item
    \textup{(Reduction to states.)} For $C\in B(H_1\otimes H_2)$, the expectation
    $\varphi_C$ is stable on $\fD(H_1)\times\fD(H_2)$ if and only if the bilinear form $m_C(\tau,\tau')=\Tr\bigl((\tau\otimes\tau')C\bigr)$ of \cref{def:slice-and-bilinear} is stable on the product of the unit balls of $\cT(H_1)$ and $\cT(H_2)$.
  \end{enumerate}
\end{lemma}

\begin{proof}
  (a) Let $F$ denote the function $\sum_jc_jd_jf_j$, and let $(x_k)_{k\in\NN}$ in $X$ and $(y_l)_{l\in\NN}$ in $Y$ be sequences along which both iterated limits of $F$ exist.
  Pass to subsequences along which, for every $j$, both iterated limits of $f_j$ exist and the numbers $c_j(x_k)$ and $d_j(y_l)$ converge to limits $c_j^\infty$ and $d_j^\infty$, respectively.
  By stability, the two iterated limits of $f_j$ have a common value $\Lambda_j$, and both iterated limits of $F$ along the subsequences equal $\sum_jc_j^\infty d_j^\infty\Lambda_j$.

  (b) Let $L_1$ and $L_2$ be the two iterated limits of $f$ along sequences $(x_k)_{k\in\NN}$ in $X$ and $(y_l)_{l\in\NN}$ in $Y$ along which both exist.
  Pass to subsequences along which both iterated limits of $f'$ exist.
  They have a common value $\Lambda$, and both $|L_1-\Lambda|$ and $|L_2-\Lambda|$ are at most $\varepsilon$.

  (c) Sequences in $X'$ and $Y'$ are mapped to sequences in $X$ and $Y$.

  (d) One direction is trivial.
  For the other, decompose each $\tau$ in the unit ball of
  $\cT(H)$, where $H$ is $H_1$ or $H_2$, as $\tau=c_1\rho_1-c_2\rho_2+ic_3\rho_3-ic_4\rho_4$, using the positive and negative parts of the hermitian parts of $\tau$, with $\rho_j\in\fD(H)$ and $0\le c_j\le1$ (choosing a fixed state when a part vanishes).
  These choices define maps $\tau\mapsto\rho_j(\tau)$ and bounded coefficients $\tau\mapsto c_j(\tau)$.
  For $\tau$ in the unit ball of $\cT(H_1)$ and $\tau'$
  in that of $\cT(H_2)$, with $c'_{j'}$ and $\rho'_{j'}$
  the coefficients and states obtained in the same way from
  $\tau'$, the value $m_C(\tau,\tau')$ is then a sum of sixteen terms $c_j(\tau)c'_{j'}(\tau')\varphi_C(\rho_j(\tau),\rho'_{j'}(\tau'))$, each multiplied by one of the factors $\pm1$ or $\pm i$, so parts (a) and (c) apply.
\end{proof}

%#
%### SUBSECTION: Historical remarks
%#
\subsection{Historical remarks}\label{sec:hist-stability}
The double limit condition for a function on a product goes back to Grothendieck~\cite{Grothendieck:1952}, as a criterion for weak compactness in spaces of continuous functions.
Independently, Shelah introduced the \emph{order property} of a formula and defined stable theories as those in which no formula has it~\cite{Shelah:1971b};
a formula has the order property exactly when it has an infinite ladder in some model of the theory or, by compactness, ladders of every finite length.
Krivine and Maurey introduced stable Banach spaces, those in which the function $(x,y)\mapsto\|x+y\|$ satisfies the double limit condition on bounded sets~\cite{Krivine-Maurey:1981}.
Iovino’s stability theory for real-valued structures takes Krivine--Maurey stability as its point of departure, and in it Shelah’s order property reappears as the two-valued case of the double limit condition~\cite{Iovino:1994a,Iovino:1999a};
in~\cite{Iovino:1999c} he characterized the stable functions on a product of two sets as those of the form
$(x,y)\mapsto\langle u(x),v(y)\rangle$, for bounded maps $u$ into a reflexive Banach space and $v$ into its dual.
Ben Yaacov and Usvyatsov~\cite{BenYaacov-Usvyatsov:2010} developed local stability in the continuous logic of Ben Yaacov, Berenstein, Henson and Usvyatsov~\cite{BenYaacov-Berenstein-Henson-Usvyatsov:2008},
and Ben Yaacov derived definability of types directly from Grothendieck’s theorem~\cite{BenYaacov:2014a}.
The extraction of ladders in \cref{lem:ladder} follows an argument of Iovino~\cite{Iovino:1999c} (see also~\cite[Chapter~1]{Iovino:0000e}).

%#
%## SECTION: Experiments: the swap test, interference and equality tests are stable
%#
\section{Examples: the swap test, interference and equality tests are stable}\label{sec:physics}

\noindent We now compare two sources of stability:
approximation in operator norm by finite sums of products, and factorization of the expectation through a Hilbert space.
The swap test in infinite dimension, and equality tests for pure point spectra with infinitely many common eigenvalues, are stable although they lie outside the minimal tensor product.
Order comparisons, by contrast, obstruct stability even though they are diagonal in a product basis;
the distinction between stable and unstable observables is therefore present in classical probability as well as in quantum measurement.

%#
%### SUBSECTION: The minimal tensor product
%#
\subsection{The minimal tensor product}
We write $B(H_1)\odot B(H_2)$ for the space of finite sums of products $A\otimes B$ of bounded operators $A$ on $H_1$ and $B$ on $H_2$, the Hilbert spaces of Alice’s and Bob’s systems.
Its norm closure, the minimal C$^*$-tensor product $B(H_1)\otimes_{\min}B(H_2)$, consists of the uniform limits of finite sums of products of a bounded operator on Alice’s system with one on Bob’s.

The distance to the minimal tensor product bounds the failure of stability.
In the next theorem, the distance of an operator to $B(H_1)\odot B(H_2)$ is taken in the operator norm.

\begin{theorem}\label{thm:stable}
  Let $C\in B(H_1\otimes H_2)$.
  \begin{enumerate}[label=\textup{(\alph*)}]
  \item
    For all sequences $(\rho_k)_{k\in\NN}$ in $\fD(H_1)$
    and $(\sigma_l)_{l\in\NN}$ in $\fD(H_2)$ along which both iterated limits of
    $\varphi_C(\rho_k,\sigma_l)$ exist, these two limits differ by at most $2\operatorname{dist}\bigl(C,\allowbreak B(H_1)\odot B(H_2)\bigr)$.
  \item
    If $C\in B(H_1)\otimes_{\min}B(H_2)$, then $C$ is stable.
  \end{enumerate}
\end{theorem}

\begin{proof}
  Finite sums of products are stable, and stability survives uniform approximation.
  Indeed, for a finite sum $C'=\sum_{j=1}^rA_j\otimes B_j$
  of products, with $A_j\in B(H_1)$ and $B_j\in B(H_2)$, the expectation $\varphi_{C'}(\rho,\sigma)=\sum_j\Tr(\rho A_j)\Tr(\sigma B_j)$ is a separated sum, stable by \cref{lem:preservation}(a), and $|\varphi_C-\varphi_{C'}|\le\|C-C'\|$;
  part~(b) of the same lemma, applied with $\varepsilon=\|C-C'\|$, bounds the difference of the two iterated limits by $2\|C-C'\|$, and the infimum over all such $C'$ gives~(a).
  Part~(b) follows, since the distance vanishes on the norm closure.
\end{proof}

\begin{corollary}\label{cor:compact-and-findim}
  \leavevmode
  \begin{enumerate}[label=\textup{(\alph*)}]
  \item
    Every compact operator on $H_1\otimes H_2$ is stable.
  \item
    If $H_1$ or $H_2$ is finite-dimensional, then every $C\in B(H_1\otimes H_2)$ is stable.
  \end{enumerate}
\end{corollary}

\begin{proof}
  (a) The rank-one operators $|\psi\otimes\chi\rangle\langle\psi'\otimes\chi'|=|\psi\rangle\langle\psi'|\otimes|\chi\rangle\langle\chi'|$, for
  $\psi,\psi'\in H_1$ and $\chi,\chi'\in H_2$, span a dense subspace of the compact operators, so \cref{thm:stable} applies.
  (b) If $H_1$ has an orthonormal basis $e_1,\dots,e_n$, then the operator $C$ equals $\sum_{i,j\le n}|e_i\rangle\langle e_j|\otimes C_{ij}$ for suitable $C_{ij}\in B(H_2)$, and therefore lies in $B(H_1)\odot B(H_2)$;
  the case of a finite-dimensional $H_2$ is symmetric.
\end{proof}

In particular, every finite-rank projection is stable, such as the projection onto a Bell state, a maximally entangled state of two qubits (quantum two-level systems).
The slice map $T_C$ of \cref{def:slice-and-bilinear} gives a necessary condition for membership in the minimal tensor product.

\begin{proposition}\label{prop:notmin}
  If $C\in B(H_1)\otimes_{\min}B(H_2)$, then the slice map $T_C$ is compact.
\end{proposition}

\begin{proof}
  The slice map $T_{A\otimes B}=\Tr(\,\cdot\,A)B$ has rank at most one, and $\|T_C-T_{C'}\|\le\|C-C'\|$ for every $C'\in B(H_1\otimes H_2)$, in particular for the finite sums of products that approximate $C$;
  norm limits of finite-rank operators are compact.
\end{proof}

%#
%### SUBSECTION: Factorization through a Hilbert space
%#
\subsection{Factorization through a Hilbert space}
In \cref{sec:swap} and \cref{prop:comparison} we show that several physically basic observables lie outside the minimal tensor product and are nevertheless stable,
because the expectation of each is the inner product of a vector that depends only on Alice’s preparation with a vector that depends only on Bob’s.

\begin{proposition}\label{prop:hilbertfactor}
  Let $X$ and $Y$ be sets, let $\cH$ be a Hilbert space, and let $u\colon X\to\cH$ and $v\colon Y\to\cH$ be bounded maps.
  Then the function $f(x,y)=\langle u(x),v(y)\rangle$ is stable.
  In particular, if $C\in B(H_1\otimes H_2)$ and there are bounded maps
  $u\colon\fD(H_1)\to\cH$ and $v\colon\fD(H_2)\to\cH$
  such that $\varphi_C(\rho,\sigma)=\langle u(\rho),v(\sigma)\rangle$ for all $\rho\in\fD(H_1)$ and $\sigma\in\fD(H_2)$, then $C$ is stable.
\end{proposition}

\begin{proof}
  Bounded sets in a Hilbert space are relatively weakly sequentially compact.
  Given sequences $(x_k)_{k\in\NN}$ in $X$ and $(y_l)_{l\in\NN}$
  in $Y$ along which both iterated limits exist, pass to subsequences $(x_{k_i})_{i\in\NN}$ and $(y_{l_j})_{j\in\NN}$ with $u(x_{k_i})\rightharpoonup\xi$ and $v(y_{l_j})\rightharpoonup\xi'$ weakly, for some vectors $\xi$ and $\xi'$ of $\cH$.
  For each $i$, we have $\lim_j\langle u(x_{k_i}),v(y_{l_j})\rangle=\langle u(x_{k_i}),\xi'\rangle$, which tends to $\langle\xi,\xi'\rangle$;
  the other iterated limit is also $\langle\xi,\xi'\rangle$.
\end{proof}

%#
%### SUBSECTION: The swap operator, symmetrization, and Hong--Ou--Mandel interference
%#
\subsection{The swap operator, symmetrization, and Hong--Ou--Mandel interference}\label{sec:swap}
Let $K$ be a complex separable Hilbert space, which in this subsection is the Hilbert space of both systems.
The \emph{swap} $F\in B(K\otimes K)$ is the unitary operator with $F(\psi\otimes\chi)=\chi\otimes\psi$ for all $\psi,\chi\in K$.
Since $F=F^*$ and $F^2=I$, the operators
\[
  P_+=\tfrac12(I+F),\qquad P_-=\tfrac12(I-F)
\]
are the projections onto the symmetric and the antisymmetric subspaces of $K\otimes K$, the eigenspaces of $F$ for the eigenvalues $1$ and $-1$, respectively.

\begin{lemma}[Swap trick]\label{lem:swaptrick}
  For all $\tau_1,\tau_2\in\cT(K)$, we have $\Tr\bigl((\tau_1\otimes\tau_2)F\bigr)=\Tr(\tau_1\tau_2)$.
\end{lemma}

\begin{proof}
  For $\tau_1=|\psi\rangle\langle\psi'|$ and $\tau_2=|\chi\rangle\langle\chi'|$, with
  $\psi,\psi',\chi,\chi'\in K$, one computes $(\tau_1\otimes\tau_2)F=|\psi\otimes\chi\rangle\langle\chi'\otimes\psi'|$, whose trace is $\langle\chi',\psi\rangle\langle\psi',\chi\rangle=\Tr(\tau_1\tau_2)$.
  Extend by bilinearity and continuity.
\end{proof}

Let $S_2(K)$ be the Hilbert--Schmidt class of operators $A$ on $K$ with $\Tr(A^*A)<\infty$, with inner product $\langle A,B\rangle_{S_2}=\Tr(A^*B)$.
By \cref{lem:swaptrick}, for density operators $\rho$
and $\sigma$ on $K$, the expectation of the swap is
$\varphi_F(\rho,\sigma)=\Tr(\rho\sigma)=\langle\rho,\sigma\rangle_{S_2}$, and consequently
\[
  \varphi_{P_\pm}(\rho,\sigma)=\tfrac12\bigl(1\pm\Tr(\rho\sigma)\bigr).
\]
The operators $F$ and $P_\pm$ are therefore stable (see \cref{prop:hilbertfactor}).
On the other hand, the slice map $T_F(\rho)=\rho$ is the inclusion $\cT(K)\to B(K)$, which is not compact when $K$ is infinite-dimensional, since the pure states given by the vectors of an orthonormal sequence have images at mutual distance one;
the slice maps $T_{P_\pm}=\tfrac12(\Tr(\,\cdot\,)I\pm T_F)$ differ from $\pm\tfrac12T_F$ by an operator of rank one, and are therefore not compact either.
Hence, by \cref{prop:notmin}, when $K$ is infinite-dimensional, the operators $F$ and $P_\pm$ do not belong to $B(K)\otimes_{\min}B(K)$.

The overlap $\Tr(\rho\sigma)$ governs three tests:
\begin{itemize}
\item
  \emph{The swap test}~\cite{Buhrman-Cleve-Watrous-deWolf:2001}.
  An ancilla, an auxiliary qubit (a quantum two-level system), controls a swap of two registers, the quantum systems prepared in $\rho$ and $\sigma$.
  The test accepts with probability $\tfrac12(1+\Tr\rho\sigma)=\varphi_{P_+}(\rho,\sigma)$, whence the frequency of acceptance in repeated runs estimates the overlap $\Tr\rho\sigma$ of the two states.
\item
  \emph{Exchange symmetry.}
  For two independently prepared systems with the same Hilbert space, the number $\varphi_{P_\pm}(\rho,\sigma)$ is the probability that a projective test for exchange symmetry (respectively, antisymmetry) succeeds.
  For genuinely identical particles a product state
  $\rho\otimes\sigma$ is in general not admissible, since the joint state of two identical particles is supported on the symmetric subspace
  (for bosons) or on the antisymmetric subspace (for fermions), and by the formula for $\varphi_{P_\pm}$ above the state $\rho\otimes\sigma$ is so supported only for bosons and $\rho=\sigma$ pure.
  The interpretation of $\varphi_{P_\pm}(\rho,\sigma)$ as a probability therefore concerns two distinguishable particles, or the internal degrees of freedom of two identical ones, as in the next item.
\item
  \emph{Hong--Ou--Mandel interference}~\cite{Hong-Ou-Mandel:1987}.
  Two independently prepared single photons with internal states $\rho$ and $\sigma$ enter the two input ports of an ideal lossless balanced ($50{:}50$) beam splitter, whose coherent action mixes the spatial modes and preserves the internal modes.
  The internal states include the degrees of freedom relevant to distinguishability.
  With ideal output counting that does not resolve these degrees of freedom, the probability of one photon in each output port is $\tfrac12(1-\Tr\rho\sigma)=\varphi_{P_-}(\rho,\sigma)$, which vanishes for equal pure states:
  the two photons then leave by the same port (bunching).
\end{itemize}

%#
%### SUBSECTION: Equality and order of discrete spectra
%#
\subsection{Equality and order of discrete spectra}
Let $A$ and $B$ be self-adjoint operators on $H_1$ and $H_2$, possibly unbounded, representing quantities of Alice’s and of Bob’s system, respectively, with pure point spectra, that is, with orthonormal bases of eigenvectors.
Their spectral decompositions are $A=\sum_\lambda\lambda\Pi^A_\lambda$
and $B=\sum_\lambda\lambda\Pi^B_\lambda$, where $\lambda$ runs over the reals and $\Pi^A_\lambda$ and $\Pi^B_\lambda$ are the projections onto the eigenspaces of $A$ and of $B$ for $\lambda$ (zero unless $\lambda$ is an eigenvalue).
The events “equal values” and “Bob’s value is larger” are described by the projections
\[
  [A=B]=\sum_\lambda \Pi^A_\lambda\otimes \Pi^B_\lambda,\qquad
  [A<B]=\sum_{\lambda<\lambda'}\Pi^A_{\lambda}\otimes \Pi^B_{\lambda'} .
\]

\begin{proposition}[Equality is stable, order is not]\label{prop:comparison}
  Let $A$ and $B$ be self-adjoint operators on $H_1$ and $H_2$, possibly unbounded, with pure point spectra.
  \begin{enumerate}[label=\textup{(\arabic*)}]
  \item
    The projection $[A=B]$ is stable.
    If $A$ and $B$ have infinitely many common eigenvalues, then the projection $[A=B]$ does not belong to $B(H_1)\otimes_{\min}B(H_2)$.
  \item
    If $A$ or $B$ has finitely many eigenvalues, then the projection $[A<B]$ is stable.
  \item
    If the point spectra of $A$ and $B$ are both unbounded above, then the projection $[A<B]$ is unstable.
  \end{enumerate}
\end{proposition}

\begin{proof}
  (1) We have $\varphi_{[A=B]}(\rho,\sigma)=\sum_\lambda\Tr(\rho \Pi^A_\lambda)\Tr(\sigma \Pi^B_\lambda) =\langle u(\rho),v(\sigma)\rangle_{\ell^2}$, where
  $u(\rho)=(\Tr\rho\Pi^A_\lambda)_\lambda$ and $v(\sigma)=(\Tr\sigma\Pi^B_\lambda)_\lambda$ are vectors of $\ell^2$-norm at most one, and the projection is stable by \cref{prop:hilbertfactor}.
  For each common eigenvalue $\lambda$, pick a unit vector $\psi_\lambda$ in the range of $\Pi^A_\lambda$.
  The slice map $T_{[A=B]}$ sends $\psi_\lambda$ to $\Pi^B_\lambda$, and these are mutually orthogonal nonzero projections, so the map is not compact;
  the conclusion now follows from \cref{prop:notmin}.

  (2) If $A$ has finitely many eigenvalues, then the projection $[A<B]=\sum_{\lambda} \Pi^A_{\lambda}\otimes\bigl(\sum_{\lambda'>\lambda}\Pi^B_{\lambda'}\bigr)$ is a finite sum of products, stable by \cref{thm:stable}.
  If $B$ has finitely many eigenvalues, the same argument applies to $[A<B]=\sum_{\lambda'}\bigl(\sum_{\lambda<\lambda'}\Pi^A_{\lambda}\bigr)\otimes\Pi^B_{\lambda'}$.

  (3) Since the point spectra of $A$ and $B$ are both unbounded above, we can choose real numbers
  $\lambda_1<\lambda'_1<\lambda_2<\lambda'_2<\cdots$ such that, for every $k\in\NN$, the operators $A$ and $B$ have unit eigenvectors $\psi_k$ and $\chi_k$ with eigenvalues $\lambda_k$ and $\lambda'_k$, respectively.
  For $k,l\in\NN$, the value $\varphi_{[A<B]}(\psi_k,\chi_l)$ is then $1$ if $k\le l$ and $0$ if $k>l$, so the iterated limit taken first in $k$ is $0$, and the one taken first in $l$ is $1$.
\end{proof}

%#
%### SUBSECTION: The order projection
%#
\subsection{The order projection}
We return to the order projection of the Introduction, the sharpest example of instability:
although it is diagonal in a product basis, and hence classical, its two iterated limits are as far apart as the expectations of a projection allow.

\begin{theorem}[Instability of the order projection]\label{thm:unstable}
  Let $(|k\rangle)_{k\in\NN}$ and $(|l\rangle)_{l\in\NN}$ be orthonormal sequences in $H_1$ and $H_2$, let $P$ be the projection onto the closed span of $\{|k\rangle\otimes|l\rangle:k<l\}$, and let $\rho_k=|k\rangle\langle k|$ and $\sigma_l=|l\rangle\langle l|$.
  Then
  \[
    \lim_{l}\lim_{k}\varphi_P(\rho_k,\sigma_l)=0\neq1=\lim_{k}\lim_{l}\varphi_P(\rho_k,\sigma_l),
  \]
  and the distance of $P$ to $B(H_1)\otimes_{\min}B(H_2)$ is at least $\tfrac12$.
  In particular, the projection $P$ is unstable.
\end{theorem}

\begin{proof}
  For all $k,l\in\NN$, the expectation $\varphi_P(\rho_k,\sigma_l)$ is $1$ if $k<l$ and $0$ otherwise, and the values form the staircase
  \[
    \begin{pmatrix}
             & |1\rangle    & |2\rangle    & |3\rangle    & \cdots \\
      |1\rangle    & 0      & 1      & 1      & \cdots \\
      |2\rangle    & 0      & 0      & 1      & \cdots \\
      |3\rangle    & 0      & 0      & 0      & \cdots \\
      \vdots & \vdots & \vdots & \vdots & \ddots
    \end{pmatrix}.
  \]
  Each column tends to $0$ and each row to $1$.
  The bound on the distance follows from \cref{thm:stable}.
\end{proof}

The examples of this section, together with the sector comparison of \cref{exa:sector}, a stable observable that is not uniformly stable, are collected in a table at the end of \cref{sec:ultra}, after uniform stability has been introduced.

%#
%## SECTION: Topology: idealized responses are observables
%#
\section{Topology: idealized responses are observables}\label{sec:char}

\noindent We now characterize stability through the slice map, which assigns to each preparation of Alice the effective observable it leaves on Bob’s system.
We show that weak compactness of this map is equivalent both to reflexive factorization of the expectation and to agreement of its two Arens extensions.
The same condition says that every idealized response of Alice, with Bob’s system idealized first, is represented on all of Bob’s probes, normal and idealized alike, by a bounded operator on his space.
We also show that pure preparations and pure probes suffice to test stability.

%#
%### SUBSECTION: Topology: weak compactness
%#
\subsection{Topology: weak compactness}\label{sec:topology}
The proofs in this section and in \cref{sec:simulation} rest on five classical theorems about a Banach space $X$, which we recall in plain terms.
A bounded operator is \emph{weakly compact} if it maps the unit ball to a relatively weakly compact set.
\begin{itemize}
\item
  \emph{Eberlein--Šmulian theorem}~\cite{Eberlein:1947,Smulian:1940}.
  A subset of $X$ is relatively weakly compact if and only if every sequence in it has a weak cluster point in $X$.
\item
  \emph{Gantmacher’s theorem}~\cite{Gantmacher:1940}.
  An operator $T\colon X\to Y$ between Banach spaces is weakly compact if and only if its bitranspose $T''$, the transpose of its transpose, maps the bidual $X^{**}=(X^*)^*$ into $Y$, regarded as a subspace of $Y^{**}$.
\item
  \emph{Mazur’s theorem}~\cite{Mazur:1933}.
  The weak closure and the norm closure of a convex subset of $X$ coincide.
\item
  \emph{Krein’s theorem}~\cite{Krein-Smulian:1940}.
  The closed convex hull of a relatively weakly compact subset of $X$ is weakly compact.
\item
  \emph{Davis--Figiel--Johnson--Pełczyński theorem}~\cite{Davis-Figiel-Johnson-Pelcynski:1974}.
  Every weakly compact operator factors through a reflexive Banach space.
\end{itemize}

\begin{keytheorem}[Characterization of stability]\label{thm:weakcompact}
  Let $C\in B(H_1\otimes H_2)$.
  Then the following are equivalent:
  \begin{enumerate}[label=\textup{(\arabic*)}]
  \item
    The operator $C$ is stable.
  \item
    The bilinear form $m_C$ of \cref{def:slice-and-bilinear} is stable on the product of the unit balls of $\cT(H_1)$ and $\cT(H_2)$.
  \item
    \textup{(Topology.)} The slice map $T_C\colon\cT(H_1)\to B(H_2)$ of \cref{def:slice-and-bilinear} is weakly compact.
  \item
    \textup{(Geometry.)} There are a reflexive Banach space $E$ and bounded operators $a\colon\cT(H_1)\to E$ and $J\colon E\to B(H_2)$ with $T_C=J\circ a$.
  \item
    \textup{(Algebra.)} The two products of \cref{def:products}, one for each order of idealization, agree at $C$:
    for all states $\omega_1$ of $B(H_1)$ and $\omega_2$ of $B(H_2)$, we have $(\omega_1\oA\omega_2)(C)=(\omega_1\oB\omega_2)(C)$.
  \item[\textup{(5$'$)}]
    The equality in \textup{(5)} holds for all ultralimit states
    $\omega_1=\ulimit{k}{\cU}\rho_k$ and $\omega_2=\ulimit{l}{\cV}\sigma_l$ of sequences $(\rho_k)_{k\in\NN}$ and $(\sigma_l)_{l\in\NN}$ of normal states of $B(H_1)$ and of $B(H_2)$ along nonprincipal ultrafilters $\cU$ and $\cV$ on $\NN$.
  \item
    \textup{(Separate continuity.)} There is a function
    $\hat\varphi\colon\cS(B(H_1))\times\cS(B(H_2))\to\CC$, weak$^*$ continuous in each variable separately, that extends the expectation $\varphi_C$, in the sense that $\hat\varphi(\rho,\sigma)=\varphi_C(\rho,\sigma)$ for all normal states $\rho$ and $\sigma$.
  \end{enumerate}
\end{keytheorem}

For an observable $C$, the extension in~(6) takes values in $[-\|C\|,\|C\|]$;
the complex codomain allows the theorem to include arbitrary bounded operators.
The third step of the proof shows, moreover, that condition~(5) holds exactly when the bitranspose $T_C''$ maps every state $\omega_1$ of $B(H_1)$ into $B(H_2)$, namely to the effective operator $Z_{\omega_1}$ that $\omega_1$ leaves on Bob’s side (see \cref{def:products}).

We prove \cref{thm:weakcompact} in three steps, one for each language:
in this subsection we show that~(1) and~(2) are equivalent and imply~(3);
in \cref{sec:geometry} we show that~(3) and~(4) are equivalent;
and in \cref{sec:algebra} we close the circle through~(5), (5$'$) and~(6).

\begin{proof}[Proof of the first step]
  The equivalence of (1) and (2) is \cref{lem:preservation}(d).

  For the implication from (1) to (3), the idea is that the effective operator that an ultralimit of preparations leaves on Bob’s side is a weak cluster point of their effective operators, and that any failure of this produces a ladder.
  Let $(\rho_k)_{k\in\NN}$ be a sequence in $\fD(H_1)$, let $\cU$ be a nonprincipal ultrafilter on $\NN$, and let $\omega_1=\ulimit{k}{\cU}\rho_k$.
  We claim that $\ulimit{k}{\cU}\omega_2(T_C\rho_k)=\omega_2(Z_{\omega_1})$ for every state $\omega_2$ of $B(H_2)$.
  Suppose instead that, for some state $\omega_2$ of $B(H_2)$, the left side is $L_1$ and the right side is $L_2\neq L_1$, and let $\varepsilon=|L_1-L_2|/5$.
  By \cref{lem:products}, we have $\ulimit{k}{\cU}\Tr(\sigma\,T_C\rho_k)=\Tr(\sigma Z_{\omega_1})$ for every normal state $\sigma$.
  Choose recursively indices $k_1<k_2<\cdots$ and normal states $\sigma_1,\sigma_2,\dots$ as follows.
  The index $k_m$ is chosen in the set
  \[
    \bigl\{k:\ |\omega_2(T_C\rho_k)-L_1|\le\varepsilon\ \text{and}\
    |\Tr(\sigma_iT_C\rho_k)-\Tr(\sigma_iZ_{\omega_1})|\le\varepsilon\ \text{for all }i<m\bigr\},
  \]
  which belongs to $\cU$.
  The state $\sigma_m$ is a normal state with $|\Tr(\sigma_mT_C\rho_{k_i})-\omega_2(T_C\rho_{k_i})|\le\varepsilon$ for all $i\le m$ and $|\Tr(\sigma_mZ_{\omega_1})-\omega_2(Z_{\omega_1})|\le\varepsilon$;
  such a state exists by \cref{lem:density}.
  The value $\varphi_C(\rho_{k_i},\sigma_m)$ is then within $2\varepsilon$ of $L_1$ for $i\le m$ and within $2\varepsilon$ of $L_2$ for $i>m$.
  Because $4\varepsilon<|L_1-L_2|$, the paragraph after \cref{cor:ultra-stable} shows that $C$ is unstable, a contradiction.
  Since the claim extends from states to all of $B(H_2)^*$ by \cref{lem:density}, the operator $Z_{\omega_1}$ is a weak cluster point of $(T_C\rho_k)$.
  The Eberlein--Šmulian theorem now shows that the set $T_C(\fD(H_1))$ is relatively weakly compact.
  Finally, the decomposition in the proof of \cref{lem:preservation}(d)
  places the image of the unit ball under $T_C$ in a sum of four sets of the form $\{c\,A:|c|\le1,\ A\in T_C(\fD(H_1))\}$, each of which is relatively weakly compact, so the slice map $T_C$ is weakly compact.
\end{proof}

%#
%### SUBSECTION: Geometry: factorization through a reflexive space
%#
\subsection{Geometry: factorization through a reflexive space}\label{sec:geometry}
Condition~(4) of \cref{thm:weakcompact} says that the expectation factors, by bounded affine maps, through the pairing of a reflexive space with its dual:
with $u(\rho)=a(\rho)$ and $v(\sigma)=J'(\sigma)$, the diagram
\begin{equation}\label{eq:thm-reflexive}
\begin{tikzcd}[row sep=large, column sep=huge]
  \fD(H_1)\times\fD(H_2) \arrow[r, "{(u,v)}"] \arrow[dr, "\varphi_C"'] & E\times E^* \arrow[d, "{\langle\cdot,\cdot\rangle}"] \\
  & \CC
\end{tikzcd}
\end{equation}
commutes.
Conversely, let $X$ and $Y$ be arbitrary sets and $E$ a reflexive Banach space.
For bounded maps $u\colon X\to E$ and $v\colon Y\to E^*$, the function $(x,y)\mapsto\langle u(x),v(y)\rangle$ is stable.
Indeed, given sequences along which both iterated limits exist, pass to subsequences on which $u(x_k)$ and $v(y_l)$ converge weakly to $\xi\in E$ and $\eta\in E^*$.
These subsequences exist because both $E$ and $E^*$ are reflexive.
The two iterated limits then equal $\langle\xi,\eta\rangle$, as in the proof of \cref{prop:hilbertfactor}.
Applied to the preparation spaces, this argument shows that condition~(A3) of Theorem~A implies stability even without the assumption that $u$ and $v$ are affine;
together with condition~(4) of \cref{thm:weakcompact}, it proves the equivalence of~(A1) and~(A3).

\begin{proof}[Proof of the second step]
  We show that \textup{(3)} and \textup{(4)} are equivalent.
  If $T_C$ is weakly compact, it factors through a reflexive space by the Davis--Figiel--Johnson--Pełczyński theorem, and the expectation becomes
  $\varphi_C(\rho,\sigma)=\Tr(\sigma\,J(a\rho))=\langle a(\rho),J'(\sigma)\rangle$, which is the commutativity of~\eqref{eq:thm-reflexive}.
  Conversely, the image under $a$ of the unit ball is bounded, hence relatively weakly compact in the reflexive space $E$, and $J$ is weak-to-weak continuous, so the slice map $T_C=J\circ a$ is weakly compact.
\end{proof}

%#
%### SUBSECTION: Algebra: the two orders of idealization agree
%#
\subsection{Algebra: the two orders of idealization agree}\label{sec:algebra}
The algebraic form of stability is Arens regularity:
the two products of \cref{def:products} agree.
The proof of this equivalence closes the circle of implications in \cref{thm:weakcompact}, and it also shows that the agreement of the two products is equivalent to the separate continuity in~(6), which returns to topology.

\begin{proof}[Proof of the third step]
  We show that (3) and (5) are equivalent, that (5) implies (5$'$), which implies (1), and that (5) and (6) are equivalent;
  together with the first two steps, this proves the theorem.
  We begin with (3) and (5).
  For a state $\omega_1$ and a normal state $\sigma$, we have $\langle T_C''\omega_1,\sigma\rangle=\omega_1(T_C'\sigma)=\omega_1(S_C\sigma) =\Tr(\sigma Z_{\omega_1})$.
  Hence the functional $T_C''(\omega_1)$ belongs to $B(H_2)$ if and only if
  $T_C''(\omega_1)=Z_{\omega_1}$, and by \cref{lem:products}, this happens if and only if $(\omega_1\oB\omega_2)(C)=\omega_2(Z_{\omega_1})=(\omega_1\oA\omega_2)(C)$ for all states $\omega_2$.
  Since states span $B(H_1)^*$ (see \cref{lem:density}), condition~(5) says that $T_C''(B(H_1)^*)\subseteq B(H_2)$, which is~(3) by Gantmacher’s theorem.

  The implication from (5) to (5$'$) is trivial.
  For the implication from (5$'$) to (1), suppose that $C$ is unstable.
  There are sequences in $\fD(H_1)$ and $\fD(H_2)$ along which the ordinary iterated limits of
  $\varphi_C$ exist and differ, and by \cref{lem:products}, the ultralimit states along any nonprincipal ultrafilters $\cU$ and $\cV$ on $\NN$ reproduce these limits as the two products, which therefore differ, contrary to~(5$'$).

  Finally, we turn to (5) and (6).
  By \cref{lem:products}, the maps $\omega_2\mapsto(\omega_1\oA\omega_2)(C)=\omega_2(Z_{\omega_1})$
  and $\omega_1\mapsto(\omega_1\oB\omega_2)(C)=\omega_1(W_{\omega_2})$, evaluations at the operators that the two idealized preparations leave on Bob’s and on Alice’s side, respectively, are weak$^*$
  continuous, so under~(5) their common value is separately continuous.
  Conversely, let $\hat\varphi$ be as in~(6).
  For a normal state $\sigma$ and a net $(\rho_i)$ of normal states of $B(H_1)$, indexed by a directed set, with limit $\omega_1$, continuity in the first variable gives $\hat\varphi(\omega_1,\sigma)=\lim_i\varphi_C(\rho_i,\sigma)=(\omega_1\oA\sigma)(C)$.
  The functions $\hat\varphi(\omega_1,\cdot\,)$ and $(\omega_1\oA\cdot\,)(C)$ are weak$^*$ continuous and agree on the normal states.
  Since these states are dense, we have
  $\hat\varphi(\omega_1,\omega_2)=(\omega_1\oA\omega_2)(C)$ for all pairs of states.
  Symmetrically, the function $\hat\varphi$ also coincides with $(\omega_1,\omega_2)\mapsto(\omega_1\oB\omega_2)(C)$, and~(5) follows.
\end{proof}

%#
%### SUBSECTION: Pure probes suffice
%#
\subsection{Pure probes suffice}
Pure states, those given by unit vectors, are the extreme points of the set of density operators and the simplest states to describe.
In the next proposition we show that pure preparations and pure probes detect every instability.

\begin{proposition}[Pure probes suffice]\label{prop:pure}
  Let $C\in B(H_1\otimes H_2)$ be given.
  If the restriction of $\varphi_C$ to the pairs of pure states, one of Alice’s system and one of Bob’s, is stable, then $C$ is stable.
\end{proposition}

\begin{proof}
  The idea is to represent effective operators as continuous functions on the closure of the pure states of $B(H_2)$, where Grothendieck’s criterion needs to test only a dense set of points.
  For density operators we have $\varphi_{C^*}=\overline{\varphi_C}$, so the hermitian parts $\tfrac12(C+C^*)$ and $\tfrac1{2i}(C-C^*)$ of $C$ inherit the hypothesis, and we may assume that $C$ is self-adjoint;
  then the operator $T_C(\rho)$ is self-adjoint for $\rho\ge0$.
  Let $\Omega\subseteq\cS(B(H_2))$ be the weak$^*$ closure of the pure states, a compact space.
  For a self-adjoint operator $A$, the function $\hat{A}(\omega)=\omega(A)$
  on $\Omega$ satisfies $\sup_\Omega|\hat{A}|=\|A\|$, so the map $A\mapsto\hat{A}$ is a linear isometry of the self-adjoint part of $B(H_2)$ onto a closed subspace of the space $\cC(\Omega)$ of continuous functions on $\Omega$, with the supremum norm.
  Since the pure states are dense in $\Omega$, restriction to the set $\mathsf{P}_2$ of pure states of Bob’s system is a linear isometry of $\cC(\Omega)$ onto a closed subspace of the space $\ell^\infty(\mathsf{P}_2)$ of bounded functions on $\mathsf{P}_2$.
  By Grothendieck’s theorem~\cite{Grothendieck:1952}
  (see also~\cite{Ptak:1963b}), a bounded subset of $\ell^\infty(\mathsf{P}_2)$ is relatively weakly compact exactly when evaluation, as a function on the product of the subset with $\mathsf{P}_2$, is stable in the sense of \cref{def:stable}.
  Applied to the double limit property of the hypothesis, it shows that the restrictions of the functions
  $\omega\mapsto\omega(T_C\psi)$, for pure states $\psi$ of Alice’s system, form a relatively weakly compact subset of $\ell^\infty(\mathsf{P}_2)$, and hence the functions themselves form a relatively weakly compact subset of $\cC(\Omega)$.
  Relative weak compactness passes to the closed subspace by the Hahn--Banach theorem, so the set of operators $T_C\psi$, for $\psi$ pure, is relatively weakly compact, and by Krein’s theorem its closed convex hull is weakly compact.
  This hull contains $T_C(\fD(H_1))$, because every density operator is a trace-norm convergent convex combination of pure states.
  As at the end of the first step of the proof of \cref{thm:weakcompact}, it follows that the slice map $T_C$ is weakly compact, whence the operator $C$ is stable by \cref{thm:weakcompact}.
\end{proof}

%#
%### SUBSECTION: Historical remarks
%#
\subsection{Historical remarks}\label{sec:hist-char}
For general bounded bilinear forms, the equivalence of~(2) and~(5) in \cref{thm:weakcompact} is Pym’s iterated limit criterion for Arens regularity~\cite{Pym:1965}, and the equivalence of~(3) and~(5) is due to Ülger~\cite{Ulger:1987};
the Arens products themselves were introduced by Arens~\cite{Arens:1951}.
The equivalence of~(1) and~(3) is an instance of Grothendieck’s theorem~\cite{Grothendieck:1952}, which equates the double limit condition with relative weak compactness;
the proof given here is a self-contained version adapted to states.
The factorization~(4) is the operator form of Iovino’s characterization of stable functions through reflexive spaces, recalled in the historical remarks of \cref{sec:stability}~\cite[Theorem~3.3]{Iovino:1999c};
its proof combines the Krivine--Maurey and Raynaud constructions~\cite{Krivine-Maurey:1981,Raynaud:1983a} with the Davis--Figiel--Johnson--Pełczyński factorization~\cite{Davis-Figiel-Johnson-Pelcynski:1974}.

%#
%## SECTION: Convexity: finite mixtures simulate idealized preparations
%#
\section{Convexity: finite mixtures simulate idealized preparations}\label{sec:simulation}

\noindent For a stable observable, Mazur’s theorem turns weak approximation of its effective observables into norm approximation by convex combinations.
We show that stability is exactly the possibility of reproducing every idealized response of Alice to arbitrary accuracy by a finite mixture of actual preparations,
with one mixture serving for all of Bob’s probes, normal and idealized alike, and with Bob’s system idealized first.
When the bilinear form $m_C$ of $C$ factors through a Hilbert space by bounded linear and conjugate-linear maps on trace-class operators, we also bound the number of terms in an equal-weight mixture.

In the next theorem, $T_C$ is the slice map of \cref{def:slice-and-bilinear}, which sends each density operator $\rho$ on $H_1$ to the effective operator $T_C(\rho)$ that it leaves on Bob’s side, and $\oA$ and $\oB$ are the products of \cref{def:products}, in which Alice’s and Bob’s system, respectively, is idealized first.

\begin{keytheorem}[Affine simulation]\label{thm:polyapprox}
  Let $C\in B(H_1\otimes H_2)$.
  Then the following are equivalent:
  \begin{enumerate}[label=\textup{(\arabic*)}]
  \item
    The operator $C$ is stable.
  \item
    \textup{(Bob idealized first.)} For every state $\omega_1$ of $B(H_1)$ and every $\varepsilon>0$, there is $\rho\in\fD(H_1)$ such that for all states $\omega_2$ of $B(H_2)$, we have
    \[
      \bigl|(\omega_1\oB\omega_2)(C)-\omega_2\bigl(T_C(\rho)\bigr)\bigr|\le\varepsilon .
    \]
  \item
    \textup{(Tail mixtures.)} For every sequence $(\rho_k)_{k\in\NN}$
    in $\fD(H_1)$, every nonprincipal ultrafilter $\cU$ on $\NN$, every $n\in\NN$ and every $\varepsilon>0$, there is a finite convex combination $\rho=\sum_{k\ge n}t_k\rho_k$ such that for all states $\omega_2$ of $B(H_2)$, we have
    \[
      \bigl|(\omega_1\oA\omega_2)(C)-\omega_2\bigl(T_C(\rho)\bigr)\bigr|\le\varepsilon,
    \]
    where $\omega_1=\ulimit{k}{\cU}\rho_k$ is the ultralimit state of the sequence along $\cU$.
  \end{enumerate}
\end{keytheorem}

The proof gives more:
when these conditions hold, the density operator $\rho$ in~(2) can be taken to be a finite mixture of elements of any given tail of any net in $\fD(H_1)$ that converges to $\omega_1$ in the weak$^*$ topology.

\begin{proof}
  We first show that (1) implies (2) and (3).
  The idea is that, under stability, the idealized response is an actual operator in the weak closure of a convex set of actual responses, and Mazur’s theorem converts weak approximation into norm approximation.
  By \cref{thm:weakcompact}, the bitranspose $T_C''$
  of the slice map sends $\omega_1$ to the operator $Z_{\omega_1}\in B(H_2)$, the effective operator that Alice’s idealized preparation leaves on Bob’s side (see \cref{sec:setting}), and both products equal $\omega_2(Z_{\omega_1})$.
  Let $\varepsilon>0$ be given.
  If $\omega_1=\ulimit{k}{\cU}\rho_k$ for a sequence $(\rho_k)_{k\in\NN}$ in $\fD(H_1)$ and a nonprincipal ultrafilter $\cU$ on $\NN$, then, for every $n\in\NN$, the state $\omega_1$ lies in the weak$^*$ closure of the convex hull of $\{\rho_k:k\ge n\}$.
  The bitranspose $T_C''$ is weak$^*$-to-weak$^*$ continuous, and on $B(H_2)$ the weak$^*$ topology of $B(H_2)^{**}$ is the weak topology.
  Consequently the operator $Z_{\omega_1}$ lies in the weak closure of the convex set $T_C(\operatorname{conv}\{\rho_k:k\ge n\})$, and therefore, by Mazur’s theorem, in its norm closure:
  there is $\rho$ in this convex hull with $\|Z_{\omega_1}-T_C(\rho)\|\le\varepsilon$, and then $|\omega_2(Z_{\omega_1}-T_C\rho)|\le\varepsilon$ for every state $\omega_2$.
  The same argument applies to the tails of every net in $\fD(H_1)$ that converges to $\omega_1$ in the weak$^*$ topology;
  since every state $\omega_1$ is the limit of such a net (see \cref{lem:density}), we obtain~(2) and the assertion that follows the theorem.

  \emph{Proof that (2) implies (1).}
  By \cref{lem:products}, the functional $G=T_C''(\omega_1)-T_C(\rho)$ on $B(H_2)^*$ satisfies $|G(\omega_2)|\le\varepsilon$ for all states $\omega_2$, and \cref{lem:density} bounds its norm by $4\varepsilon$.
  Since $\varepsilon$ is arbitrary and $B(H_2)$ is norm closed in its bidual, the functional $T_C''(\omega_1)$ lies in $B(H_2)$, and this holds for every state $\omega_1$.
  By the remark after \cref{thm:weakcompact}, condition~(5) of that theorem holds, and the operator $C$ is stable.

  \emph{Proof that (3) implies (1).}
  Let $(\rho_k)_{k\in\NN}$ in $\fD(H_1)$ and $(\sigma_l)_{l\in\NN}$ in $\fD(H_2)$ be sequences along which the ordinary iterated limits
  \[
    L_1=\lim_k\lim_l\varphi_C(\rho_k,\sigma_l),\qquad
    L_2=\lim_l\lim_k\varphi_C(\rho_k,\sigma_l)
  \]
  exist.
  Let $\cU$ be a nonprincipal ultrafilter on $\NN$, let $\omega_1=\ulimit{k}{\cU}\rho_k$ and $\omega_2=\ulimit{l}{\cU}\sigma_l$, and write $\alpha_k$ for the row limit $\lim_l\varphi_C(\rho_k,\sigma_l)=\omega_2(T_C\rho_k)$.
  By \cref{lem:products}, the product $(\omega_1\oA\omega_2)(C)$ equals $L_2$.
  Given $\varepsilon>0$, choose $n$ with $|\alpha_k-L_1|\le\varepsilon$ for all $k\ge n$, and choose $\rho=\sum_{k\ge n}t_k\rho_k$ as in~(3).
  Then we have
  \[
    |L_2-L_1|\le|L_2-\omega_2(T_C\rho)|+\Bigl|\sum_kt_k\alpha_k-L_1\Bigr|\le2\varepsilon .
  \]
  Since $\varepsilon>0$ is arbitrary, the two iterated limits agree, and the operator $C$ is stable.
\end{proof}

In \cref{exa:order-matters} below, we show that the order of idealization in~(2) cannot be reversed:
if Alice’s system is idealized first, approximation by actual states can succeed for an unstable observable.
The example also distinguishes approximation by an arbitrary actual preparation from approximation by mixtures of a prescribed tail of a fixed sequence.
For the order projection, by contrast, some idealized preparation has responses that no actual preparation reproduces to within any error less than $1$, the largest possible error for a yes/no test, even on normal probes (see \cref{exa:nomixture}).

\begin{example}[The order of idealization matters]\label{exa:order-matters}
  Let $(e_k)_{k\ge0}$ and $(f_l)_{l\ge1}$ be orthonormal bases of $H_1$ and $H_2$, let $Q_k$ be the projection onto the span of $f_1,\dots,f_k$, and consider the operator
  \[
    P'=|e_0\rangle\langle e_0|\otimes I+\sum_{k\ge1}|e_k\rangle\langle e_k|\otimes Q_k .
  \]
  It is a projection, and for $k\ge1$ we have $\varphi_{P'}(e_k,f_l)=1$ if and only if $l\le k$, so the projection $P'$ is unstable.
  On the other hand, for a density operator $\rho$ on $H_1$ with diagonal entries $\rho_{kk}=\langle e_k|\rho|e_k\rangle$, the operator $T_{P'}(\rho)$ is diagonal in the basis $(f_l)_{l\ge1}$, with $l$-th diagonal entry $\rho_{00}+\sum_{k\ge l}\rho_{kk}$.
  These entries form a nonincreasing sequence in $[0,1]$ whose first term is $1$, and every such sequence $(c_l)_{l\ge1}$ arises, from $\rho_{kk}=c_k-c_{k+1}$ and $\rho_{00}=\lim_lc_l$.
  The set of these diagonal operators is weak$^*$ closed, so for every state $\omega_1$, the operator $Z_{\omega_1}$, which lies in the weak$^*$ closure of $T_{P'}(\fD(H_1))$, equals $T_{P'}(\rho)$ for some actual $\rho$.
  Thus $(\omega_1\oA\omega_2)(P')=\omega_2(T_{P'}\rho)$ exactly, for all states $\omega_2$.
  For a nonprincipal ultrafilter $\cU$ on $\NN$ and $\omega_1=\ulimit{k}{\cU}e_k$, the simulating state is $e_0$;
  yet for $\omega_2=\ulimit{l}{\cU}f_l$ the two products differ:
  $(\omega_1\oB\omega_2)(P')=0$, whereas $(\omega_1\oA\omega_2)(P')=\omega_2(T_{P'}e_0)=1$.
\end{example}

\begin{example}[Maximal failure for the order projection]\label{exa:nomixture}
  Let $(e_k)_{k\in\NN}$ and $(f_l)_{l\in\NN}$ be an orthonormal basis of
  $H_1$ and an orthonormal sequence in $H_2$, respectively, let
  $P$ be the order projection, the projection onto the closed span of the vectors $e_k\otimes f_l$ with $k<l$, and let $\omega_1=\ulimit{k}{\cU}e_k$ for a nonprincipal ultrafilter $\cU$ on $\NN$.
  The effective observables $T_P(e_k)=\sum_{l>k}|f_l\rangle\langle f_l|$ decrease to $0$ strongly, although each of them has norm $1$;
  hence the operator $Z_{\omega_1}$ vanishes, and $(\omega_1\oA\sigma)(P)=0$ for every normal probe $\sigma$.
  For every $\rho\in\fD(H_1)$, however, with diagonal entries $\rho_{kk}=\langle e_k|\rho|e_k\rangle$, we have $T_P(\rho)=\sum_k\rho_{kk}\sum_{l>k}|f_l\rangle\langle f_l|$, and $\Tr(f_N\,T_P\rho)=\sum_{k<N}\rho_{kk}$ tends to $1$ as $N\to\infty$.
  Thus no actual preparation reproduces the responses of $\omega_1$ to within any error less than $1$, even on normal probes.
\end{example}

\begin{remark}[Why degree one]\label{rem:degree-one}
  In real-valued model theory, every type of a stable formula over a model, that is, every consistent assignment of a value of the formula to each parameter from the model, has a definition:
  a function on the parameters that gives these values and is a uniform limit of polynomials, or of lattice expressions, in instances of the formula~\cite{BenYaacov-Usvyatsov:2010}.
  The instances are the functions obtained from the formula by fixing its first variable at elements of the model;
  for the formula $\varphi(x,y)$ interpreted by the expectation $\varphi_C$ of an observable $C\in B(H_1\otimes H_2)$, they are the actual responses $\varphi_C(\rho,\,\cdot\,)$, for $\rho\in\fD(H_1)$.
  Since $\varphi_C$ is affine in the state, a convex combination of instances is again an instance, namely the response of the corresponding mixture (see the Introduction).
  Thus in \cref{thm:polyapprox} the approximating “formula” is an instance $\varphi_C(\rho,\,\cdot\,)$ whose parameter $\rho$ is a physical preparation, and polynomials of degree one suffice.
\end{remark}

%#
%### SUBSECTION: A dimension-free Monte Carlo bound
%#
\subsection{A dimension-free Monte Carlo bound}
Mazur’s theorem says nothing about the size of the simulating mixture in \cref{thm:polyapprox}.
In \cref{thm:montecarlo} below, we show that when the expectation of an observable factors through a Hilbert space, a sampling argument does:
the mixture can be taken to be an empirical average of independent draws, and the law of large numbers in the Hilbert space bounds the error.

For a map $a$ from $\cT(H_1)$ into a Hilbert space, we write $R_a=\sup_{\rho\in\fD(H_1)}\|a(\rho)\|$ for the radius of its image of the density operators, and for a bounded map $b$ on $\cT(H_2)$ we write $\|b\|=\sup_{\|\tau'\|_1\le1}\|b(\tau')\|$.

\begin{keytheorem}[Monte Carlo simulation]\label{thm:montecarlo}
  Let $C\in B(H_1\otimes H_2)$, let $\cH$ be a Hilbert space, and let
  $a\colon\cT(H_1)\to\cH$ and $b\colon\cT(H_2)\to\cH$
  be bounded maps, the first linear and the second conjugate-linear, such that $m_C(\tau,\tau')=\langle b(\tau'),a(\tau)\rangle$ for all $\tau\in\cT(H_1)$ and $\tau'\in\cT(H_2)$.
  Let $(\rho_k)_{k\in\NN}$ be a sequence in $\fD(H_1)$, let $\cU$ be a nonprincipal ultrafilter on $\NN$, and let $\omega_1=\ulimit{k}{\cU}\rho_k$.
  Then for every $n\in\NN$, every $\varepsilon>0$
  and every positive integer $M\ge(2R_a\|b\|/\varepsilon)^2$, there are indices $k_1,\dots,k_M\ge n$, not necessarily distinct, such that the equal-weight mixture $\bar\rho=\frac1M\sum_{j=1}^M\rho_{k_j}$ satisfies
  \[
    \bigl|(\omega_1\oA\omega_2)(C)-\omega_2\bigl(T_C\bar\rho\bigr)\bigr|\le\varepsilon
    \qquad\text{for all states }\omega_2\text{ of }B(H_2).
  \]
\end{keytheorem}

\begin{proof}
  If $\|b\|=0$, then $b=0$, so $m_C=0$ and hence $C=0$;
  any choice of $M$ indices from the prescribed tail gives the assertion.
  We may therefore assume that $\|b\|>0$.
  For $\xi\in\cH$, let $B_\xi\in B(H_2)$ be the operator with $\Tr(\tau'B_\xi)=\langle b(\tau'),\xi\rangle$ for all $\tau'\in\cT(H_2)$.
  The map $\xi\mapsto B_\xi$ is linear, satisfies $\|B_\xi\|\le\|b\|\,\|\xi\|$, and sends $a(\rho)$ to $T_C(\rho)$.
  Since $\|a(\rho_k)\|\le R_a$, weak compactness of the closed ball in the Hilbert space gives a weak limit $\xi_\infty$ along $\cU$.
  For a normal state $\sigma$, \cref{lem:products} yields $\Tr(\sigma Z_{\omega_1})=\ulimit{k}{\cU}\langle b(\sigma),a(\rho_k)\rangle=\langle b(\sigma),\xi_\infty\rangle$, which identifies $Z_{\omega_1}$ with $B_{\xi_\infty}$.
  Every tail belongs to the nonprincipal ultrafilter, so $\xi_\infty$ lies in the weak closure of $\{a(\rho_k):k\ge n\}$ and hence in the norm closure of its convex hull by Mazur’s theorem.
  Thus there is a finite convex combination $\xi=\sum_{k\ge n}t_ka(\rho_k)$ with $\|\xi_\infty-\xi\|\le\varepsilon/(2\|b\|)$.
  Let $\kappa_1,\dots,\kappa_M$ be independent random indices with $\Pr[\kappa_j=k]=t_k$.
  Since the summands are independent with expectation $\xi$, we have
  \[
    \EE\Bigl\|\xi-\frac1M\sum_{j=1}^Ma(\rho_{\kappa_j})\Bigr\|^2
    =\frac1M\,\EE\bigl\|a(\rho_{\kappa_1})-\xi\bigr\|^2\le\frac{R_a^2}{M},
  \]
  so some realization $k_1,\dots,k_M$ satisfies $\|\xi-\frac1M\sum_ja(\rho_{k_j})\|\le R_a/\sqrt M\le\varepsilon/(2\|b\|)$.
  Hence $\|Z_{\omega_1}-T_C\bar\rho\|=\|B_{\xi_\infty-a(\bar\rho)}\|\le\varepsilon$.
  Evaluation at any state $\omega_2$, which has norm one, gives the asserted uniform error bound.
\end{proof}

The operator $C$ of \cref{thm:montecarlo} is stable, by \cref{prop:hilbertfactor} with the two variables exchanged.

In the next corollary, a mixture \emph{simulates} $\omega_1$ within $\varepsilon$, uniformly over all probes, when it satisfies the inequality of \cref{thm:montecarlo}.
The swap operator $F$, which exchanges two systems with the same Hilbert space $K$, and the projections $P_\pm$ are those of \cref{sec:swap}, and the equality tests $[A=B]$ between self-adjoint operators with pure point spectra are those of \cref{prop:comparison}.

\begin{corollary}\label{cor:montecarlo}
  Let $(\rho_k)_{k\in\NN}$ be a sequence of density operators, let $\cU$ be a nonprincipal ultrafilter on $\NN$, let $\omega_1=\ulimit{k}{\cU}\rho_k$, let $n\in\NN$, and let $\varepsilon>0$.
  For the swap operator $F$ and for the equality tests $[A=B]$,
  some equal-weight mixture of $\lceil4/\varepsilon^2\rceil$ preparations $\rho_k$ with $k\ge n$, not necessarily distinct, simulates $\omega_1$ within $\varepsilon$, uniformly over all probes.
  For the swap test $P_+$ and the Hong--Ou--Mandel projection $P_-$, equal-weight mixtures of $\lceil1/\varepsilon^2\rceil$ such preparations suffice.
\end{corollary}

\begin{proof}
  For $F$, take for $\cH$ the space $S_2(K)$ of Hilbert--Schmidt operators on $K$, with $a(\tau)=\tau$ and $b(\tau')=\tau'^*$, whence $\langle\tau'^*,\tau\rangle_{S_2}=\Tr(\tau'\tau)$;
  since the Hilbert--Schmidt norm satisfies $\|\tau\|_2\le\|\tau\|_1$, we have $R_a\le1$ and $\|b\|\le1$.
  On states, the slice maps satisfy $T_{P_\pm}(\rho)=\frac12(I\pm T_F(\rho))$, and correspondingly, writing
  $Z^F_{\omega_1}$ and $Z^{P_\pm}_{\omega_1}$ for the operators $Z_{\omega_1}$ of $F$ and of $P_\pm$, we have $Z^{P_\pm}_{\omega_1}=\frac12(I\pm Z^F_{\omega_1})$, so the error for $P_\pm$ is half the error for $F$.
  For $[A=B]$, take $\cH=\ell^2$, $a(\tau)=(\Tr\tau\Pi^A_\lambda)_\lambda$ and $b(\tau')=(\overline{\Tr\tau'\Pi^B_\lambda})_\lambda$, where $\Pi^A_\lambda$ and $\Pi^B_\lambda$ are the spectral projections of $A$ and $B$ for the eigenvalue $\lambda$;
  then $\|a(\tau)\|_2\le\sum_\lambda|\Tr\tau\Pi^A_\lambda|\le\|\tau\|_1$, and similarly for $b$.
\end{proof}

These bounds do not depend on the dimension of the systems.

Hilbert space can be replaced by spaces of Rademacher type.
A Banach space has \emph{Rademacher type} $r$
if, for every positive integer $m$ and all vectors $x_1,\dots,x_m$ of the space, the average of $\|\pm x_1\pm\dots\pm x_m\|$ over the $2^m$ choices of signs is at most a constant multiple, independent of $m$, of $(\|x_1\|^r+\dots+\|x_m\|^r)^{1/r}$.
If, in the setting of \cref{thm:montecarlo}, the slice map
$T_C$ factors through a reflexive Banach space of Rademacher type
$r>1$ in place of the Hilbert space $\cH$, the same argument, with the type-$r$ form of Maurey’s empirical method~\cite{Pisier:1981}, yields mixtures of size $M=O(\varepsilon^{-r/(r-1)})$ at accuracy $\varepsilon$.
Reflexivity alone, which characterizes stability, yields no rate through this argument (see \cref{ques:rates}).

%#
%### SUBSECTION: Historical remarks
%#
\subsection{Historical remarks}\label{sec:hist-sim}
That the definitions of the types of a stable formula over a model, in real-valued model theory, are uniform limits of polynomials or lattice expressions in instances follows from the Stone--Weierstrass theorem (see~\cite{BenYaacov-Usvyatsov:2010});
Ben Yaacov observed that, by Mazur’s theorem, averages of instances already suffice~\cite{BenYaacov:2014a}.
Pták’s lemma~\cite{Ptak:1963b} is the classical source of convex means in the study of weak compactness.
The sampling argument in the proof of \cref{thm:montecarlo} is Maurey’s empirical method, published by Pisier~\cite{Pisier:1981}.

%#
%## SECTION: Probability: Fubini’s theorem for finitely additive measures
%#
\section{Probability: Fubini’s theorem for finitely additive measures}\label{sec:probability}

\noindent Randomization gives another formulation of the same exchange problem.
A finitely additive probability measure on preparations has a barycenter that may be idealized, and the two iterated averages of the expectation evaluate the two Arens extensions at these barycenters.
We show that their agreement for every pair of such measures is equivalent to stability.
For countably additive Borel probabilities in the trace-norm topology, the barycenters are normal and the averages always commute;
finitely additive measures are therefore essential to the characterization.
The order projection makes the obstruction explicit through noncommuting ultrafilter products.

%#
%### SUBSECTION: Fubini’s theorem for means
%#
\subsection{Fubini’s theorem for means}
A \emph{mean} on a set $\Gamma$ is a finitely additive probability measure defined on all subsets of $\Gamma$;
integration identifies such a measure with a positive linear functional on the bounded complex-valued functions on $\Gamma$ that takes the value $1$ at the constant function $1$, and we denote the measure and the functional by the same letter $\mu$.
We write $\int g\,d\mu$ for $\mu(g)$.
The \emph{barycenter} of a mean $\mu$ on $\fD(H)$ is the state $\omega_\mu(A)=\int\Tr(\rho A)\,d\mu(\rho)$ of $B(H)$.

\begin{keytheorem}[Fubini’s theorem for means]\label{thm:fubini}
  Let $C\in B(H_1\otimes H_2)$.
  Then $C$ is stable if and only if, for all means $\mu$ on $\fD(H_1)$ and $\nu$ on $\fD(H_2)$,
  \[
    \int\!\!\Bigl(\int\varphi_C(\rho,\sigma)\,d\mu(\rho)\Bigr)d\nu(\sigma)
    =\int\!\!\Bigl(\int\varphi_C(\rho,\sigma)\,d\nu(\sigma)\Bigr)d\mu(\rho).
  \]
\end{keytheorem}

\begin{proof}
  The idea is that, since the expectation is affine in each state, iterated integrals depend only on barycenters, whence they are the values at $C$ of the two products of the barycenters (see \cref{def:products}).
  The slice map $S_C$ sends each preparation $\sigma$ of Bob’s system to the effective operator $S_C(\sigma)$ that it leaves on Alice’s side, and the expectation $\varphi_C(\rho,\sigma)=\Tr(\rho\,S_C\sigma)$ is affine in $\rho$.
  Therefore the inner integral $\int\varphi_C(\rho,\sigma)\,d\mu(\rho)$
  equals $\omega_\mu(S_C\sigma)=\Tr(\sigma Z_{\omega_\mu})$, where
  $Z_{\omega_\mu}$ is the effective operator that the barycenter $\omega_\mu$ leaves on Bob’s side, and integration against $\nu$ yields $(\omega_\mu\oA\omega_\nu)(C)$, the product with Alice’s system idealized first.
  Symmetrically, the other iterated integral is $(\omega_\mu\oB\omega_\nu)(C)$.
  Every state $\omega$ is a barycenter:
  by \cref{lem:density}, there is a net $(\rho_i)$ of normal states, indexed by a directed set, with limit
  $\omega$, and if $\cW$ is an ultrafilter on its index set containing all tails, then $\mu(g)=\ulimit{i}{\cW}g(\rho_i)$ defines a mean with barycenter $\omega$.
  Thus the identity holds for all means if and only if the two products agree on all pairs of states, which is condition~(5) of \cref{thm:weakcompact}, that is, if and only if $C$ is stable.
\end{proof}

%#
%### SUBSECTION: Countably additive measures
%#
\subsection{Countably additive measures}
The means in \cref{thm:fubini} cannot be replaced by countably additive probability measures, because, as we show in \cref{prop:sigmaadditive} below, these never detect instability.

\begin{proposition}[Countably additive measures never detect instability]\label{prop:sigmaadditive}
  Let $\mu$ be a countably additive Borel probability measure on $\fD(H_1)$, with the trace norm, and let $\nu$ be a mean on $\fD(H_2)$.
  Then for every $C\in B(H_1\otimes H_2)$, the two iterated integrals of $\varphi_C$ against $\mu$ and $\nu$ agree.
\end{proposition}

\begin{proof}
  Since $H_1$ is separable, the trace-class space $\cT(H_1)$ is separable, and the identity map $\rho\mapsto\rho$ on $\fD(H_1)$ is strongly measurable and has constant trace norm one.
  Its Bochner integral $\bar\rho=\int\rho\,d\mu(\rho)$ therefore exists, is positive, and has trace one;
  the barycenter of $\mu$ is the normal state given by $\bar\rho$.
  Both iterated integrals are well defined:
  the functions $\rho\mapsto\varphi_C(\rho,\sigma)$ are Lipschitz with common constant $\|C\|$, and averaging them against the mean $\nu$ preserves the bound, so the outer integrand for $\mu$ is Borel measurable.
  By the proof of \cref{thm:fubini}, the two iterated integrals are the values at $C$ of the two products of the barycenters, and these values agree by \cref{cor:normal-factor}.
\end{proof}

Failure of this Fubini identity therefore requires genuinely finitely additive averages on both systems.
Their barycenters must both be idealized, since a normal factor makes the two products agree.

%#
%### SUBSECTION: Ultrafilter products
%#
\subsection{Ultrafilter products}
For nonprincipal ultrafilters $\cU$ and $\cV$ on $\NN$, let $\cU\otimes\cV$ be the ultrafilter on $\NN^2$ consisting of the sets $Q\subseteq\NN^2$ with $\{k:\{l:(k,l)\in Q\}\in\cV\}\in\cU$.
Let $C\in B(H_1\otimes H_2)$ be given, let $(\rho_k)_{k\in\NN}$ and $(\sigma_l)_{l\in\NN}$ be sequences in $\fD(H_1)$ and $\fD(H_2)$, and let $\omega_1=\ulimit{k}{\cU}\rho_k$ and $\omega_2=\ulimit{l}{\cV}\sigma_l$ be their ultralimit states.
By \cref{lem:products}, the product $(\omega_1\oB\omega_2)(C)$, with Bob’s system idealized first, is the limit of
$\varphi_C(\rho_k,\sigma_l)$ along $\cU\otimes\cV$, and the product $(\omega_1\oA\omega_2)(C)$ is the limit along the transpose of $\cV\otimes\cU$, its image under the map $(l,k)\mapsto(k,l)$.
The tensor square of a single nonprincipal ultrafilter $\cW$ on $\NN$ is not symmetric:
the set $\{(k,l):k<l\}$ belongs to $\cW\otimes\cW$ but not to the transpose of $\cW\otimes\cW$.
The order projection of \cref{thm:unstable} is built from orthonormal sequences $(e_k)_{k\in\NN}$ in $H_1$ and $(f_l)_{l\in\NN}$ in $H_2$:
it projects onto the closed span of the vectors $e_k\otimes f_l$ with $k<l$, whence its expectation on the pairs $(e_k,f_l)$ is the indicator of the set $\{(k,l):k<l\}$.
For $\omega_1=\ulimit{k}{\cW}e_k$ and $\omega_2=\ulimit{l}{\cW}f_l$, the product $\oB$, with Bob’s system idealized first, therefore takes the value $1$ at the order projection, whereas the product $\oA$ takes the value $0$.
The ultrafilter $\cW$ corresponds to a $\{0,1\}$-valued mean on $\NN$, which gives measure $1$ exactly to the sets in $\cW$;
by \cref{prop:sigmaadditive}, no countably additive version of this example exists.

%#
%### SUBSECTION: Means, mixtures and samples
%#
\subsection{Means, mixtures and samples}
Finitely supported means on $\fD(H_1)$ correspond to the finite mixtures a laboratory can implement.
In this language, \cref{thm:polyapprox} says that, for a stable observable and every $\varepsilon>0$, a mean
$\mu$ on $\fD(H_1)$ in \cref{thm:fubini} can be replaced by a finitely supported mean, with an error at most
$\varepsilon$ uniformly in the mean $\nu$ on $\fD(H_2)$, because the iterated integrals depend on $\mu$ only through its barycenter (see the proof of \cref{thm:fubini}).
For observables whose expectation factors through a Hilbert space, \cref{thm:montecarlo} replaces Mazur’s existence statement by an empirical average of independent draws.

%#
%### SUBSECTION: Historical remarks
%#
\subsection{Historical remarks}\label{sec:hist-prob}
Pták’s combinatorial lemma~\cite{Ptak:1963b} derives weak compactness from the existence of convex means that interchange limits, and the Fubini form for finitely additive means was studied by Simons~\cite{Simons:1968} and by Levy~\cite{Levy:2007}.
In model theory, Keisler measures are finitely additive probability measures on definable sets~\cite{Keisler:1987}.
For global measures in a theory without the independence property, an NIP theory, invariance over a common small model gives Borel dependence on parameters and hence products in both orders of integration.
Here global means that parameters range over a large saturated model, and invariance means that the measure is unchanged by automorphisms fixing the small model pointwise.
Commutation of these products is tied to stability and to generic stability, a property of a single type or measure even when the theory is not stable.
In an NIP theory, a global type or measure is generically stable exactly when it is both definable over and finitely satisfiable in some small model~\cite{Hrushovski-Peterzil-Pillay:2008,Hrushovski-Pillay:2011a}:
definability gives continuous dependence on parameter types, while finite satisfiability requires each formula in the type, or each definable set of positive measure, to have a witness in the model.
The two products $\oA$ and $\oB$ of \cref{sec:setting}, with Alice’s and with Bob’s system idealized first, respectively, are the operator-algebraic counterparts of these two orders of integration (see the proof of \cref{thm:fubini}).

%#
%## SECTION: Independence: Harrington’s lemma and the independent product
%#
\section{Independence: Harrington’s lemma and the independent product}\label{sec:independence}

\noindent On the space of stable bounded operators, we construct a symmetric product of arbitrary states of the two systems.
For each stable observable, we prove that the resulting expectation is the unique separately weak$^*$ continuous extension of its values on pairs of density operators.
We compare this construction with Harrington’s lemma, according to which the definitions of two types over a model of a stable theory commute~\cite{Harnik-Harrington:1984}.
We recall the independence calculus that gives this exchange its model-theoretic meaning, and then establish the analytic counterpart for observables.

%#
%### SUBSECTION: Abstract independence
%#
\subsection{Abstract independence}\label{sec:abstract-independence}
The axiomatic study of independence is older than model theory:
it begins with Steinitz’s exchange argument for transcendence bases and with van der Waerden’s abstract dependence relations, and leads to matroids and pregeometries, closure operators with the exchange property of linear span
(see the historical remarks at the end of this section).
In model theory, Shelah’s forking~\cite{Shelah:1990a} carries independence beyond pregeometries.
Its negation, nonforking, singles out the extensions of a type to larger sets of parameters whose realizations are independent of the new parameters;
nonforking independence is not given by a closure operator, yet in a stable theory it is symmetric and has existence and uniqueness of independent extensions over models.
Lascar and Poizat~\cite{Lascar-Poizat:1979} described it through definability of types and through heirs and coheirs of types over a model:
the extensions in which every formula, once its new parameters are replaced by suitable parameters from the model, belongs to the original type, and those in which every formula is satisfied in the model, respectively.
Harnik and Harrington~\cite{Harnik-Harrington:1984} characterized stable theories as those carrying an abstract independence relation with these properties.

%#
%### SUBSECTION: Harrington’s lemma as an exchange of limits
%#
\subsection{Harrington’s lemma as an exchange of limits}\label{sec:harrington}
We recall the two-valued setting of Iovino’s notes on stability theory~\cite[Chapter~4]{Iovino:0000e}.
Let $\cM$ be a model of a stable complete first-order theory, let $p$ and $q$ be complete types over $\cM$ in finite tuples of variables $x$ and $y$, and let $\varphi(x,y)$ be a formula with parameters from $\cM$.
Stability provides definitions $d_p\varphi(y)$ and $d_q\varphi(x)$ over $\cM$:
for every tuple $b$ from $\cM$, the formula $\varphi(x,b)$ belongs to $p$ if and only if $d_p\varphi(b)$ holds in $\cM$, and symmetrically for $q$.
In \cref{sec:independent-product} below, a type over $\cM$ will correspond to a state of $B(H)$.
Realizing one coheir before another corresponds to taking the limit over the states of the corresponding system as the inner limit.

\begin{fact}[Harrington’s lemma~\cite{Harnik-Harrington:1984}]\label{fact:harrington}
  Let $\cM$ be a model of a stable complete first-order theory, let $p$ and $q$ be complete types over $\cM$ in finite tuples of variables $x$ and $y$, and let $\varphi(x,y)$ be a formula with parameters from $\cM$.
  Then the definitions $d_p\varphi(y)$ and $d_q\varphi(x)$ of $p$ and $q$ over $\cM$ satisfy $d_p\varphi(y)\in q$ if and only if $d_q\varphi(x)\in p$.
\end{fact}

Harrington’s lemma is the key step in the proof that nonforking independence is symmetric in stable theories, and hence in the Harnik--Harrington characterization~\cite{Harnik-Harrington:1984}.
For us, the point is that Harrington’s lemma asserts that two iterated limits can be exchanged:
every type over $\cM$ is finitely satisfiable in $\cM$, so the type
$p$ is the trace of an ultrafilter $\cU$ on the tuples from $\cM$ of the length of $x$, in the sense that a formula $\theta(x)$ with parameters from $\cM$ belongs to $p$ if and only if the set of its solutions in $\cM$ belongs to $\cU$.
Let $\cV$ be an ultrafilter with trace $q$, and regard $\varphi$ as a $\{0,1\}$-valued function on pairs of tuples from $\cM$.
Then the two sides of Harrington’s lemma take the form
\[
\begin{aligned}
  d_p\varphi(y)\in q&\quad\text{if and only if}\quad \ulimit{b}{\cV}\ \ulimit{a}{\cU}\varphi(a,b)=1,\\
  d_q\varphi(x)\in p&\quad\text{if and only if}\quad \ulimit{a}{\cU}\ \ulimit{b}{\cV}\varphi(a,b)=1,
\end{aligned}
\]
so, read in the language of types, the agreement of the two iterated ultrafilter limits of \cref{cor:ultra-stable} is exactly what Harrington’s lemma asserts.
The computation takes place entirely in $\cM$:
it needs neither heirs nor elements realizing types in an elementary extension of $\cM$.
The same computation explains the two orders of realization:
let $b$ realize $q$, and let $a$ realize the coheir of
$p$ over $\cM b$, the set $\cM$ with the entries of $b$ adjoined, determined by $\cU$, that is, the type over $\cM b$ that contains a formula exactly when the set of its solutions in $\cM$ belongs to $\cU$.
Then the formula $\varphi(a,b)$ holds if and only if $\ulimit{a'}{\cU}\ulimit{b'}{\cV}\varphi(a',b')=1$, and realizing $p$ first and then the coheir of $q$ over $\cM a$ determined by $\cV$ exchanges the limits.
Thus the two orders of realizing coheirs yield the same type exactly when the two iterated limits agree, and the symmetry of independence reduces to the agreement of iterated ultrafilter limits.
The proofs of these statements, and of the Harnik--Harrington characterization, are given in the same notes~\cite[Chapter~4]{Iovino:0000e};
the real-valued versions, for formulas with values in a compact interval, are proved in Iovino’s thesis~\cite{Iovino:1994a} (see also~\cite{BenYaacov-Usvyatsov:2010}).

%#
%### SUBSECTION: The independent product of idealized preparations
%#
\subsection{The independent product of idealized preparations}\label{sec:independent-product}
The correspondence with \cref{sec:setting} is as follows.
A type over a model corresponds to a preparation represented by a state of $B(H)$, which may be normal or idealized.
The trace of an ultrafilter corresponds to an ultralimit state.
Realizing Bob’s coheir first and Alice’s over it computes the inner limit over Bob’s states, just as the product $\oB$ of \cref{def:products}, in which Bob’s system is idealized first, computes it;
the other order computes $\oA$.
Under this correspondence, the conclusion of Harrington’s lemma becomes condition~(5) of \cref{thm:weakcompact}, the agreement of the two products on every pair of states.
The analogy is one of form:
the theorem below is proved directly from \cref{thm:weakcompact}, without model theory.
To state the resulting symmetry, let $\cO$ be the set of stable bounded operators in $B(H_1\otimes H_2)$, whose self-adjoint members are the stable observables.
In \cref{prop:operatorsystem} below, we show that $\cO$ is a norm-closed self-adjoint subspace of $B(H_1\otimes H_2)$ containing $I$;
a \emph{state} of $\cO$ is a linear functional on $\cO$ that is positive on its positive elements and takes the value $1$ at $I$.
For states
$\omega_1$ of $B(H_1)$ and $\omega_2$ of $B(H_2)$, the \emph{independent product} $\omega_1\otimes\omega_2$ is the functional on $\cO$ that takes at each $C\in\cO$ the common value of the two products, which agree at $C$ by \cref{thm:weakcompact}:
\[
  (\omega_1\otimes\omega_2)(C)=(\omega_1\oA\omega_2)(C)=(\omega_1\oB\omega_2)(C)
  \qquad(C\in\cO).
\]
The independent product $\omega_2\otimes\omega_1$, for states of $B(H_2)$ and of $B(H_1)$, is defined in the same way on the stable bounded operators on $H_2\otimes H_1$, and $F_{12}\colon H_1\otimes H_2\to H_2\otimes H_1$ denotes the unitary that exchanges the two factors.

\begin{keytheorem}[The independent product]\label{thm:independent-product}
  Let $\cO$ be the set of stable bounded operators in $B(H_1\otimes H_2)$.
  Then the following hold:
  \begin{enumerate}[label=\textup{(\alph*)}]
  \item
    \textup{(Existence.)} For all states $\omega_1$ of $B(H_1)$ and $\omega_2$ of $B(H_2)$, the independent product $\omega_1\otimes\omega_2$ is a state of $\cO$.
  \item
    \textup{(Extension.)} If states $\omega_1$ of $B(H_1)$ and $\omega_2$ of $B(H_2)$ are given by density operators $\rho$ and $\sigma$, then $(\omega_1\otimes\omega_2)(C)=\Tr((\rho\otimes\sigma)C)$ for all $C\in\cO$.
  \item
    \textup{(Uniqueness.)} For every $C\in\cO$, the assignment
    $(\omega_1,\omega_2)\mapsto (\omega_1\otimes\omega_2)(C)$
    is separately weak$^*$ continuous, and it is the only separately weak$^*$ continuous function on $\cS(B(H_1))\times\cS(B(H_2))$ that takes the value $\Tr((\rho\otimes\sigma)C)$ at every pair of normal states given by density operators $\rho$ and $\sigma$.
  \item
    \textup{(Symmetry.)} For every $C\in\cO$, the operator $F_{12}CF_{12}^*$
    on $H_2\otimes H_1$ is stable, and $(\omega_2\otimes\omega_1)(F_{12}CF_{12}^*)=(\omega_1\otimes\omega_2)(C)$
    for all states $\omega_1$ of $B(H_1)$ and $\omega_2$ of $B(H_2)$.
  \end{enumerate}
\end{keytheorem}

\begin{proof}
  (a) The slice map $S_C$, which sends each preparation of Bob’s system to the effective operator that it leaves on Alice’s side, is linear in
  $C$, and so is $(\omega_1\oA\omega_2)(C)=\omega_2(Z_{\omega_1})$, where $Z_{\omega_1}=S_C'(\omega_1)$ is the effective operator that $\omega_1$ leaves on Bob’s side.
  If $C\ge0$, then the operator $S_C(\sigma)$ is positive for every positive $\sigma\in\cT(H_2)$, so $\Tr(\sigma Z_{\omega_1})=\omega_1(S_C\sigma)\ge0$ for every such $\sigma$;
  hence the operator $Z_{\omega_1}$ is positive, and $(\omega_1\oA\omega_2)(C)=\omega_2(Z_{\omega_1})\ge0$.
  For $C=I$, we have $S_I(\sigma)=\Tr(\sigma)I$ and $Z_{\omega_1}=I$.

  (b) This is \cref{cor:normal-factor}.

  (c) Both statements come from the step of the proof of \cref{thm:weakcompact} in which condition~(5), the agreement of the two products, is shown to be equivalent to condition~(6), the existence of a separately continuous extension.

  (d) Write $C'=F_{12}CF_{12}^*$.
  For density operators we have $\varphi_{C'}(\sigma,\rho)=\varphi_C(\rho,\sigma)$, so the operator $C'$ is stable:
  interchanging the two variables interchanges the two iterated limits.
  Let $(\rho_i)$ and $(\sigma_j)$ be nets of normal states of $B(H_1)$ and of $B(H_2)$, indexed by directed sets, with limits $\omega_1$ and $\omega_2$.
  By \cref{lem:products}, the product $(\omega_2\oB\omega_1)(C')$, in which the second factor $\omega_1$ is idealized first, equals $\lim_j\lim_i\varphi_{C'}(\sigma_j,\rho_i)=\lim_j\lim_i\varphi_C(\rho_i,\sigma_j) =(\omega_1\oA\omega_2)(C)$.
  Since $C'$ and $C$ are stable, the two sides are the independent products $(\omega_2\otimes\omega_1)(C')$ and $(\omega_1\otimes\omega_2)(C)$, and (d) follows.
\end{proof}

In the language of independence relations, (a) is the existence of an independent combination, (c) is its uniqueness (stationarity), and (d) is its symmetry;
by \cref{thm:weakcompact}, the stable bounded operators are exactly those at which the two products agree for all pairs of states.
Outside $\cO$ the independent combination is lost:
by \cref{thm:weakcompact}, an unstable observable admits no separately continuous extension of its expectation to idealized states, and the two orders of idealization disagree on some pair of idealized states.
For the order projection, whose expectation on pairs of vectors from two orthonormal sequences records whether Alice’s index is smaller than Bob’s, the two orders of idealization give
$1$ with Bob’s system idealized first and $0$ with Alice’s (see \cref{thm:unstable}).

%#
%### SUBSECTION: Historical remarks
%#
\subsection{Historical remarks}\label{sec:hist-independence}
Steinitz~\cite{Steinitz:1910} introduced transcendence bases of field extensions and proved, by the exchange argument that also gives the dimension of a vector space, that any two of them have the same cardinality.
Van der Waerden~\cite{vanderWaerden:1930} observed that linear and algebraic dependence obey the same few rules,
and derived the existence of bases and the invariance of their cardinality from an abstract dependence relation whose central axiom is exchange.
Whitney~\cite{Whitney:1935} isolated these rules as \emph{matroids}, with equivalent axiomatizations by independent sets, bases, circuits and rank;
Nakasawa arrived at the theory independently at about the same time~\cite{Nishimura-Kuroda:2009}, and Birkhoff~\cite{Birkhoff:1935} and Mac Lane~\cite{MacLane:1936} related it to lattices and to projective geometry.
Crapo and Rota~\cite{Crapo-Rota:1970} developed matroids as \emph{combinatorial geometries};
in their terminology a closure operator with the exchange property is a \emph{pregeometry}, the term still used in model theory.
Baldwin and Lachlan~\cite{Baldwin-Lachlan:1971a} showed that algebraic closure in a strongly minimal set, an infinite definable set whose definable subsets are finite or cofinite in every elementary extension, is a pregeometry, so that every model of a strongly minimal theory has a dimension.
Kim and Pillay~\cite{Kim-Pillay:1997} extended the Harnik--Harrington characterization to simple theories, a class that properly contains the stable theories.

%#
%## SECTION: Combinatorics: polynomial versus exponential growth
%#
\section{Combinatorics: ladders, trees and uniform stability}\label{sec:trees}

\noindent The preceding characterizations concern stability on the given preparation spaces.
To bound adaptive discrimination at each positive margin, we pass to uniform stability.
We show that bounded ladder lengths, bounded decision-tree heights and stability of the ultrapower expectation are equivalent forms of this condition.
The tree bounds yield polynomial growth and finite mistake bounds in the fixed-threshold game, uniform in the threshold, as stated in Theorem~B.
A stable observable need not satisfy this stronger condition, as the sector comparison below demonstrates.

Throughout the section, we fix a self-adjoint observable $C$ on $H_1\otimes H_2$, where $H_1$ and $H_2$ are the Hilbert spaces of Alice’s and Bob’s systems;
its expectation $\varphi_C(\rho,\sigma)=\Tr\bigl((\rho\otimes\sigma)C\bigr)$, for density operators $\rho\in\fD(H_1)$ and $\sigma\in\fD(H_2)$, is therefore real.
For a general bounded operator, the results apply to its hermitian parts (see \cref{def:uniform}).
For every integer $n\ge0$, we write $\{0,1\}^n$ for the set of $0$--$1$ sequences of length $n$, and $\{0,1\}^{<n}$ for the union $\bigcup_{m<n}\{0,1\}^m$;
the sets of infinite and of finite $0$--$1$ sequences are written $\{0,1\}^\NN$ and $\{0,1\}^{<\NN}$, respectively.
For a finite $0$--$1$ sequence $s$ and a finite or infinite $0$--$1$ sequence $\eta$, we write $s\seg\eta$ if $s$ is an initial segment of $\eta$, and $\cat{s}{\iota}$ for $s$ followed by $\iota\in\{0,1\}$.

%#
%### SUBSECTION: From ladders to decision trees
%#
\subsection{From ladders to decision trees}\label{sec:ladders-trees}

\begin{definition}[Ladders and trees]\label{def:trees}
  Let $\delta>0$ be given.
  \begin{enumerate}[label=\textup{(\alph*)}]
  \item
    For a positive integer $N$, a \emph{$\delta$-ladder of length $N$} for $C$ consists of states $\rho_1,\dots,\rho_N\in\fD(H_1)$ and $\sigma_1,\dots,\sigma_N\in\fD(H_2)$, a real number $\gamma$ and a sign $\theta\in\{\pm1\}$ such that
    \[
      \theta\bigl(\varphi_C(\rho_i,\sigma_j)-\gamma\bigr)\ge\delta/2\ \text{ if }i\le j,
      \qquad
      \theta\bigl(\varphi_C(\rho_i,\sigma_j)-\gamma\bigr)\le-\delta/2\ \text{ if }i>j .
    \]
    Infinite ladders are defined in the same way.
  \item
    For an integer $n\ge0$, a \emph{$\delta$-tree of height $n$} for $C$ consists of states
    $\rho_\eta\in\fD(H_1)$ for $\eta\in\{0,1\}^n$ (the \emph{leaves}), states $\sigma_s\in\fD(H_2)$ for $s\in\{0,1\}^{<n}$ (the \emph{nodes}), a real number $\gamma$ and a sign $\theta\in\{\pm1\}$ such that, whenever $\cat{s}{\iota}\seg\eta$, we have
    \[
      \theta\bigl(\varphi_C(\rho_\eta,\sigma_s)-\gamma\bigr)\ge\delta/2\ \text{ if }\iota=1,
      \qquad
      \theta\bigl(\varphi_C(\rho_\eta,\sigma_s)-\gamma\bigr)\le-\delta/2\ \text{ if }\iota=0 .
    \]
  \end{enumerate}
\end{definition}

A tree describes an adaptive measurement protocol.
Bob starts at the root, the empty sequence $\emptyset$, and, at each node $s$, probes with $\sigma_s$;
the expectation of $C$ falls on one side of the threshold $\gamma$ or the other, and Bob moves to $\cat{s}{1}$ or $\cat{s}{0}$ accordingly.
If Alice prepared one of the leaves, then after $n$ probes he knows which of the $2^n$ preparations $\rho_\eta$ she made.
Two leaves that first separate at the node $s$ have responses to $\sigma_s$ that differ by at least $\delta$.

\begin{proposition}[Infinite ladders from instability]\label{prop:inst-ladder}
  If sequences $(\rho_k)_{k\in\NN}$ in $\fD(H_1)$ and
  $(\sigma_l)_{l\in\NN}$ in $\fD(H_2)$ are such that the iterated limits
  $L_1=\lim_k\lim_l\varphi_C(\rho_k,\sigma_l)$ and $L_2=\lim_l\lim_k\varphi_C(\rho_k,\sigma_l)$ exist and differ, then for every $\delta$ with $0<\delta<|L_1-L_2|$ there is an infinite $\delta$-ladder for $C$ whose states are taken from these sequences.
\end{proposition}

\begin{proof}
  For any nonprincipal ultrafilters $\cU$ and $\cV$ on $\NN$, the iterated limits along $\cU$ and $\cV$ equal the ordinary iterated limits $L_1$ and $L_2$.
  Apply \cref{lem:ladder} with $\varepsilon=(|L_1-L_2|-\delta)/2$, and take $\gamma=(L_1+L_2)/2$ and $\theta$ the sign of $L_1-L_2$.
\end{proof}

The passage from ladders to trees rests on a single arithmetic observation, which we prove in \cref{prop:ladders-to-trees} below:
if each finite binary sequence is placed at the point of
$(0,1)$ whose binary expansion consists of its digits complemented and followed by a final
$1$, then each leaf lies to the left of every node at which its branch takes the label $1$, and to the right of every node at which it takes the label $0$ (see \cref{fig:ladder-tree}).

\begin{proposition}[Ladders give trees]\label{prop:ladders-to-trees}
  For every $\delta>0$ and every integer $n\ge0$, every $\delta$-ladder of length $2^{n+1}-1$ for $C$ yields a $\delta$-tree of height $n$ for $C$, with the same threshold $\gamma$ and sign $\theta$, and with leaves and nodes taken from the ladder.
\end{proposition}

\begin{proof}
  For a $0$--$1$ sequence $s$ of length $|s|\le n$, with digits $s(0),\dots,s(|s|-1)$, let
  \[
    \pos(s)=\sum_{k<|s|}\bigl(1-s(k)\bigr)2^{-(k+1)}+2^{-(|s|+1)}\in(0,1)
  \]
  The binary expansion of $\pos(s)$ is $1-s(0),\dots,1-s(|s|-1),1$, so the map $\pos$ is injective, and its image has $2^{n+1}-1$ elements.
  We claim that if $\cat{s}{\iota}\seg\eta$ with $\eta\in\{0,1\}^n$, then the leaf $\eta$ lies to the left of $s$, that is, $\pos(\eta)<\pos(s)$, if and only if $\iota=1$.
  Indeed, with $k=|s|$, we have
  \[
    \pos(\eta)-\pos(s)=-\iota\,2^{-(k+1)}+\Delta,\qquad
    \Delta=\sum_{k<j<n}\bigl(1-\eta(j)\bigr)2^{-(j+1)}+2^{-(n+1)},
  \]
  and $0<\Delta\le(2^{-(k+1)}-2^{-n})+2^{-(n+1)}<2^{-(k+1)}$.
  The bound $0<\Delta<2^{-(k+1)}$ is the point of the construction:
  the digits after position $k$ can move $\pos(\eta)$ by less than the digit at position $k$, so the turn taken at $s$ alone decides on which side of $\pos(s)$ the leaf lies.

  Let $\rho_1,\dots,\rho_N$ and $\sigma_1,\dots,\sigma_N$, with $N=2^{n+1}-1$, be the states of the ladder, and let $\pi$ be the increasing bijection from the image of $\pos$ onto $\{1,\dots,N\}$.
  The leaves and nodes of the tree are $\rho_\eta=\rho_{\pi(\pos(\eta))}$, for $\eta\in\{0,1\}^n$, and $\sigma_s=\sigma_{\pi(\pos(s))}$, for $s\in\{0,1\}^{<n}$.
  If $\cat{s}{\iota}\seg\eta$ and $\iota=1$, then $\pi(\pos(\eta))<\pi(\pos(s))$, and the ladder yields $\theta(\varphi_C(\rho_\eta,\sigma_s)-\gamma)\ge\delta/2$;
  if $\iota=0$, then $\pi(\pos(\eta))>\pi(\pos(s))$, and the opposite inequality follows.
\end{proof}

\begin{figure}[ht]
  \centering
  \begin{tikzpicture}[x=1.35cm,y=1.1cm,
      node/.style={circle,draw=deepblue,fill=palepurple,inner sep=1.6pt},
      leaf/.style={circle,draw=deepblue,fill=deepblue,inner sep=1.6pt},
      every label/.style={font=\scriptsize}]
    \foreach \i in {1,...,7} {\node[font=\scriptsize,text=black!60] at (\i,-0.55) {\i};}
    \draw[black!35] (0.6,-0.25)--(7.4,-0.25);
    \node[node,label=below:{$\sigma_\emptyset$}] (r) at (4,0.4) {};
    \node[node,label=left:{$\sigma_{1}$}] (s1) at (2,1.4) {};
    \node[node,label=right:{$\sigma_{0}$}] (s0) at (6,1.4) {};
    \node[leaf,label=above:{$\rho_{11}$}] (l11) at (1,2.5) {};
    \node[leaf,label=above:{$\rho_{10}$}] (l10) at (3,2.5) {};
    \node[leaf,label=above:{$\rho_{01}$}] (l01) at (5,2.5) {};
    \node[leaf,label=above:{$\rho_{00}$}] (l00) at (7,2.5) {};
    \draw[deepblue,thick] (r)--(s1) node[midway,above left,font=\scriptsize]{$1$};
    \draw[deepblue,thick] (r)--(s0) node[midway,above right,font=\scriptsize]{$0$};
    \draw[deepblue] (s1)--(l11); \draw[deepblue] (s1)--(l10);
    \draw[deepblue] (s0)--(l01); \draw[deepblue] (s0)--(l00);
    \foreach \x in {1,...,7} {\draw[black!25,densely dotted] (\x,-0.25)--(\x,2.5);}
  \end{tikzpicture}
  \caption{The map $\pos$ for $n=2$, followed by the increasing bijection $\pi$ of its image onto $\{1,\dots,7\}$, places the seven sequences of length at most $2$ at the positions $1,\dots,7$ of a ladder of length $7$;
the nodes $\sigma_s$ are drawn as light circles, the leaves $\rho_\eta$ as dark ones, and the tree grows upward.}
  \label{fig:ladder-tree}
\end{figure}
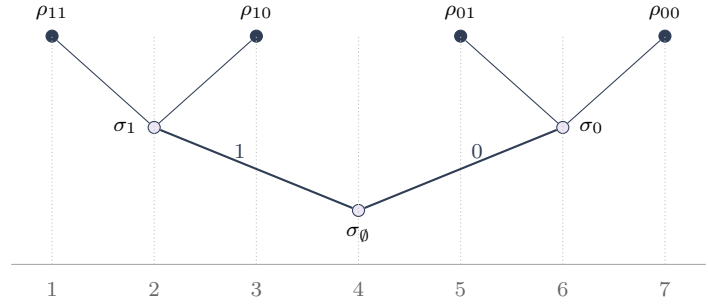

\begin{keytheorem}[Unstable observables encode trees]\label{thm:trees}
  If $C$ is unstable, then there is $\delta>0$ such that for every $n\in\NN$, there is a $\delta$-tree of height $n$ for $C$.
  More precisely, if sequences $(\rho_k)_{k\in\NN}$ in
  $\fD(H_1)$ and $(\sigma_l)_{l\in\NN}$ in $\fD(H_2)$
  are such that the two iterated limits of $\varphi_C(\rho_k,\sigma_l)$
  exist and differ by $\Delta>0$, then for every $\delta$ with $0<\delta<\Delta$ and every $n\in\NN$, there is a $\delta$-tree of height $n$ for $C$ whose leaves and nodes are terms of these sequences.
\end{keytheorem}

\begin{proof}
  Combine \cref{prop:inst-ladder,prop:ladders-to-trees}.
\end{proof}

In particular, the leaves and nodes are pure states when the terms of the two sequences are pure.

For the order projection, built from two orthonormal sequences
$(e_k)_{k\in\NN}$ and $(f_l)_{l\in\NN}$, and for the comparison
$[A<B]$ of two self-adjoint operators $A$ and $B$ with pure point spectra that are unbounded above, the values of the expectation on the vectors
$e_k$ and $f_l$ and on eigenvectors of $A$ and $B$, respectively, form $1$-ladders of every length (see \cref{thm:unstable} and \cref{prop:comparison}(3));
hence, by \cref{prop:ladders-to-trees}, there are $1$-trees of every height whose leaves and nodes are the pure states given by these vectors.
The converse of \cref{thm:trees} fails (see \cref{exa:sector} below):
trees of unbounded height detect a uniform form of instability, which we now describe through ultrapowers.

%#
%### SUBSECTION: Ultrapowers and uniform stability
%#
\subsection{Ultrapowers and uniform stability}\label{sec:ultra}
Uniform stability bounds the finite ladders at every positive margin.

\begin{definition}\label{def:uniform}
  A self-adjoint observable $C$ is \emph{uniformly stable} if for every $\delta>0$, the lengths of $\delta$-ladders for $C$ are bounded.
  An arbitrary bounded operator $C$ is \emph{uniformly stable} if its hermitian parts $\frac12(C+C^*)$ and $\frac1{2i}(C-C^*)$ are uniformly stable.
\end{definition}

To describe uniform stability through ultrapowers, fix a nonprincipal ultrafilter $\cU$ on $\NN$.
A sequence $x=(x_k)_{k\in\NN}$ of states of Alice’s system represents an element of the ultrapower of her preparations;
it should be thought of as a family of experiments indexed by $k$, in which both parties vary their preparations together.
The \emph{ultrapower expectation} of $C$ is the function
\[
  \Phi^\cU_C(x,y)=\ulimit{k}{\cU}\varphi_C(x_k,y_k)
  \qquad\bigl(x\in\fD(H_1)^\NN,\ y\in\fD(H_2)^\NN\bigr).
\]
It is a joint limit, not an iterated one, and so it differs in kind from the two products of idealized states in \cref{sec:setting};
its value at $(x,y)$ is that, at the classes of $x$
and $y$, of the bilinear form induced by $m_C(\tau,\tau')=\Tr\bigl((\tau\otimes\tau')C\bigr)$ on the Banach-space ultrapowers of the spaces $\cT(H_1)$ and $\cT(H_2)$ of trace-class operators~\cite{Heinrich:1980b}.
Ladders and trees for $\Phi^\cU_C$ are defined as in \cref{def:trees}, with sequences in $\fD(H_1)^\NN$ and $\fD(H_2)^\NN$ in place of states.
For $\delta>0$, an \emph{infinite $\delta$-tree} for
$\Phi^\cU_C$ consists of leaves $x_\eta\in\fD(H_1)^\NN$
for every infinite branch $\eta\in\{0,1\}^\NN$, nodes $y_s\in\fD(H_2)^\NN$
for $s\in\{0,1\}^{<\NN}$, a real number $\gamma$ and a sign $\theta$, satisfying the inequalities of \cref{def:trees}(b), with $\Phi^\cU_C$ in place of $\varphi_C$, whenever $\cat{s}{\iota}\seg\eta$.

\begin{keytheorem}[Uniform stability]\label{thm:uniform}
  For every self-adjoint $C\in B(H_1\otimes H_2)$, the following are equivalent:
  \begin{enumerate}[label=\textup{(\arabic*)}]
  \item
    The observable $C$ is uniformly stable.
  \item
    For every $\delta>0$, the heights of $\delta$-trees for $C$ are bounded.
  \item
    The ultrapower expectation $\Phi^\cU_C$ is stable for every nonprincipal ultrafilter $\cU$ on $\NN$.
  \item
    The ultrapower expectation $\Phi^\cU_C$ is stable for some nonprincipal ultrafilter $\cU$ on $\NN$.
  \end{enumerate}
\end{keytheorem}

The proof shows two further facts:
every uniformly stable observable is stable;
and if the conditions of \cref{thm:uniform} fail, then there is $\delta>0$ such that for every nonprincipal ultrafilter $\cU$ on $\NN$, there is an infinite $\delta$-tree for $\Phi^\cU_C$.

\begin{proof}
  The implication from (2) to (1) is \cref{prop:ladders-to-trees}.

  \emph{Proof that (1) implies (3).}
  Suppose that $\Phi^\cU_C$ is unstable.
  The proof of \cref{prop:inst-ladder} uses only \cref{lem:ladder}, which holds for any bounded function on a product of sets, so it provides, for some
  $\delta>0$, an infinite $\delta$-ladder for $\Phi^\cU_C$, consisting of sequences $x_i\in\fD(H_1)^\NN$ and $y_j\in\fD(H_2)^\NN$ for $i,j\in\NN$, a threshold $\gamma$ and a sign $\theta$;
  we write $x_{i,k}$ and $y_{j,k}$ for the $k$-th entries of $x_i$ and $y_j$.
  Fix a positive integer $N$.
  For each of the $N^2$ pairs $1\le i,j\le N$, the set of $k$ with $\theta(\varphi_C(x_{i,k},y_{j,k})-\gamma)\ge\delta/4$ (if $i\le j$), respectively $\le-\delta/4$ (if $i>j$), belongs to $\cU$.
  For $k$ in the intersection of these sets, the states $x_{i,k}$ and $y_{j,k}$ form a $\delta/2$-ladder of length $N$ for $C$.
  Hence the observable $C$ is not uniformly stable.

  The implication from (3) to (4) is trivial.

  \emph{Proof that (4) implies (2).}
  The idea is to glue trees of increasing height into an infinite tree in the ultrapower, and to observe that an infinite tree has $2^{\aleph_0}$ separated leaves, which no separable dual can accommodate.
  Suppose that for some $\delta>0$ there are $\delta$-trees $(\rho^{(n)}_\eta,\sigma^{(n)}_s,\gamma_n,\theta_n)$ of every height $n$, and let $\cU$ be any nonprincipal ultrafilter on $\NN$.
  Pass to a set in $\cU$ on which $\theta_n=\theta$ is constant, and let $\gamma=\ulimit{n}{\cU}\gamma_n$, which exists because, for $n\ge1$, the root of the $n$-th tree has leaves on both sides of $\gamma_n$, whence $|\gamma_n|\le\|C\|$.
  For an infinite branch $\eta\in\{0,1\}^\NN$, let $x_\eta$ be the sequence $(\rho^{(n)}_{\eta|n})_n$, where $\eta|n$ is the initial segment of length $n$, and for $s\in\{0,1\}^{<\NN}$, let $y_s$ be the sequence $(\sigma^{(n)}_s)_n$, with arbitrary entries for $n\le|s|$.
  If $\cat{s}{\iota}\seg\eta$, then the same holds with $\eta|n$ in place of $\eta$ for all $n>|s|$, and passing to the limit, we obtain $\theta(\Phi^\cU_C(x_\eta,y_s)-\gamma)\ge\delta/2$ if $\iota=1$ and $\le-\delta/2$ if $\iota=0$.
  This is an infinite $\delta$-tree for $\Phi^\cU_C$, which proves the assertion on infinite $\delta$-trees made after the theorem.

  Now take for $\cU$ the ultrafilter of~(4), for which $\Phi^\cU_C$ is stable.
  After an affine rescaling to values in $[0,1]$, which multiplies gaps by
  $c=1/(2\|C\|)$, the theorem of Iovino~\cite[Theorem~3.3]{Iovino:1999c} yields a reflexive Banach space
  $E$ and maps $u\colon\fD(H_1)^\NN\to E$ and $v\colon\fD(H_2)^\NN\to E^*$, with values in the unit balls, such that $c(\Phi^\cU_C+\|C\|)=\langle u(\cdot),v(\cdot)\rangle$.
  Let $Y$ be the closed span in $E^*$ of the countably many vectors $v(y_s)$.
  The space $Y$ is separable and reflexive, so its dual $Y^*$ is separable.
  If $\eta\ne\eta'$ first differ after the segment $s$, then
  $|\langle u(x_\eta)-u(x_{\eta'}),v(y_s)\rangle|\ge c\delta$, so the restrictions of the vectors $u(x_\eta)$ to $Y$ are $2^{\aleph_0}$ points of $Y^*$ at mutual distance at least $c\delta$, which is impossible in a separable space.
  This contradiction shows that the heights of $\delta$-trees for $C$ are bounded for every $\delta>0$, which is~(2).

  Finally, an unstable observable has infinite ladders by \cref{prop:inst-ladder}, so it is not uniformly stable.
\end{proof}

The same equivalences hold for families of observables, with one exception.
Conditions~(1), (2) and~(3) of \cref{thm:uniform}, but not condition~(4),
remain equivalent for a bounded family $(C_k)_{k\in\NN}$
of self-adjoint observables $C_k$ on finite-dimensional spaces
$H_1^{(k)}\otimes H_2^{(k)}$, with $\Phi^\cU(x,y)=\ulimit{k}{\cU}\Tr((x_k\otimes y_k)C_k)$
for sequences $x=(x_k)_{k\in\NN}$ and $y=(y_k)_{k\in\NN}$
with $x_k\in\fD(H_1^{(k)})$ and $y_k\in\fD(H_2^{(k)})$, when the ladders and trees of the family are those of its members and their lengths and heights are required to be bounded independently of~$k$.
The finite-intersection argument above still transfers every finite ladder from an ultrapower to a member of the family.
Conversely, suppose that the tree heights are unbounded at some fixed positive margin.
Each finite-dimensional member belongs to its minimal tensor product and therefore has bounded tree heights by \cref{prop:uniform-examples} below.
We may thus choose distinct indices $k_n$, for positive integers $n$, such that $C_{k_n}$ has a tree of height at least $n$, and truncate the tree to height $n$.
Choose a nonprincipal ultrafilter $\cU$ containing $\{k_n:n\ge1\}$.
At coordinate $k_n$, use the leaves and probes of the height-$n$ tree in the construction above, and fill all other coordinates with arbitrary states.
The uniform bound on $\|C_k\|$ bounds the thresholds, so taking their ultralimit and fixing the sign on a set in $\cU$ gives an infinite tree.
Its ultrapower expectation is unstable by the separation argument above.
Condition~(4) alone does not suffice:
if $C_k=0$ for every odd~$k$, then $\Phi^\cU=0$ for every $\cU$ that contains the set of odd numbers, even when the members of even index have unbounded ladder lengths at a fixed positive margin.
Uniform stability of a family, with lengths and heights bounded independently of~$k$, is the relevant notion for registers of growing size (see \cref{sec:outlook}).

\begin{proposition}\label{prop:uniform-examples}
  Every self-adjoint $C$ in the minimal tensor product
  $B(H_1)\otimes_{\min}B(H_2)$ is uniformly stable, and so is every self-adjoint
  $C\in B(H_1\otimes H_2)$ for which there are a Hilbert space $\cH$ and bounded maps $u\colon\fD(H_1)\to\cH$ and $v\colon\fD(H_2)\to\cH$ with $\varphi_C(\rho,\sigma)=\langle u(\rho),v(\sigma)\rangle$ for all $\rho\in\fD(H_1)$ and $\sigma\in\fD(H_2)$.
  In particular, the swap operator $F$, the projections $P_\pm$ and the equality tests $[A=B]$ of \cref{sec:physics} are uniformly stable.
\end{proposition}

\begin{proof}
  For a finite sum $C'=\sum_{j=1}^rA_j\otimes B_j$ of products, with
  $A_j\in B(H_1)$ and $B_j\in B(H_2)$, the function $\Phi^\cU_{C'}(x,y)=\sum_j\ulimit{k}{\cU}\Tr(x_kA_j)\cdot\ulimit{k}{\cU}\Tr(y_kB_j)$ is a separated sum, so it is stable by \cref{lem:preservation}(a).
  Since $|\Phi^\cU_C(x,y)-\Phi^\cU_{C'}(x,y)|\le\|C-C'\|$ for every $x,y$, norm approximation of $C$ by such finite sums and \cref{lem:preservation}(b) show that $\Phi^\cU_C$ is stable.
  In the Hilbert case, we have $\Phi^\cU_C(x,y)=\langle(u(x_k))_\cU,(v(y_k))_\cU\rangle$
  in the ultrapower Hilbert space $\cH^\cU$, where $(u(x_k))_\cU$ denotes the element of $\cH^\cU$ represented by the sequence $(u(x_k))_{k\in\NN}$, and \cref{prop:hilbertfactor} applies.
  Conclude with \cref{thm:uniform}.
\end{proof}

\begin{example}[Stable but not uniformly stable]\label{exa:sector}
  Let both $H_1$ and $H_2$ be the Hilbert space $\bigoplus_{n\ge1}\CC^n$, with orthonormal bases $(e^n_i)_{i\le n}$ of the summands, which we call sectors;
  for instance, for $j\in\{0,\frac12,1,\frac32,\dots\}$, the summand
  $\CC^{2j+1}$ carries the spin-$j$ representation, the irreducible representation of dimension $2j+1$ of the group $\mathrm{SU}(2)$ that implements rotations in quantum mechanics, with basis vectors $|j,m\rangle$, for $m=-j,-j+1,\dots,j$, ordered by $m$.
  Consider the projection
  \[
    C_{\mathrm{sec}}=\sum_{n\ge1}\ \sum_{1\le i\le l\le n}|e^n_i\rangle\langle e^n_i|\otimes|e^n_l\rangle\langle e^n_l| ,
  \]
  onto “both systems lie in the same sector, and Alice’s level does not exceed Bob’s”.
  In the sector $\CC^n$, the states $e^n_1,\dots,e^n_n$ form a $1$-ladder of length $n$, so the projection $C_{\mathrm{sec}}$ is not uniformly stable and has $1$-trees of every finite height.
  The projection $C_{\mathrm{sec}}$ is, however, stable:
  for a density operator $\rho$ on $H_1$, the effective observable $T_{C_{\mathrm{sec}}}(\rho)$ that it leaves on Bob’s side depends only on the diagonal entries $\rho^n_{ii}=\langle e^n_i|\rho|e^n_i\rangle$, and $T_{C_{\mathrm{sec}}}=J\circ a$ with
  \[
    a\colon\cT(H_1)\to\Bigl(\bigoplus_n\ell^1_n\Bigr)_{\ell^2},\quad
    a(\tau)=\bigl((\tau^n_{ii})_{i\le n}\bigr)_n,
    \qquad
    J\bigl((w_n)_n\bigr)=\bigoplus_n\sum_{l\le n}\Bigl(\sum_{i\le l}w_{n,i}\Bigr)|e^n_l\rangle\langle e^n_l| ,
  \]
  where $\ell^1_n$ is $\CC^n$ with the norm $\|w\|_1=\sum_i|w_i|$, and $(\bigoplus_n\ell^1_n)_{\ell^2}$ is the space of sequences $(w_n)_n$ with $w_n=(w_{n,i})_{i\le n}$ in $\ell^1_n$ and $\sum_n\|w_n\|_1^2<\infty$, with the norm $(\sum_n\|w_n\|_1^2)^{1/2}$.
  Here, for $\tau\in\cT(H_1)$ with diagonal entries
  $\tau^n_{ii}=\langle e^n_i|\tau|e^n_i\rangle$, we have
  $\|a(\tau)\|\le\sum_n\sum_i|\tau^n_{ii}|\le\|\tau\|_1$, and $\|J\|\le1$, and the $\ell^2$-sum of finite-dimensional spaces is reflexive, so \cref{thm:weakcompact}(4), the factorization of the slice map through a reflexive space, applies.
\end{example}

The space $(\bigoplus_n\ell^1_n)_{\ell^2}$ is reflexive but not superreflexive.
Here a Banach space is \emph{superreflexive} if every Banach space finitely representable in it, one whose finite-dimensional subspaces are almost isometric to subspaces of it, is reflexive;
equivalently, all its ultrapowers are reflexive~\cite{Heinrich:1980b}.
Superreflexivity is relevant to uniform stability:
if the slice map $T_C$ of an observable $C$ factors through a superreflexive space $E$, then $E^\cU$ is reflexive for every nonprincipal ultrafilter $\cU$~\cite{Heinrich:1980b}.
The ultrapower expectation factors through the pairing of $E^\cU$ with $(E^*)^\cU$, canonically embedded in $(E^\cU)^*$.
It is therefore stable by the argument of \cref{sec:geometry}, so $C$ is uniformly stable by \cref{thm:uniform}.
Whether the converse holds is \cref{ques:superreflexive}.

The following table collects the examples of \cref{sec:physics} and the sector comparison;
the entries in the column on uniform stability are justified by \cref{thm:uniform}, \cref{prop:uniform-examples} and \cref{exa:sector}.
Uniform stability corresponds to what is called superstability in Banach space theory, where a space is superstable when all its ultrapowers are stable~\cite{Raynaud:1983a};
we prefer the name uniform stability because superstability has a different meaning in model theory~\cite{Shelah:1971b}.

\begin{center}
  \renewcommand{\arraystretch}{1.25}
  \begin{tabular}{p{5.0cm}p{1.3cm}p{1.9cm}p{4.4cm}}
    \hline
    Operator or observable                                             & Stable      & Uniformly stable & Mechanism                     \\
    \hline
    $A\otimes B$, $A\otimes I+I\otimes B$                              & yes         & yes              & minimal tensor product        \\
    compact, Bell projections                                          & yes         & yes              & minimal tensor product        \\
    any $C$, one factor finite-dimensional                             & yes         & yes              & finite sum of products        \\
    swap $F$; $P_+$ (swap test); $P_-$ (Hong--Ou--Mandel)                           & yes         & yes              & factors through $S_2(K)$      \\
    $[A=B]$, pure point spectra                                        & yes         & yes              & factors through $\ell^2$      \\
    $[A<B]$, one spectrum finite                                       & yes         & yes              & finite sum of products        \\
    sector comparison $C_{\mathrm{sec}}$ (see \cref{exa:sector} below) & yes         & \textbf{no}      & reflexive, not superreflexive \\
    $[A<B]$, point spectra unbounded above                             & \textbf{no} & \textbf{no}      & encodes the order of $\NN$    \\
    order projection $P$                                         & \textbf{no} & \textbf{no}      & encodes the order of $\NN$    \\
    \hline
  \end{tabular}
\end{center}

%#
%### SUBSECTION: Counting the preparations resolved by adaptive probes
%#
\subsection{Counting the preparations resolved by adaptive probes}\label{sec:growth}
For every positive integer $m$, a complete decision tree of height $m$ realizes all $2^m$ binary transcripts of thresholded expectation responses.
We now count the transcripts when the probe tree need not have every branch realized.
At each fixed positive margin, bounded decision-tree height yields a polynomial bound, while unbounded height yields all $2^m$ transcripts for every positive integer $m$.

Fix $\delta>0$, and for the moment also a threshold $\gamma$ and a sign $\theta$.
For an integer $m\ge0$, a \emph{probe tree of depth $m$} is a family $(\sigma_s)_{s\in\{0,1\}^{<m}}$ of states of Bob’s system.
A branch $\eta\in\{0,1\}^m$ is \emph{realized} by $\rho\in\fD(H_1)$ if, whenever $\cat{s}{\iota}\seg\eta$, the inequalities of \cref{def:trees}(b) hold for $\varphi_C(\rho,\sigma_s)$.
Realized branches are the outcomes that the adaptive protocol can return with margin $\delta$, and two preparations that realize different branches are told apart by $m$ probes.
For a set $\mathcal{R}\subseteq\fD(H_1)$ of preparations, let $d(\mathcal{R})\in\{-1,0,1,\dots,\infty\}$ be the supremum of the heights of the $\delta$-trees with these $\gamma$ and $\theta$ whose leaves lie in $\mathcal{R}$, with $d(\emptyset)=-1$.
The \emph{growth function} of $C$ is
\[
  \gr_C(m,\delta)=\sup\#\{\eta\in\{0,1\}^m:\eta\text{ is realized by some }\rho\in\fD(H_1)\},
\]
the supremum being over $\gamma$, $\theta$ and probe trees of depth $m$, and the \emph{tree dimension} $d_C(\delta)=\sup_{\gamma,\theta}d(\fD(H_1))$ is the supremum of the heights of the $\delta$-trees for $C$.
We write $\binom{m}{\le d}=\sum_{i=0}^d\binom mi$.

\begin{keytheorem}[Polynomial versus exponential growth]\label{thm:growth}
  Let $C\in B(H_1\otimes H_2)$ be self-adjoint, and let $\delta>0$ be given.
  \begin{enumerate}[label=\textup{(\arabic*)}]
  \item
    If the tree dimension $d=d_C(\delta)$ is finite, then for every integer $m\ge0$ we have
    \[
      \gr_C(m,\delta)\le\binom{m}{\le d},
    \]
    a polynomial in $m$ of degree $d$, and $\binom{m}{\le d}\le(\ee m/d)^d$ whenever $m\ge d\ge1$.
  \item
    If the tree dimension $d_C(\delta)$ is infinite, then the growth function satisfies $\gr_C(m,\delta)=2^m$ for every integer $m\ge0$.
  \end{enumerate}
  Consequently, the observable $C$ is uniformly stable if and only if $\gr_C(\cdot,\delta)$ grows polynomially for every $\delta>0$;
  otherwise, for some $\delta>0$, the growth function satisfies $\gr_C(m,\delta)=2^m$ for every integer $m\ge0$.
  In particular, unstable observables have exponential growth.
\end{keytheorem}

\begin{proof}
  Statement~(2) holds because a $\delta$-tree of height at least $m$, restricted to its nodes of length less than $m$ and one leaf above each branch of length $m$, realizes every branch of a probe tree of depth $m$.
  For~(1), the idea is that at each probe, at most one of the two halves of the remaining preparations can carry a tree as tall as the whole set.
  Fix $\gamma$ and $\theta$.
  We show by induction on $m$ that, for every probe tree of depth
  $m$ and every set $\mathcal{R}\subseteq\fD(H_1)$ with $d(\mathcal{R})<\infty$, the number of branches realized by elements of $\mathcal{R}$ is at most $\binom{m}{\le d(\mathcal{R})}$, with the convention that $\binom{m}{\le-1}=0$.
  For $m=0$ this is clear.
  The bound is immediate when $\mathcal{R}=\emptyset$, so in the induction step we may assume that $\mathcal{R}$ is nonempty and $0\le d(\mathcal{R})<\infty$.
  For $m\ge1$, let $\mathcal{R}_\iota$ be the set of elements of $\mathcal{R}$ that satisfy the inequality for $\iota$ at the root.
  The sets $\mathcal{R}_0$ and $\mathcal{R}_1$ are disjoint, because the two inequalities are separated by $\delta$.
  A branch $\cat{\iota}{\eta'}$ is realized in $\mathcal{R}$ exactly when $\eta'$ is realized in $\mathcal{R}_\iota$ for the subtree $(\sigma_{\cat{\iota}{s}})_s$ of depth $m-1$.
  We have $d(\mathcal{R}_\iota)\le d(\mathcal{R})$, and the two numbers $d(\mathcal{R}_0)$ and $d(\mathcal{R}_1)$ cannot both equal $d(\mathcal{R})$:
  two $\delta$-trees of height $d(\mathcal{R})$ with leaves in $\mathcal{R}_0$ and in $\mathcal{R}_1$, placed above the root probe $\sigma_\emptyset$, would form a $\delta$-tree of height $d(\mathcal{R})+1$ with leaves in $\mathcal{R}$.
  By induction and Pascal’s rule, the number of realized branches is at most $\binom{m-1}{\le d}+\binom{m-1}{\le d-1}=\binom{m}{\le d}$, with $d=d(\mathcal{R})$.
  The bound $\binom{m}{\le d}\le(\ee m/d)^d$ is standard.
  The last statements follow from \cref{thm:trees,thm:uniform}.
\end{proof}

\begin{remark}[Rowbottom’s definition of stability]\label{rem:rowbottom}
  Stability has many equivalent characterizations, but its original definition was a growth condition.
  Rowbottom~\cite{Rowbottom:1964} called a complete theory
  $T$ \emph{stable in} an infinite cardinal $\lambda$ if for every model $\cM$ of $T$ and every set $A\subseteq\cM$ of cardinality $\lambda$, the space $S_1(A)$ of complete $1$-types over $A$ has cardinality at most $\lambda$.
  The reverse inequality is trivial, since distinct elements of $A$ realize distinct types over $A$;
  hence stability in $\lambda$ says that $|S_1(A)|=|A|$, that is, the number of types grows linearly with the number of parameters, whereas a priori it can be as large as $2^\lambda$.
  For a single formula $\varphi(x,y)$, stability in
  $T$ is equivalent to the bound $|S_\varphi(A)|\le|A|$
  for every infinite set $A$ of parameters from a model of
  $T$, where $S_\varphi(A)$ is the space of complete $\varphi$-types over $A$, the complete descriptions of which instances $\varphi(x,a)$, with $a$ from $A$, an element satisfies (see \cite[Chapter~1]{Pillay:1996} and~\cite{Shelah:1990a}).
  The polynomial bound of \cref{thm:growth} is a finite and adaptive counterpart of this condition:
  for a uniformly stable observable, the number of preparations resolved by $m$ probes at a fixed margin is bounded by a polynomial in $m$.
  For an infinite cardinal $\lambda$ and an integer $d\ge1$, we have $\lambda^d=\lambda$, so on infinite parameter sets the distinction between linear and polynomial bounds disappears.
\end{remark}

In \cref{thm:perceptron} below, we show that for observables that factor through a Hilbert space, the degree of the polynomial is controlled explicitly, by a linear learner.

\begin{keytheorem}[Linear bound for Hilbert-factorable observables]\label{thm:perceptron}
  Let $C\in B(H_1\otimes H_2)$ be self-adjoint, and suppose that there are a Hilbert space
  $\cH$, positive real numbers $R_u$ and $R_v$, and maps
  $u\colon\fD(H_1)\to\cH$ and $v\colon\fD(H_2)\to\cH$ such that $\|u(\rho)\|\le R_u$, $\|v(\sigma)\|\le R_v$ and $\varphi_C(\rho,\sigma)=\langle u(\rho),v(\sigma)\rangle$ for all $\rho\in\fD(H_1)$ and $\sigma\in\fD(H_2)$.
  Then for every $\delta>0$, we have
  \[
    d_C(\delta)\le16\Bigl(\frac{R_uR_v}{\delta}\Bigr)^{2},
  \]
  independently of the dimensions of $H_1$ and $H_2$.
  In particular, the tree dimension is at most $16/\delta^2$ for the swap operator $F$ and for the equality tests $[A=B]$ of \cref{sec:physics}, and at most $4/\delta^2$ for the swap test $P_+$ and the Hong--Ou--Mandel projection $P_-$.
\end{keytheorem}

\begin{proof}
  The perceptron updates a linear classifier by adding each misclassified example multiplied by its label $\pm1$.
  We run it along the tree, always moving to the child that contradicts its prediction.
  The leaf reached at the end is a linear separator with margin $\delta/2$, against which the perceptron makes few mistakes.
  Regard the Hilbert space as a real inner product space with $\langle\cdot,\cdot\rangle_{\RR}=\operatorname{Re}\langle\cdot,\cdot\rangle$;
  since $\varphi_C$ is real, we have $\varphi_C=\langle u,v\rangle_{\RR}$.
  Let a $\delta$-tree of height $m$ be given, with threshold $\gamma$ and sign $\theta$.
  We may assume that $m\ge1$;
  the tree then requires values of $\varphi_C$ on both sides of $\gamma$, which forces $|\gamma|\le R_uR_v$.
  Define $\hat{u}(\rho)=(u(\rho),R_u)$ and $\hat{v}(\sigma)=(v(\sigma),-\gamma/R_u)$ in $\cH\oplus\RR$;
  then $\langle\hat{u}(\rho),\hat{v}(\sigma)\rangle_{\RR}=\varphi_C(\rho,\sigma)-\gamma$, $\|\hat{u}\|\le\sqrt2R_u$ and $\|\hat{v}\|\le\sqrt2R_v$.
  Start with $z_0=0$ at the root $s_0=\emptyset$.
  At the node $s_i$ reached at step $i$, for $0\le i<m$, let $\hat{y}_i=\operatorname{sign}\langle z_i,\hat{v}(\sigma_{s_i})\rangle_{\RR}$, with
  $\operatorname{sign}0=1$, move to the child $s_{i+1}=\cat{s_i}{\iota}$ whose label $y_i=\theta(2\iota-1)$ equals $-\hat{y}_i$, and let $z_{i+1}=z_i+y_i\,\hat{v}(\sigma_{s_i})$.
  After $m$ steps we reach a branch $\eta$, and its leaf $\hat{u}_\eta=\hat{u}(\rho_\eta)$ satisfies $y_i\langle\hat{u}_\eta,\hat{v}(\sigma_{s_i})\rangle_{\RR}\ge\delta/2$ for every $i<m$.
  Summing over the steps, we obtain $\langle\hat{u}_\eta,z_m\rangle_{\RR}\ge m\delta/2$, while $\|z_m\|^2\le\sum_i\|\hat{v}(\sigma_{s_i})\|^2\le2mR_v^2$, because each step adds $2y_i\langle z_i,\hat{v}(\sigma_{s_i})\rangle_{\RR}\le0$ to the cross term.
  By the Cauchy--Schwarz inequality, we have $m\delta/2\le\sqrt2R_u\cdot\sqrt{2m}\,R_v$, that is, $m\le16(R_uR_v/\delta)^2$.

  For $F$ and $[A=B]$, both radii $R_u$ and $R_v$ are at most $1$ (see \cref{sec:swap} and \cref{prop:comparison}).
  Since $\varphi_{P_\pm}=\frac12(1\pm\varphi_F)$, every $\delta$-tree for $P_\pm$ is a $2\delta$-tree for $F$.
\end{proof}

\begin{remark}[Node-dependent thresholds]\label{rem:sfat}
  The proof of \cref{thm:perceptron} uses the threshold only through the probe vectors $\hat{v}(\sigma_s)$, and the bound $|\gamma|\le R_uR_v$ holds at each node separately, since every node has leaves on both sides of its threshold.
  The proof therefore applies without change to trees in which the threshold $\gamma_s$ depends on the node $s$.
  Consequently, for a Hilbert-factorable observable, the response class
  $\{\sigma\mapsto\varphi_C(\rho,\sigma):\rho\in\fD(H_1)\}$
  has sequential fat-shattering dimension at most $16(R_uR_v/\delta)^2$
  at scale $\delta$ in the sense of Rakhlin, Sridharan and Tewari~\cite{Rakhlin-Sridharan-Tewari:2015}, that is,
  $\delta$-trees for $C$ with node-dependent thresholds have height at most this number.
  The class is therefore learnable in the supervised online setting with absolute loss, the absolute prediction error, and sublinear regret, with bounds independent of the dimension.
  Here sublinear regret means that the expected excess cumulative loss over the best fixed response in hindsight, divided by the number of rounds, has a worst-case upper bound tending to zero.
\end{remark}

%#
%### SUBSECTION: Affine approximation versus exponential growth
%#
\subsection{Affine approximation versus exponential growth}\label{sec:linear-exp}
Affine approximation, the subject of \cref{sec:simulation}, and the growth of adaptive discrimination are different phenomena, and the second is governed by the stronger classification.
The following table sets the statements proved here side by side.
\begin{center}
  \renewcommand{\arraystretch}{1.3}
  \begin{tabular}{p{3.6cm}p{5.1cm}p{4.6cm}}
    \hline
    \textbf{Mathematical question} & \textbf{Hypothesis and conclusion} & \textbf{Scope or obstruction} \\
    \hline
    approximation of idealized responses & stability: uniform limits of \emph{convex} combinations of actual responses, that is, responses of finite mixtures, with Bob idealized first (see \cref{thm:polyapprox}) & for an unstable observable, some Bob-first response has positive distance from all actual responses; for the order projection this uniform distance is $1$, even on normal probes (see \cref{exa:nomixture}) \\
    transcripts realized by $m$ adaptive expectation queries at margin $\delta$ & uniform stability: at most $\sum_{i=0}^{d_C(\delta)}\binom mi=O(m^{d_C(\delta)})$ transcripts at every fixed margin (see \cref{thm:growth}) & failure of uniform stability: $2^m$ transcripts at some fixed margin for every $m$; qualitative stability alone does not exclude this behavior \\
    quantitative bounds from Hilbert-space factorization & for factors $u$ and $v$ on states bounded by $1$: tree dimension at most $16/\delta^2$, and $4/\delta^2$ for the swap test and Hong--Ou--Mandel coincidences (see \cref{thm:perceptron}) & for a linear and a conjugate-linear factor of the bilinear form, both bounded by $1$: mixtures of $\lceil4/\varepsilon^2\rceil$ terms suffice at accuracy $\varepsilon$, uniformly over all probes and independently of dimension (see \cref{thm:montecarlo}) \\
    \hline
  \end{tabular}
\end{center}
The approximating combinations in the first row are convex, with coefficients that are probabilities;
expectations preserve these combinations because they are affine in the state (see \cref{rem:degree-one}).
The polynomial upper bound in the second row has exponent $d_C(\delta)$, which need not be one:
for a positive integer $n$, the comparison $\mathrm{GT}_n$ of \cref{sec:outlook}, the projection onto the event that, on two registers of $n$ qubits read as binary numbers, Alice’s number exceeds Bob’s, has tree dimension $d_{\mathrm{GT}_n}(1)\ge n-1$.
In the Hilbert-factorable case, linearity enters twice more:
idealized responses are inner products with a fixed vector, and tree heights are bounded by a linear classifier.
The factor maps in the tree bound need only be bounded on states;
the mixture bound assumes the stated linear and conjugate-linear factors on trace-class operators.
Mixture terms may repeat, and their number bounds the size of the support;
neither this number nor the query count is a bound on measurement shots.
The first two rows concern different notions.
The sector comparison of \cref{exa:sector} is stable and therefore admits convex simulation of idealized responses;
it is not uniformly stable, however, and its growth at margin $1$ is $2^m$.

%#
%### SUBSECTION: Learning: Littlestone’s mistake bound
%#
\subsection{Learning: Littlestone’s mistake bound}\label{sec:learn}
The $\delta$-trees of \cref{def:trees} also govern mistake bounds in online learning~\cite{Littlestone:1988}.
The predictions concern thresholded expectations;
estimating these expectations from repeated preparations is a separate statistical task, whose measurement cost is not counted by the mistake bound.
Fix a threshold $\gamma$ and a margin $\delta>0$, and consider the following game.
Alice prepares an unknown state $\rho$.
In each round, Bob chooses a probe $\sigma$ for which $|\varphi_C(\rho,\sigma)-\gamma|\ge\delta/2$;
a learner predicts whether $\varphi_C(\rho,\sigma)$ lies above or below $\gamma$, and is then told the answer.
Let $\mist_C(\gamma,\delta)$ be the least integer
$m$ such that some deterministic learner makes at most $m$ mistakes, for every $\rho\in\fD(H_1)$ and every sequence of probes, or $\infty$ if there is no such integer, and let $\mist_C(\delta)=\sup_\gamma\mist_C(\gamma,\delta)$.
Let $d_C(\gamma,\delta)$ be the supremum of the heights of the $\delta$-trees for $C$ with threshold $\gamma$ and sign $\theta=1$.
Since reflecting a tree, that is, exchanging the two children at every node, exchanges $\theta=1$ and $\theta=-1$, we have $\sup_\gamma d_C(\gamma,\delta)=d_C(\delta)$.

\begin{proposition}[Mistake bound equals tree dimension]\label{prop:littlestone}
  For every self-adjoint $C\in B(H_1\otimes H_2)$, every real number
  $\gamma$ and every $\delta>0$, the mistake bound $\mist_C(\gamma,\delta)$ equals $d_C(\gamma,\delta)$.
  Consequently, the observable $C$ is uniformly stable if and only if $\mist_C(\delta)$ is finite for every $\delta>0$, and then $\mist_C(\delta)=d_C(\delta)$.
\end{proposition}

\begin{proof}
  If there is a tree of height $n$, an adversary walks down it, choosing at each node the child that contradicts the learner’s prediction;
  the leaf reached is a preparation consistent with all answers, so every learner makes $n$ mistakes.
  If $d_C(\gamma,\delta)=\infty$, these lower bounds already give $\mist_C(\gamma,\delta)=\infty$.
  For the converse bound, assume that $d_C(\gamma,\delta)<\infty$ and run Littlestone’s standard optimal algorithm:
  keep the set $V$ of preparations consistent with the answers so far;
  for a probe $\sigma$, let $V^1$ and $V^0$ be the sets of elements of $V$ above $\gamma+\delta/2$ and below $\gamma-\delta/2$, and predict the side $\iota$ for which $d(V^\iota)$ is larger, in the notation of \cref{sec:growth}.
  As in the proof of \cref{thm:growth}, the numbers $d(V^0)$ and $d(V^1)$ cannot both equal $d(V)$, so after a mistake the new consistent set $V^{1-\iota}$ satisfies $d(V^{1-\iota})<d(V)$.
  Since the true preparation stays in $V$, we have $d(V)\ge0$ throughout, and there are at most $d_C(\gamma,\delta)$ mistakes.
  The second statement follows from \cref{thm:uniform}.
\end{proof}

By \cref{thm:trees,prop:littlestone}, for an unstable observable the mistake bound in the threshold game is infinite at some margin, and the sector comparison of \cref{exa:sector} shows that stability alone does not guarantee a finite mistake bound.
The threshold game is a restricted form of online learning:
general online learnability of real-valued classes allows thresholds that vary from probe to probe (see the historical remarks below and \cref{ques:sfat}).
In the Hilbert-factorable case the restriction is immaterial:
by \cref{rem:sfat}, the perceptron of \cref{thm:perceptron} learns online, even when the thresholds vary and in every dimension, with at most
$16(R_uR_v/\delta)^2$ mistakes, where $R_u$ and $R_v$ bound the norms of the two factors $u$ and $v$ of the expectation.

%#
%### SUBSECTION: Historical remarks
%#
\subsection{Historical remarks}\label{sec:hist-trees}
Stability in a cardinal $\lambda$, which extends Morley’s
$\omega$-stability (stability in $\aleph_0$)~\cite{Morley:1965} to uncountable languages,
was introduced by Rowbottom~\cite{Rowbottom:1964} in his work on the Łoś conjecture for uncountable theories,
the conjecture that a complete theory with exactly one model, up to isomorphism, in some cardinal larger than the cardinality of its language has exactly one in every such cardinal.
Shelah studied the stability spectrum, the class of cardinals in which a theory is stable, and related it to the order property~\cite{Shelah:1971b,Shelah:1990a}.
For Banach space structures, the stability spectrum theorem was proved by Iovino in his dissertation~\cite{Iovino:1994a}.
Binary trees of formulas and the associated rank are due to Shelah~\cite{Shelah:1971b,Shelah:1990a},
and the equivalence between bounded ladders, bounded trees and stability of a formula in all models of a theory, which underlies \cref{thm:uniform}, is part of classical stability theory~\cite{Shelah:1990a}.
For Banach spaces, the stability of all ultrapowers is Raynaud’s superstability~\cite{Raynaud:1983a}, the analytic counterpart of condition~(3) of \cref{thm:uniform};
the passage from ladders to trees in \cref{prop:ladders-to-trees} adapts Proposition~2.3.1 of Iovino’s notes on stability theory~\cite{Iovino:0000e}, where it is proved for first-order formulas, with the gap
$\delta$ given by the failure of the double limit property in place of the negation of a formula;
in \cref{thm:uniform}, ultrapowers take the place of compactness.
Quantitative bounds in the converse direction, from trees back to ladders, come from a Ramsey argument for trees due to Hodges~\cite{Hodges:1993};
Conant and Terry give explicit bounds for real-valued functions~\cite{Conant-Terry:2026}.
The counting bound of \cref{thm:growth} is the tree analogue of the Sauer--Shelah lemma~\cite{Sauer:1972,Shelah:1972b};
versions for Littlestone classes appear in the work of Ben-David, Pál and Shalev-Shwartz~\cite{BenDavid-Pal-ShalevShwartz:2009} and of Rakhlin, Sridharan and Tewari~\cite{Rakhlin-Sridharan-Tewari:2015}.
Littlestone introduced the dimension and the optimal algorithm that bear his name~\cite{Littlestone:1988};
for bounded real-valued classes, supervised online learnability with absolute loss and sublinear regret is characterized by finiteness of the sequential fat-shattering dimension at every scale~\cite{Rakhlin-Sridharan-Tewari:2015},
which allows node-dependent thresholds (see \cref{ques:sfat}).
The link between stability and Littlestone dimension in model theory is due to Chase and Freitag~\cite{Chase-Freitag:2019}.
The bound in \cref{thm:perceptron} is Novikoff’s mistake bound for the perceptron~\cite{Novikoff:1962,Block:1962}, read as a bound on tree height.
For $n$-qubit states probed by two-outcome measurements, Aaronson, Chen, Hazan, Kale and Nayak showed that the sequential fat-shattering dimension is $O(n)$ at each fixed scale~\cite{Aaronson-Chen-Hazan-Kale-Nayak:2018}.

%#
%## SECTION: Order: stable observables are closed under norm limits and orthogonal sums
%#
\section{Order: stable observables are closed under norm limits and orthogonal sums}\label{sec:lattice}

\noindent We now determine how stability behaves under operations on bounded operators and on observables.
We show that stable bounded operators form a norm-closed operator system, and that the same is true of uniformly stable bounded operators.
The self-adjoint members of these systems are the stable and uniformly stable observables, respectively.
Stable projections are closed under orthocomplements and orthogonal sums, but two uniformly stable projections can have an unstable meet.
Stability need not survive strong operator limits either.
These examples distinguish linear and norm closure from closure under the lattice operations of quantum logic.

%#
%### SUBSECTION: Topological and algebraic closure
%#
\subsection{Topological and algebraic closure}
A \emph{norm-closed operator system} in $B(H_1\otimes H_2)$ is a self-adjoint linear subspace that contains $I$ and is closed in the norm topology.

\begin{proposition}[Topological and algebraic closure]\label{prop:operatorsystem}
  The set $\cO$ of stable bounded operators is a norm-closed operator system in $B(H_1\otimes H_2)$.
  It contains the minimal tensor product $B(H_1)\otimes_{\min}B(H_2)$ and every bounded operator whose expectation factors through a Hilbert space.
  The same holds for the set of uniformly stable bounded operators.
\end{proposition}

\begin{proof}
  The identity $I=I\otimes I$ is stable, and for density operators we have $\varphi_{C^*}=\overline{\varphi_C}$.
  Linear combinations and norm limits are handled by \cref{lem:preservation}(a) and~(b), and for uniform stability the same lemma applies to the ultrapower expectation $\Phi^\cU_C$ of \cref{sec:ultra}.
  The last two statements are \cref{thm:stable}, \cref{prop:hilbertfactor} and \cref{prop:uniform-examples}.
\end{proof}

Ordered by inclusion of ranges, with $I-P$ as the orthocomplement of $P$, the projections on $H_1\otimes H_2$ form an \emph{orthomodular poset}:
any two mutually orthogonal projections $P$ and $Q$, that is, projections with $PQ=0$, have a join, namely $P+Q$, and whenever $P\le Q$, the projection $Q$ is the join of $P$ and a projection orthogonal to $P$, namely $Q-P$;
this last property is the \emph{orthomodular law}.

\begin{corollary}[Orthocomplements and orthogonal sums]\label{cor:orthoposet}
  The stable projections form an orthomodular poset:
  if a projection $P$ on $H_1\otimes H_2$ is stable, then so is $I-P$, and if two stable projections $P$ and $Q$ on $H_1\otimes H_2$ satisfy $PQ=0$, then the projection $P+Q$ is stable.
\end{corollary}

\begin{proof}
  If the projections $P$ and $Q$ are stable, then the projections $I-P$, $P+Q$ (when $PQ=0$) and $Q-P$ (when $P\le Q$) are linear combinations of stable observables, and hence they are stable by \cref{prop:operatorsystem}.
  Thus the stable projections are closed under orthocomplements and orthogonal sums, and the orthomodular law, valid for all projections on $H_1\otimes H_2$, holds among them.
\end{proof}

%#
%### SUBSECTION: The boundary of closure: meets and strong limits
%#
\subsection{The boundary of closure: meets and strong limits}
Closure first fails at meets.
The meet $P_1\wedge P_2$ of two projections, the quantum-logical “and”, is the projection onto the intersection of their ranges;
in \cref{prop:meet} below, we construct two uniformly stable projections whose meet is unstable.

\begin{proposition}[The boundary: meets]\label{prop:meet}
  If $H_1$ and $H_2$ are infinite-dimensional, then there are projections $P_1$ and $P_2$ in $B(H_1)\otimes_{\min}B(H_2)$, both uniformly stable, whose meet $P_1\wedge P_2$ is unstable.
\end{proposition}

\begin{proof}
  The idea is to hide the order projection inside the intersection of two subspaces that are compact perturbations of a product projection, by tilting each matrix entry into a fresh direction through a small angle, in opposite senses for the two subspaces.
  Let $(e_k)_{k\ge1}$ and $(e'_k)_{k\ge1}$ be sequences in
  $H_1$ whose terms together form an orthonormal set, let
  $(f_l)_{l\ge1}$ and $(f'_l)_{l\ge1}$ be such sequences in $H_2$, and let $P$ be the order projection built from $(e_k)$ and $(f_l)$, the projection onto the closed span of the vectors $e_k\otimes f_l$ with $k<l$.
  For $k\ge l$, let $\zeta_k=1/k$ and
  \[
    n_{kl}=\cos\zeta_k\,e_k\otimes f_l+\sin\zeta_k\,e'_k\otimes f'_l,\qquad
    n'_{kl}=\cos\zeta_k\,e_k\otimes f_l-\sin\zeta_k\,e'_k\otimes f'_l,
  \]
  and let $P_1$ and $P_2$ be the projections onto the closed subspaces $\operatorname{ran}P\oplus\overline{\operatorname{span}}\{n_{kl}:k\ge l\}$ and $\operatorname{ran}P\oplus\overline{\operatorname{span}}\{n'_{kl}:k\ge l\}$.
  The planes spanned by $e_k\otimes f_l$ and $e'_k\otimes f'_l$ are mutually orthogonal and orthogonal to the range of $P$, and in each of them the lines through $n_{kl}$ and $n'_{kl}$ meet only in $0$, since $0<\zeta_k<\pi/2$.
  Both subspaces are block diagonal for this decomposition, so their intersection is the range of $P$, that is, $P_1\wedge P_2=P$, which is unstable by \cref{thm:unstable}.

  Let $Q_1$ and $Q_2$ be the projections onto the closed spans of $(e_k)$ and $(f_l)$.
  The operator $P_1-Q_1\otimes Q_2$ is the orthogonal sum over $k\ge l$ of the rank-two operators $|n_{kl}\rangle\langle n_{kl}|-|e_k\otimes f_l\rangle\langle e_k\otimes f_l|$, of norm $\sin\zeta_k$.
  For every $\varepsilon>0$, only finitely many pairs $k\ge l$ have $1/k\ge\varepsilon$, so this sum is compact.
  Since $Q_1\otimes Q_2$ and the compact operators lie in $B(H_1)\otimes_{\min}B(H_2)$, so does $P_1$, and likewise $P_2$.
  Conclude with \cref{thm:stable} and \cref{prop:uniform-examples}.
\end{proof}

By von Neumann’s alternating projection theorem, the meet is the strong limit $P_1\wedge P_2=\lim_{m\to\infty}(P_1P_2)^m$~\cite{Deutsch:2001}.
The convergence is in norm if and only if the sum of the ranges of $P_1$ and $P_2$ is closed~\cite{Deutsch:2001}.
In \cref{prop:meet}, the angles $\zeta_k$ accumulate at $0$, so the sum of the ranges is not closed.
Since $B(H_1)\otimes_{\min}B(H_2)$ is an algebra, the operators $(P_1P_2)^m$, for positive integers $m$, are stable, and their strong limit $P_1\wedge P_2$ is unstable.
Thus stability passes to norm limits (see \cref{lem:preservation}(b)), but not to strong limits.

%#
%## SECTION: Outlook: registers of growing size
%#
\section{Outlook: registers of growing size}\label{sec:outlook}

\noindent We indicate, finally, how the uniform notions of \cref{sec:trees} apply to quantum computation.
There the natural setting is a family of registers of growing size,
each with a readout observable on two copies of the register.
Deep computations, in the sense of~\cite{Alva-Duenez-Iovino-Walton:2024,Duenez-Iovino-MatosWiederhold-Salvetti-Tall:2026}, are idealized as ultralimits of the outputs of longer and longer compositions of quantum channels,
the linear maps that send density operators to density operators even when they act on one part of a larger system.

To make this precise, for every positive integer $n$, let $H^{(n)}=(\CC^2)^{\otimes n}$ be the register of $n$ qubits, or two-level systems, and let $C_n$ be a readout observable on $H^{(n)}\otimes H^{(n)}$;
the outputs of deep computations are the ultralimit states
$\ulimit{k}{\cU}\mathcal{E}_k\circ\dots\circ\mathcal{E}_1(\rho_0)$, that is, the states
$A\mapsto\ulimit{k}{\cU}\Tr\bigl(\mathcal{E}_k\circ\dots\circ\mathcal{E}_1(\rho_0)\,A\bigr)$, along a nonprincipal ultrafilter
$\cU$ on $\NN$, for an initial density operator $\rho_0$ on a register and a sequence $(\mathcal{E}_k)_{k\in\NN}$ of quantum channels on that register.
On a single register, which is finite-dimensional, such a state is again given by a density operator.
Each $C_n$ acts on a finite-dimensional space and is thus stable (see \cref{cor:compact-and-findim});
the meaningful condition is the uniform stability of the family, in the sense of the paragraph after \cref{thm:uniform}.

Two families illustrate the distinction.
The equality tests $\mathrm{EQ}_n=\sum_x|x\rangle\langle x|\otimes|x\rangle\langle x|$, where
$x=x_1\cdots x_n$ runs over the strings in $\{0,1\}^n$
and $|x\rangle=|x_1\rangle\otimes\dots\otimes|x_n\rangle$
is the vector that $x$ labels in the computational basis of
$H^{(n)}$, the product basis built from a fixed orthonormal basis $|0\rangle,|1\rangle$ of $\CC^2$, factor through $\ell^2$ with constants independent of $n$ (see the proof of \cref{prop:comparison}).
The family is therefore uniformly stable:
its tree dimension at every margin $\delta>0$ is at most
$16/\delta^2$ for all $n$ (see \cref{thm:perceptron}),
so that its growth is polynomial of a degree independent of $n$, and the Monte Carlo bound of \cref{cor:montecarlo} holds, at every accuracy $\varepsilon>0$, with mixtures of $M=\lceil4/\varepsilon^2\rceil$ preparations for all $n$.
The comparisons $\mathrm{GT}_n=\sum_{x>y}|x\rangle\langle x|\otimes|y\rangle\langle y|$, in which the strings
$x$ and $y$ are compared as binary numbers, contain $1$-ladders of length $2^n-1$, hence $1$-trees of height $n-1$ by \cref{prop:ladders-to-trees}:
the family is not uniformly stable, and $n-1$ adaptive probes resolve $2^{n-1}$ basis preparations.
The same dichotomy appears in communication complexity, the study of how many bits two parties holding separate inputs must exchange to compute a function of both:
equality has constant public-coin randomized cost, that is, parties who share random bits and may err with small probability can compute it by exchanging a number of bits independent of
$n$, while Greater-Than on $n$-bit inputs has cost $\Theta(\log n)$~\cite{Kushilevitz-Nisan:1997,Viola:2015}.
In the language of graphs, the communication matrix of Greater-Than, the table of its values on all pairs of inputs, is the biadjacency matrix of a \emph{half-graph}:
its entries form a staircase of $0$s and $1$s, like the values of a ladder in the sense of \cref{def:trees}.
A family of communication problems excludes large Greater-Than instances exactly when its matrices contain no large half-graphs.
Harms, Wild and Zamaraev showed that this exclusion, which is stability, is necessary for constant randomized cost~\cite{Harms-Wild-Zamaraev:2022}, and Harms and Zamaraev proved it sufficient for matrices of sign-rank at most
$3$, those that record the signs of the entries of a real matrix of rank at most $3$, and for unit disk graphs, the intersection graphs of disks of equal radius in the plane~\cite{Harms-Zamaraev:2023}.

%#
%## SECTION: Questions
%#
\section{Questions}\label{sec:questions}

\noindent We close with questions raised by the results above, on the
structure of the operator system of stable bounded operators, on the gap between stability and uniform stability, and on quantitative bounds:
\begin{enumerate}[label=(Q\arabic*)]
\item
  Is every stable observable a norm limit of observables whose slice maps, which send each preparation of Alice’s system to the effective observable that it leaves on Bob’s side, factor through a Hilbert space?
\item
\label{ques:superreflexive} Is a self-adjoint observable $C$ uniformly stable if and only if its slice map $T_C$ factors through a superreflexive Banach space?
  By the discussion after \cref{exa:sector}, factorization through a superreflexive space implies uniform stability.
\item
  Is the operator system $\cO$ of stable bounded operators closed under products, and hence a C$^*$-algebra?
  A positive answer would make the meet of two stable projections whose ranges have a closed sum stable, by the alternating projection theorem of \cref{sec:lattice}.
\item
  Is the meet of two \emph{commuting} stable projections stable?
\item
\label{ques:rates} For a uniformly stable family of readouts on $n$-qubit registers, what is the minimal size of a mixture simulating idealized preparations to precision $\varepsilon$, and when is it independent of $n$?
  For Hilbert-factorable families, \cref{thm:montecarlo} gives a bound independent of $n$.
\item
\label{ques:sfat} The trees of \cref{thm:growth} and \cref{prop:littlestone} have a constant threshold.
  Does uniform stability of a self-adjoint observable $C$ imply that its response class $\{\sigma\mapsto\varphi_C(\rho,\sigma):\rho\in\fD(H_1)\}$ has finite sequential fat-shattering dimension at every scale, with node-dependent thresholds?
  For Hilbert-factorable observables, \cref{rem:sfat} gives a positive answer.
\item
  Is the tree dimension $d_C(\delta)$ of a uniformly stable self-adjoint observable
  $C$, the supremum of the heights of its $\delta$-trees, always bounded by a polynomial in $1/\delta$, as in the Hilbert case, and is the degree of that polynomial related to the Rademacher type of a space through which the slice map $T_C$ factors?
\end{enumerate}

\bibliographystyle{amsalpha}
\bibliography{references}

\end{document}